\documentclass[reqno]{amsart}

\usepackage[utf8]{inputenc}
\usepackage[T1]{fontenc}
\usepackage{mathtools}
\usepackage{amssymb}
\usepackage{amsthm,thmtools}
\usepackage{amsfonts}
\usepackage{graphicx}
\usepackage{enumerate}
\usepackage{enumitem}
\usepackage{tikz}
\usepackage{tikz-cd}
\usepackage{commath}
\usepackage{booktabs}
\usepackage[mathscr]{euscript}
\usepackage{stmaryrd}
\usepackage{hyperref}
\usepackage{parskip}
\usepackage{epigraph}
\usepackage[normalem]{ulem}         
\usepackage{amsrefs}
\usepackage[scr=boondox,  % heavily sloped
            cal=esstix]   % slightly sloped
           {mathalpha}
\usepackage{bbm}
\usepackage{dsfont}
\usepackage[all]{xy}
\usepackage[T1]{fontenc}
\usepackage{pgfplots}
\pgfplotsset{compat=1.15}
\usepackage{mathrsfs}
\usepackage{comment}

\title{Operator Algebras over the $p$-adic integers}

\DeclareMathOperator*\alg{alg}
\DeclareMathOperator*\supp{supp}

\DeclareMathOperator*\Aut{Aut}
\DeclareMathOperator*\Homeo{Homeo}
\DeclareMathOperator*\id{id}

\DeclareMathOperator*\Bis{Bis}

\theoremstyle{definition}
\newtheorem{theorem}{Theorem}[section]
\newtheorem{definition}[theorem]{Definition}
\newtheorem{example}[theorem]{Example}
\newtheorem{lemma}[theorem]{Lemma}
\newtheorem{proposition}[theorem]{Proposition}
\newtheorem{corollary}[theorem]{Corollary}

\newtheorem{remark}[theorem]{Remark}
\newtheorem*{theorem*}{Theorem}

\newcommand*{\nb}{\nobreakdash}

\newcommand{\defeq}{\mathrel{:=}} % per Definition

\newcommand{\an}{\mathrm{an}}

\newcommand{\Cc}{\mathfrak{C}_c}
\newcommand{\Cz}{\mathfrak{C}_0}

\newcommand{\zZ}{\mathbb{Z}}
\newcommand{\nN}{\mathbb{N}}
\newcommand{\rR}{\mathbb{R}}
\newcommand{\qQ}{\mathbb{Q}}

\newcommand{\fF}{\mathbb{F}}
\newcommand{\Fp}{\mathbb{F}_p}

\newcommand*{\R}{\mathbb R}
\newcommand*{\T}{\mathbb T}
\newcommand*{\Z}{\mathbb Z}
\newcommand*{\N}{\mathbb N}
\newcommand*{\C}{\mathbb C}
\newcommand*{\Q}{\mathbb Q}
\newcommand{\coma}{\widehat}
\newcommand{\haptimes}{\mathbin{\coma{\otimes}}}% adically complete tp

\newcommand*{\into}{\hookrightarrow}

\newcommand*{\s}{s} % source map
\newcommand*{\rg}{r}% range map

\newcommand{\idealin}{\mathrel{\trianglelefteq}} % relation of being an ideal

\newcommand*{\NN}{\mathcal N}
\renewcommand*{\SS}{\mathcal S}

\newcommand*{\sbe}{\subseteq}

\DeclarePairedDelimiterX{\setgiven}[2]{\{}{\}}{#1\,{:}\,\mathopen{}#2}% set given by
\newcommand*{\cont}{C}%continuous functions
\newcommand*{\contz}{\cont_0}%continuous functions vanishing at infinity
\newcommand*{\contc}{\cont_c}%continuous functions with compact support
\newcommand*{\contb}{\cont_b}%continuous bounded functions
\newcommand*{\Mat}{\mathbb M}%matrices
\newcommand*{\congto}{\xrightarrow\sim}

\newcommand*{\K}{\mathcal K}% Compact operators

\newcommand{\zP}{\mathbb{Z}_{p}}
\newcommand{\Zp}{\mathbb{Z}_{p}}
\newcommand{\qP}{\mathbb{Q}_{p}}
\newcommand{\Qp}{\mathbb{Q}_{p}}
\newcommand{\OO}{\mathcal{O}}

\newcommand{\BB}{\mathcal{B}}
\newcommand{\BUN}{\mathcal{B}_{\leq 1}}

\newcommand{\GG}{\mathcal{G}}

\newcommand{\UU}{\mathcal{U}}

\newcommand{\colim}{\mathrm{colim}}

\newcommand{\ev}{\mathrm{ev}}
\newcommand{\Hom}{\mathrm{Hom}}

\newcommand{\OA}{\mathrm{OpeAlg}_p}

\newcommand{\omax}{\otimes_{\mathrm{max}}}
\newcommand{\omin}{\otimes_{\mathrm{min}}}
\newcommand{\comb}{\overbracket[.4pt][0.8pt]}% bornological completion

\newcommand\hot{\mathbin{\comb{\otimes}}}% bornologically complete tp
\newcommand{\oalg}{\otimes_{\mathrm{alg}}}

\newcommand*{\cstar}{\texorpdfstring{$C^*$\nobreakdash-\hspace{0pt}}{*-}}

\newcommand*{\braket}[2]{\langle#1\!\mid\!#2\rangle}

\makeatletter
    \def\@Obs#1{{\color{blue}{\uline{#1}}}}
    \def\Obs{\@ifstar{\@Obs}{\marginpar[\textcolor{blue}{\((\star)\)}]{\textcolor{blue}{\((\star)\)}}\@Obs}}
    
\makeatother

\begin{document}

\title{Operator algebras over the \(p\)-adic integers}
\author{Alcides Buss}
\email{alcides.buss@ufsc.br}
\address{Departamento de Matem\'atica\\
 Universidade Federal de Santa Catarina\\
 88.040-900 Florian\'opolis-SC\\
 Brazil}

\author{Luiz Felipe Garcia}
\email{lfgarcia98@gmail.com}
\address{Pós-Graduação em Matem\'atica\\
 Universidade Federal de Santa Catarina\\
 88.040-900 Florian\'opolis-SC\\
 Brazil}

\author{Devarshi Mukherjee}
\email{devarshi.mukherjee@uni-muenster.de}
\address{Westf\"alische-Wilhelms Universit\"at M\"unster\\ Mathematics M\"unster\\ Einsteinstrasse 62\\ 48149 M\"unster, Germany}

\begin{abstract}
We introduce $p$-adic operator algebras, which are nonarchimedean analogues of $C^*$-algebras. We demonstrate that various classical examples of operator algebras - such as group(oid) \(C^*\)-algebras - have nonarchimedean counterparts. The category of $p$-adic operator algebras exhibits similar properties to those of the category of real and complex $C^*$-algebras, featuring limits, colimits, tensor products, crossed products and an enveloping construction permitting us to construct $p$-adic operator algebras from involutive algebras over $\zP$. In several cases of interest, the enveloping algebra construction recovers the \(p\)-adic completion of the underlying \(\zP\)-algebra. We then discuss an analogue of topological \(K\)-theory for Banach \(\zP\)-algebras, and compute it in basic examples such as the \(p\)-adic Cuntz algebra and rotation algebras. Finally, for a large class of \(p\)-adic operator algebras, we show that our \(K\)-theory coincides with the reduction mod \(p\) of Quillen's algebraic K-theory. 
\end{abstract}



\thanks{The authors thank Andreas Thom, Guillermo Corti\~nas, Ralf Meyer and Adam Dor-On for several interesting discussions that led to this article. The first named author was supported by CNPq and Fapesc - Brazil, and by Humboldt - Germany. The second named author was supported by CAPES. The third named author was funded by a Feodor-Lynen Fellowship of the Alexander von Humboldt Foundation, which was carried out at the University of Buenos Aires. We thank the anonymous referee for helpful comments.}

\maketitle

\tableofcontents

\section{Introduction}

The theory of commutative Banach algebras over nonarchimedean valued fields has been extensively developed in the works of Tate, Berkovich (\cite{berkovich2012spectral}) and Huber (\cite{huber1993continuous}) in the study of nonarchimedean analytic geometry. Briefly, quotients of \(p\)-adic completions of polynomial rings with coefficients in the \(p\)-adic numbers provide local models for rigid analytic spaces. Recent work by Ben-Bassat-Kremnizer (\cite{ben2017non}) and Scholze-Clausen (\cite{condensed1}) interpret and unify various forms of analytic geometry (both complex and nonarchimedean) within their respective frameworks of bornological and condensed mathematics. Related to these approaches to analytic geometry is the work of Bambozzi-Mihara (\cite{bambozzi2021derived}), where the authors use commutative Banach algebras of continuous functions valued in an arbitrary Banach ring to topologise the category of compact Hausdorff spaces, with a view towards a theory of topological stacks. Now while the developments in the commutative world have been exciting, there has been essentially no progress in the (noncommutative) operator algebraic realm in the nonarchimedean setting to match the role played by \(C^*\)-algebras in complex and real noncommutative geometry and topology. Indeed, the pursuit of such a theory goes at least as far back as the 1970s (\cite{vanrooij}), where after developing the basic results on commutative Banach algebras and their Gelfand spectra, the author poses the problem of developing a theory analogous to \(C^*\)-algebras in the \(p\)-adic setting. This is the paradigm of this article.

There are several factors that have motivated us to study this class of algebras. We list some of these below:

\subsection*{Noncommutative analytic geometry}

One of several important ideas in Connes' noncommutative different geometry is the desingularisation of a ``bad quotient'' using a noncommutative operator algebra, such as a von Neumann or a \(C^*\)-algebra. More concretely, the orbit space \(X/G\) of the action of a topological group \(G\) on a locally compact Hausdorff space (or a smooth manifold) \(X\) is in general badly behaved as a topological space, as a consequence of which the commutative algebras of continuous (or smooth) functions on \(X/G\) are no longer the correct objects of study. Instead, one considers a noncommutative algebra, namely, the \textit{crossed product \(C^*\)-algebra} \(C_0(X)\rtimes G\), which is Morita equivalent to \(C_0(X/G)\) in case the original properties of the space \(X\) are inherited by the orbit space, as happens when the action is free and proper. In the nonarchimedean world, one could consider the action of a finite group \(G\) on an affine variety \(X = \mathsf{Spec}(A)\) over \(\fF_p\) and obtain an \(\fF_p\)-algebra \(A \rtimes G\) - the algebraic crossed product of the action of \(G\) on the ring of regular functions on the variety. Such algebras can be lifted to noncommutative Banach algebras of the form \(R \rtimes G\), where \(R\) is the \(p\)-adic completion of a smooth \(\zP\)-algebra lifting \(A\). Likewise, one could consider the action of a \(p\)-adic Lie group acting on a rigid analytic space or a totally disconnected compact Hausdorff space. If one views rigid analytic or totally disconnected spaces as classical spaces in the nonarchimedean world, the quotient spaces just described are \emph{noncommutative analytic spaces}, whose function algebras should be suitable analogues of the \(C^*\)-algebraic crossed product algebra.

\subsection*{Idempotent conjectures}

Kaplansky's idempotent conjecture states that for any ring \(R\) and a torsionfree group \(G\), the group algebra \(R[G]\) contains no nontrivial idempotents. In the complex case, this conjecture is known to be true for a large class of groups (including word-hyperbolic groups and reductive \(p\)-adic groups). The proof uses a formulation of the conjecture in terms of the Baum-Connes assembly map into the topological \(K\)-theory or local cyclic homology of the \emph{reduced group \(C^*\)-algebra}, and the latter contains the same information on idempotents as the group ring. When \(R = \fF_p\) or \(\zP\), the conjecture is still open for a general group \(G\). This was one of the main motivations for Anton Clau{\ss}nitzer and Andreas Thom \cite{thom-claussnitzer-article} to introduce separable \(p\)-adic Hilbert spaces, which they used to define an analogue of the reduced group \(C^*\)-algebra over \(\zP\) for \emph{countable} groups. Furthermore, in his thesis \cite{claussnitzer-thesis}, Clau{\ss}nitzer showed that idempotents in the \(\fF_p\)-algebra \(\fF_p[G]\) lift to idempotents in the resulting \(p\)-adic group \(C^*\)-algebra over \(\zP\). As a consequence of these idempotent lifting results, together with appropriate versions of topological \(K\)-theory and local cyclic homology for Banach \(\zP\)-algebras, one may formulate a version of the Baum-Connes assembly map. In forthcoming work, this will be used to investigate idempotent conjectures in the mixed characteristic setting.  

\subsection*{From operator algebras to their reduction mod \(p\)}
In what we have described so far, the general approach is to ``lift'' a geometric or algebraic property (such as smoothness and idempotence) over the residue field \(\fF_p\) to \(\zP\), and use operator algebraic techniques on some resulting topologically complete \(\zP\)-algebra. It turns out that several properties of topological \(\zP\)-algebras - which are manifestly analytic in nature - can be studied using their reduction mod \(p\), which is purely algebraic. For instance, the analogues of topological \(K\)-theory and local cyclic homology in the nonarchimedean world depend only on their reduction mod \(p\), that is, we have weak equivalences 
\begin{equation}\label{eq:KH-mod-p}
KH^{an}(\widehat{R}) \simeq KH(R/p), \quad \mathsf{HL}(\widehat{R}) \simeq \mathsf{HA}(R/p).
\end{equation}
In forthcoming work, we relate the simplicity of the rationalised \(p\)-adic operator algebras \(\OO_p(\mathcal{G}) \otimes \qP\) associated to Hausdorff ample groupoids \(\mathcal{G}\) to the simplicity of the Steinberg algebra \(\OO_p(\mathcal{G})/ p = \mathcal{S}(\mathcal{G}, \fF_p)\), which in turn is equivalent to the groupoid being minimal and effective \cite{Steinberg-Simplicity}*{Theorem~3.5}. 

We now briefly highlight our approach. Recall that a $C^*$-algebra can be characterised as a Banach $*$-algebra that can be isometrically represented in the space of bounded operators on a Hilbert space. This characterisation is taken as a definition for Banach algebras acting on \(L^q\)-spaces, where \(q \in [1, \infty)\), in what now constitutes the \emph{\(L^q\)-operator algebras}. Our definition of a $p$-adic operator algebra takes a morally similar route as that taken in the study of \(L^q\)-operator algebras. That is, we define our \(p\)-adic operator algebras as suitable subalgebras of bounded operators on a so-called \(p\)-adic Hilbert space. Of course, this already requires clarification as the simple-minded definition of a \(\qP\)-vector space with a sesquilinear form valued in \(\qP\) does not in general yield a self-dual topological vector space.  In \cite{claussnitzer-thesis} and \cite{thom-claussnitzer-article}, a definition of a $p$-adic Hilbert space is proposed, enabling the definition of $p$-adic operator algebras as $p$-adic Banach algebras with an involution that can be represented in the space of bounded operators on this $p$-adic Hilbert space. More precisely, given a countable set $X$, a $\zP$-module $\qP(X)$ consisting of a certain space of functions $X\to \qP$ is constructed, and this serves as a model for a ``separable'' $p$-adic Hilbert space. However, the countability on $X$ imposed in \cite{claussnitzer-thesis} significantly restricts the types of $p$-adic algebras that can be represented, which includes most of the algebras discussed in the motivating examples. In this article, we remove this cardinality constraint, which enables us to import several examples from the complex to the nonarchimedean setting, including groupoid algebras, crossed products, tensor products. In particular we obtain operator algebraic $p$-adic versions of some of the most important \cstar{}algebras, like graph algebras -- including in particular Cuntz algebras -- and rotation algebras, and more generally groupoid algebras. We also treat natural examples from number theory such as the Iwasawa algebra of a profinite group, and the Tate algebra as operator algebras. 

An important conceptual reason for us to be able to construct a large class of operator algebras in this setting is the following theorem (see \ref{prop:OA-has-limits} and \ref{prop:OA-has-colimits}).:

\begin{theorem}
The category of \(p\)-adic operator algebras and contractive \(*\)-algebra homomorphisms has all limits and colimits.
\end{theorem}

 The \(L^q\)-operator algebras are not closed under quotients by closed ideals for $q \neq 2$ (c.f. \cite{GETH}*{Theorem~2.5}). This makes it implausible to abstractly characterise such algebras in a manner similar to \(C^*\)-algebras. In this sense, our category of \(p\)-adic operator algebras behaves like \(C^*\)-algebras which has all limits and colimits. However, our algebras are also sufficiently different due to the absence of positivity in the \(p\)-adic setting, which is a key ingredient in the GNS construction.   

Implicit in the construction of quotients is an \emph{enveloping algebra} construction that associates to an involutive \(\zP\)-algebra \(A\) a \(p\)-adic operator algebra \(A^u\). This construction is analogous to the \(C^*\)-completion of a \(*\)-algebra over \(\C\). However, an important difference that is specific to the \(p\)-adic setting is that quotients introduce \(p\)-torsion - for instance, the quotient of \(\zP\) by the (closed) ideal \(p\zP\) is \(\Fp\), which can never even be a Banach algebra over $\zP$. As one might expect, our enveloping \(p\)-adic operator algebra kills such quotients, that is, \(A^u = 0\) for any \(\Fp\)-algebra viewed as a $\zP$-algebra. For \(p\)-torsionfree algebras, we have the following theorem (see \ref{prop:A=A^u}):

\begin{theorem}\label{thm:A^u=p}
    Let \(A\) be a torsionfree \(*\)-algebra over \(\zP\) admitting an isometric representation into \(\qP(X)\) for some set \(X\). Suppose the canonical \(p\)-adic norm on \(A\) is equivalent to the norm induced by the representation on \(\qP(X)\), then \(A^u\) is isometrically \(*\)-isomorphic to the \(p\)-adic completion \(\widehat{A} = \varprojlim_{n} A/p^n A\) with the \(p\)-adic norm. 
\end{theorem}

The hypotheses of Theorem \ref{thm:A^u=p} are satisfied by all our algebras of interest, including \'etale groupoid algebras with coefficients in \(\zP\), affinoid Banach \(\zP\)-algebras and algebras of continuous functions \(C_0(X,\zP)\) on locally compact Hausdorff spaces.  

The enveloping algebra construction also allows for the definition of a maximal tensor product. As in the \(C^*\)-case, the \emph{maximal tensor product} \(A \omax B\) of two \(p\)-adic operator algebras \(A\) and \(B\) is defined as the enveloping operator algebra \((A \otimes B)^u\) of the algebraic tensor product. There is also a minimal tensor product \(A \omin B\) using the given representations of \(A\) and \(B\) on \(p\)-adic Hilbert spaces. However, in sharp contrast to the archimedean, both these tensor products coincide with the \emph{completed projective tensor product} \(\otimes_{\pi}\) in several examples (see \ref{thm:nuclearity1}):

\begin{theorem}\label{thm:nuclear}
    Let \(A\) and \(B\) be \(p\)-adic operator algebras such that the algebraic tensor product \(A \otimes B\) satisfies the assumptions of Theorem \ref{thm:A^u=p}. Then \[A \omax B \cong A \otimes_{\pi} B \cong A \omin B.\]
\end{theorem}

The last section of the article is dedicated to \(K\)-theory computations using ~\eqref{eq:KH-mod-p}. Here even for the relatively straightforward cases we consider, such as the Cuntz algebras and rotation algebras, the computations and methods start to diverge significantly from the archimedean case. For instance, our computations show that the $p$-adic rotation algebras are not homotopy equivalent and have different invariants; this is expected as the $p$-adic rotation algebras should be viewed as noncommutative versions of a ``disconnected $p$-adic torus'', and a ``disconnected $p$-adic circle'' prevents a homotopy to exist.

 Conceptually, topological \(K\)-theory in the \(p\)-adic setting behaves very differently from its complex counterpart as in general, \(KH_n^{\mathrm{an}}(A) \neq KH_{n+2}^{\mathrm{an}}(A)\). Moreover, Bott periodicity fails for both the
torsion subgroups and for the rationalized groups; in particular, it is not
rationally isomorphic to rational cyclic homology theories such as analytic
and periodic cyclic homology studied in \cites{cortinas2017nonarchimedean, cortinas2019non, meyer2023local}. Our most general result (see \ref{lem:supercoherent}) compares analytic and algebraic \(K\)-theory of the reduction mod \(p\) as follows:

\begin{theorem}\label{thm:K-main}
    Let \(A\) be a \(p\)-adic operator algebra whose reduction mod \(p\) is a regular, supercoherent \(\fF_p\)-algebra. Then \(KH_n^\an(A) \cong K_n(A/pA)\) for all \(n \in \Z\). 
\end{theorem}

Theorem \ref{thm:K-main} answers a question of Ryszard Nest. The regular supercoherence hypothesis is satisfied for instance by Noetherian regular rings (such as the coordinate rings of algebraic varieties) in the commutative setting and Leavitt path algebras in the noncommutative setting.

\section{Preliminaries on nonarchimedean analysis}

In this section, we provide some background material on analysis over the \(p\)-adic integers to ensure that the article is reasonably self-contained to the operator algebraist. Throughout this article, \(p\) will denote a fixed prime number and \(\qP\) the field of \(p\)-adic numbers, which is constructed by completing the field of rationals \(\qQ\) in the \(p\)-adic norm \[\abs{x}_p = \frac{1}{p^l},\] for $x \in \qQ \setminus \set{0}$, where \(l\) is the integer satisfying \(x = \frac{m}{n}p^l\), with $m$ and $n$ not divisible by $p$. The resulting object is a nonarchimedean Banach field, where the qualifier ``nonarchimedean'' refers to the strict triangle inequality \(\abs{x+y}_p \leq \max \{\abs{x}_p,\abs{y}_p\}\) for all \(x\), \(y \in \qP\). It turns out that the strict triangle inequality leads to a vastly different theory from the archimedean analogue, starting with the observation that the \(p\)-adic topology on  \(\qP\) is totally disconnected. 

Our main objects of study are actually topological modules and algebras over the complete discrete valuation ring \(\zP = \setgiven{x \in \qP}{\abs{x}_p \leq 1}\), rather than \(\qP\) - this mixed characteristic nature of analysis is another difference from the archimedean case. The topological ring \(\zP\) is an example of a (nonarchimedean) \textit{Banach ring}, that is, a ring with a nonarchimedean norm, with respect to which it is complete. 

\begin{definition}\label{def:Banach-module}
    A \textit{nonarchimedean semi-normed} \(\zP\)-module \(M\) is a \(\zP\)-module together with a function \(\norm{-} \colon M \to \rR_{\geq 0}\) such that:
    \begin{itemize}
        \item \(\norm{\lambda \cdot x} = \abs{\lambda}_p \norm{x}\) for all \(x \in M\) and \(\lambda \in \zP\);
        \item \(\norm{x + y} \leq \max \{ \norm{x}, \norm{y}\}\) for all \(x\), \(y \in M\). 
    \end{itemize}
  A nonarchimedean semi-norm is called a \textit{norm} if \(\norm{x} = 0\) if and only if \(x = 0\). Finally, if the norm on \(M\) is complete, we say that \(M\) is a nonarchimedean \textit{Banach \(\zP\)-module}.   
\end{definition}

Note that the definition above makes sense for any complete discrete valuation ring in place of \(\zP\). Furthermore, it might appear to be a redundancy to talk about \emph{nonarchimedean} semi-norms when the base ring is nonarchimedean, but we point out that this is additional data. In ongoing work (see \cite{ben2020fr}, for instance) on derived analytic geometry, the authors make the distinction and consider archimedean modules over nonarchimedean Banach rings (and vice versa). In this article we will only study nonarchimedean normed \(\zP\)-modules. Having made this clarification, we will drop the prefix ``nonarchimedean'', unless specifically warranted.

\begin{example}\label{ex:main-Banach}
    Let \(M\) be a flat (or equivalently, \(p\)-torsionfree) \(\zP\)-module. Additionally suppose that $$\bigcap_{n \in \nN} p^n M = \set{0}$$such $\zP$-modules are called \textit{$p$-adically separated}. Define \[\widehat{M} = \varprojlim_{m} M/p^m M,\]known as the \(p\)-adic completion of $M$, we get a torsionfree, \(p\)-adically complete \(\zP\)-module. Consider the function \(\nu_p(x)= \sup \setgiven{n \in \N}{x \in p^n M}\) on \(M\); this is called the \textit{\(p\)-adic valuation of \(M\)}.  Then the assignment \[\norm{\cdot}_p \colon M \to \R_{\geq 0}, \quad \norm{x}_p = \frac{1}{p^{\nu_p(x)}}\] is a nonarchimedean norm, which we call the \textit{canonical $p$-adic norm}. The completion of $M$ under $\|\cdot\|_p$ is isomorphic to $\widehat{M}$ as a normed $\zP$-module. Since \(\widehat{M}\) is \(p\)-adically complete, \((\widehat{M}, \norm{\cdot}_p)\) is a Banach \(\zP\)-module.  Conversely, if \(M\) is a flat \(\zP\)-module that is complete in the canonical norm, then \(M\) is \(p\)-adically complete.  
\end{example}

Example~\ref{ex:main-Banach} yields an algorithm to generate examples of Banach \(\zP\)-modules. Furthermore, notice that the requirement that a bounded seminorm on a \(\zP\)-module be a norm necessitates that we rule out \(\fF_p\)-vector spaces (and \(\fF_p\)-algebras) from the domain of our interest. 

There are two possible categories, or morphisms, one could consider between Banach \(\zP\)-modules. We call a \(\zP\)-linear map \(T \colon M \to N\) \textit{bounded} if \(\norm{T(x)} \leq C \norm{x}\) for all \(x \in M\) and some \(C >0 \). A bounded \(\zP\)-linear map is said to be \textit{contractive} if \(\norm{T(x)} \leq \norm{x}\) for all \(x \in M\).  We denote the former and the latter category by \(\mathsf{Ban}_{\zP}\) and \(\mathsf{Ban}_{\zP}^{\leq 1}\). In the complex case, one usually works in the category \(\mathsf{Ban}_\C\) of complex Banach spaces - particularly if \(C^*\)-algebras are one's primary concern. This is because the contractive category of Banach spaces is not \(\Z\)-linear, and \(*\)-homomorphisms of \(C^*\)-algebras are automatically contracting - the latter implying the existence of all limits and colimits of \(C^*\)-algebras. In other words, the automatic contraction property means that one can stay within the additive category \(\mathsf{Ban}_\C\), while at the same time using the existence of arbitrary limits and colimits in  \(\mathsf{Ban}_{\C}^{\leq 1}\) to get a rich category of \(C^*\)-algebras. Herein lies another key difference in the nonarchimedean setting - the ultrametric property implies that the category \(\mathsf{Ban}_{\zP}^{\leq 1}\) is additive, and still has all limits and coproducts. However, as we are working over \(\zP\) (rather than \(\qP\)), quotients are badly behaved as they create torsion. Indeed, in the simplest case, the (algebraic) cokernel of the map \[\zP \overset{p}\to \zP\] is \(\fF_p\), which can never be a Banach \(\zP\)-module for any norm.

We now discuss symmetric monoidal structures on  \(\mathsf{Ban}_{\zP}\). Let \(M\) and \(N\) be two Banach spaces. Their completed projective tensor product (\cite{schneider2013nonarchimedean})  \(\otimes_\pi\) is given by the completion of the algebraic tensor product \(M \oalg N\) in the norm defined by 
\[\norm{x} = \inf \bigg\{\max \setgiven{\norm{m_i}\norm{n_i}}{i = 1, \dotsc, n} :x = \sum_{i= 1}^n m_i \otimes n_i\bigg\}.\] 

 This bifunctor \(- \otimes_\pi - \colon \mathsf{Ban}_{\zP} \times \mathsf{Ban}_{\zP} \to \mathsf{Ban}_{\zP}\) yields a closed symmetric monoidal structure on \(\mathsf{Ban}_{\zP}\) as well as on \(\mathsf{Ban}_{\zP}^{\leq 1}\). A \textit{Banach \(\zP\)-algebra} is an algebra object in the category \(\mathsf{Ban}_{\zP}^{\leq 1}\). The \(p\)-adic operator algebras that we will define in this article will be a certain subcategory of these Banach algebras that can be represented on a \(p\)-adic analogue of a Hilbert space, that we will define in the next section. 

We end this section by introducing another subcategory of Banach \(\zP\)-algebras in which our algebras will often lie. Recall from Example~\ref{ex:main-Banach} that starting with a flat, $p$-adically separated \(\zP\)-module \(M\), one can take its \(p\)-adic completion \(\widehat{M}\), which is a Banach module with respect to the canonical $p$-adic norm. If one considers a different seminorm \((M, \norm{-})\), then writing \(x \in M\) as \(x = p^m y\) for some \(m \in \N\), and \(y \in M \setminus pM\), we have \(\norm{x} = \norm{x}_p \norm{y}\). Thus, if \(\norm{-}\) is bounded (which will always be the case in our theory), then the identity map \((M, \norm{-}_p) \to (M, \norm{-})\) is bounded. The identity map \((M, \norm{-}) \to (M, \norm{-}_p)\) need not, however, be bounded. In what follows, we provide several equivalent criteria for when this happens to be the case:% and  if and only if \(p M\) is open in (the norm topology on) \(M\), and when this happens, we call \((M, \norm{-})\) a \textit{bornological Banach \(\zP\)-module}.

\begin{lemma}\cite{cortinas2017nonarchimedean}*{Example 2.11}\label{lem:bornological-Banach}
    Let \(M\) be a \(p\)-adically separated \(\zP\)-module with a bounded seminorm. The following are equivalent:
    \begin{enumerate}
        \item\label{lem:banach1} The identity map \((M, \norm{-}) \to (M, \norm{-}_p)\) is bounded;
        \item\label{lem:banach2} The seminorm \(\norm{-}\) is equivalent to the canonical $p$-adic seminorm on each closed ball \(B_\rho = \setgiven{x \in M}{\norm{x} \leq \rho}\) for \(\rho > 0\);
        \item\label{lem:banach3} There is a \(\delta>0\) such that \(B_\delta \subset p M\), that is, \(pM\) is open in the norm topology on \(M\).
    \end{enumerate}
\end{lemma}

\begin{definition}\label{def:bornological-Banach}
    A \(p\)-adically complete \(\zP\)-module with a bounded norm satisfying any of the equivalent conditions of Lemma~\ref{lem:bornological-Banach} is called a \textit{bornological Banach module}.
\end{definition}

Algebras whose underlying Banach modules are bornological will play an important role in the later sections. 

\begin{remark}\label{rem:d-complete}
    In \cite{claussnitzer-thesis}, a related notion of \textit{\(d\)-complete} \(\zP\)\nb-modules is discussed. More concretely, in \cite{claussnitzer-thesis}*{Definition~2.3.2} the author says that a seminormed \(\zP\)\nb-module \((M, \norm{-})\) is \textit{\(b\)-ultra-seminormed} if its norm is discretely valued, that is, the values \(\norm{M \setminus \{0\}}\) of the seminorm lie in \(p^{ - \N} \cup \{0\}\). And a \(b\)-ultra\-semi\-normed \(\zP\)-module is called \(d\)-complete if \(B_{1/p} \subseteq p M\). In particular, by Lemma~\ref{lem:bornological-Banach} \(d\)-complete \(\zP\)\nb-modules that are complete in their norm are bornological Banach \(\zP\)-modules.  
\end{remark}
 
\section{The \(p\)-adic analogue of a Hilbert space}

In this section, we recall the notion of a \(p\)-adic Hilbert space developed by Clau{\ss}nitzer and Thom in \cite{thom-claussnitzer-article,claussnitzer-thesis}.  Briefly, their construction associates to a countable set \(X\) a self-dual topological abelian group \(\qP(X)\) that has the structure of a \(\zP\)-module (and not in general a \(\qP\)-vector space). Our main examples, however, necessitate that the Thom-Clau{\ss}nitzer construction be extended to uncountable sets \(X\). This poses some technical challenges, whose resolution is the content of this section. 

\begin{definition}[The space $\qP(X)$]\label{def1} Let $X$ be an arbitrary nonempty set. Define
\begin{align*}
    \qP(X) = \set{\xi \colon X \to \qP \mid \xi(x) \in \zP \text{ for all but finitelly many }x \in X}.
\end{align*}
Note that $\qP(X)$ can be written as a filtered union of the spaces
\begin{align*}
    \qP(X) = \bigcup_{\substack{P \subseteq X \\ P\text{ finite subset}}} \prod_{x \in P} \qP \times \prod_{x \not\in P} \zP. 
\end{align*}
We equip $\qP(X)$ with the colimit topology of this union, where each $\prod_{x \in P} \qP \times \prod_{x \not\in P} \zP $ carries the product topology. Explicitly, $U \subseteq \qP(X)$ is an open subset if and only if for every finite subset $P \subseteq X$, the intersection $$U \cap \left( \prod_{x \in P} \qP \times \prod_{x \not\in P} \zP \right)$$ is open in the product topology on $U_P = \prod_{x \in P} \qP \times \prod_{x \not\in P} \zP$. Note in particular that the sets \(U_P\) are open for each finite set \(P\), and taking unions over \(P\) gives a basis for the topology. We denote this topology by \(\tau\); observe its similarity to the restricted product topology on the adeles. In particular, the topological space \((\qP(X),\tau)\) is locally compact. When \(X\) is a finite set of cardinality \(n\), \(\qP(X)\) coincides with the product space \(\qP^n\). However, in general \(\qP(X)\) is only a \(\zP\)-module.

We may also consider $\qP(X)$ with the topology induced by the supremum norm
\begin{align*}
    \norm{\xi}_\infty \defeq \sup_{x \in X} \abs{\xi(x)}_p
\end{align*}
for $\xi \in \qP(X)$.
\end{definition}

\begin{lemma}\label{lem:Hilbert-space-Banach}
   The \(\zP\)-module \(\qP(X)\) with the norm above is complete, that is, it is a Banach \(\zP\)-module.
\end{lemma}

\begin{proof}
    The proof of \cite{claussnitzer-thesis}*{Lemma 1.5.3} works mutatis-mutandis.
\end{proof}

Although the $\tau$-topology is less convenient to work with than the norm topology, it is the more structurally significant one, as it turns $\qP(X)$ into a self-dual locally compact abelian group (see Proposition~\ref{selfdual} below), which will be relevant later on. We now provide a useful criterion for convergence of sequences in the $\tau$-topology.

\begin{lemma}[Sequential convergence criterion]\label{CCC} Let $X$ be a nonempty set. A sequence $(\xi_n)_n$ in $\qP(X)$ with the $\tau$-topology converges to $\xi$ if and only if it converges entrywise to $\xi$, and if the set $\set{x \in X \mid \exists n \in \nN : |\xi_n(x)| > 1}$ is finite.
\end{lemma}
\begin{proof}
Suppose that $\lim_n \xi_n = \xi$ in the $\tau$-topology. Let
\begin{align*}
    P := \set{x \in X \mid |\xi(x)| > 1}
\end{align*}
Note that $P$ is finite as $\xi \in \qP(X)$. Let $x \in X$ arbitrary. Define for every $m \in \nN \setminus \set{0}$ the open neighborhood of $\xi$ $$V_{x, m} \defeq B(\xi(x), \frac{1}{m+1}) \times \prod_{y \in P \setminus \set{x}} B(\xi(y), 1) \times \prod_{x \in P^c \setminus \set{x}} \zP$$ where $B(a, r)$ is the open ball with center $a$ and radius $r$ in $\qP$. As $\lim_n \xi_n = \xi$, for each $m$, there exists a $n_0 \in \nN$, such that $\xi_n \in V_{x, m}$ for $n > n_0$. Varying $m \in \nN$, we see that $\xi_n$ converges to $\xi$ in the $x$ entry, and since $x$ was arbitrary it follows that \((\xi_n)\) converges pointwise to \(\xi\). 

For the second part, fix $P \subseteq X$ a finite subset such that $\xi \in \prod_{x \in P} \qP \times \prod_{x \not\in P} \zP$. Note that $\prod_{x \in P} \qP \times \prod_{x \not\in P} \zP$ is a open neighborhood of $\xi$, therefore, there is an $n_0 \in \nN$, such that for $n > n_0$ we have $\xi_n \in \prod_{x \in P} \qP \times \prod_{x \not\in P} \zP$. Consequently, the set $$\set{x \in X \mid \exists n > n_0 : |\xi_n(x)| > 1} \subseteq P$$ and is hence finite. As we have finitely many options, the set
$$\set{x \in X \mid \exists n \leq n_0 : |\xi_n(x)| > 1} $$ is also finite, proving the claim that $$\set{x \in X \mid \exists n \in \nN : |\xi_n(x)| > 1}$$ is finite.

Conversely, suppose we are given a sequence $(\xi_n)$ converging entrywise to $\xi$, and that the set $$P_0 \defeq \set{x \in X \mid \exists n \in \nN : |\xi_n(x)| > 1}$$ is finite. Note that for every $n \in \nN$ we have 
$$\xi, \xi_n \in \prod_{x \in P_0} \qP \times \prod_{x \not\in P_0} \zP.$$ 
Let $U \subseteq \qP(X)$ be an open set containing $\xi$. Since $\xi \in \prod_{x \in P_0} \qP \times \prod_{x \not\in P_0} \zP$ we can choose $U \subseteq \prod_{x \in P_0} \qP \times \prod_{x \not\in P_0} \zP$. Therefore $U$ is of the form
\begin{align}\label{U-form}
    U = U_{x_1} \times \cdots \times U_{x_n} \times V_{y_1} \times \cdots \times V_{y_m} \times \prod_{x \in X \setminus \set{x_1,\dots, x_n, y_1, \dots, y_m}} \zP
\end{align}
for some $U_{x_i}$ open subset of $\qP$, $V_{y_i}$ open subset of $\zP$ where $x_i \in P_0$ and $y_i \in X \setminus P_0$. Hence, as $\xi_n$ converges entrywise to $\xi$, it is easy to see from \ref{U-form} that there exists $n_0 \in \nN$ such that for $n > n_0$ we have $\xi_n \in U$. Showing that $\xi_n$ converges to $\xi$.
\end{proof}

Let  $\widehat{\qP(X)} \defeq \mathsf{Hom}(\qP(X), \mathbb{T})$ be the Pontryagin dual of the locally compact abelian group $\qP(X)$. Similar to complex or real Hilbert spaces, the topological \(\zP\)-module $\qP(X)$ comes with a natural pairing induced by an isomorphism of locally compact abelian groups $\qP(X) \cong \widehat{\qP(X)}$ (see \ref{selfdual} below). This pairing plays the role of an inner product in the nonarchimedean setting. 

\begin{definition}[Pairing] Let $X$ be a nonempty set. There is a natural pairing
\begin{align*}
    \langle \cdot , \cdot \rangle : \qP(X) \times \qP(X) \to \frac{\qP}{\zP} \cong \frac{\zZ[1/p]}{\zZ} \subseteq \T
\end{align*}
defined by
\begin{align*}
    \langle \xi, \eta \rangle \defeq \sum_{x \in X} (\xi(x) \eta(x) + \zP).
\end{align*}
\end{definition}

Note that this sum is finite as only finitely many terms of this sum are not in $\zP$. 

\begin{proposition} The pairing described above is $\zP$-bilinear, symmetric and jointly continuous.
\end{proposition}
\begin{proof}
The bilinearity and symmetry are clear from the definition. We now verify continuity where we equip the set $\qP/ \zP$ with the discrete topology. Consequently, it remains to check that the map \(\qP(X) \times \qP(X) \to \qP/\zP\) is continuous. Let $b = \langle \cdot , \cdot \rangle$, and let $a + \zP \in \qP/\zP$ be arbitrary. We need to show that $b^{-1}(a + \zP)$ is an open subset. To this end, let $(\xi_0, \eta_0) \in b^{-1}(a + \zP)$ and let 
\begin{align*}
    \varepsilon := \min \set{\frac{1}{2} \max \set{ \|\xi_0\|_{\infty}^{-1}, \|\eta_0\|_{\infty}^{-1} }, 1}.
\end{align*}
Consider the finite sets
\begin{align*}
    S := \set{ x \in X \mid |\xi_0(x)| > 1} \\
    T := \set{ x \in X \mid |\eta_0(x)| > 1}
\end{align*}
and the open subsets
\begin{align*}
    U := \prod_{x \in S \cup T} B(\xi_0(x) , \varepsilon) \times \prod_{x \not\in S \cup T} \zP \\
    V := \prod_{x \in S \cup T} B(\eta_0(x) , \varepsilon) \times \prod_{x \not\in S \cup T} \zP
\end{align*}
Therefore $U \times V$ is an open neighborhood of $(\xi_0, \eta_0)$. Now for any $(\xi_1, \eta_1) \in U \times V$, we have
\begin{align*}
    \langle \xi_1 , \eta_1 \rangle &= \sum_{x \in X} (\xi_1(x) \eta_1(x) + \zP)  \\
    &= \sum_{x \in X} (((\xi_1(x) - \xi_0(x)) + \xi_0(x)) ((\eta_1(x) - \eta_0(x)) + \eta_0(x)) + \zP)  \\
    &= \left(\sum_{x \in X} ((\xi_1(x) - \xi_0(x)) (\eta_1(x) - \eta_0(x)) + \zP)   \right) \\ &+ \left(\sum_{x \in X} ((\xi_1(x) - \xi_0(x)) \eta_0 (x) + \zP)   \right) \\
     &+ \left(\sum_{x \in X} \xi_0(x) (\eta_1(x) - \eta_0(x)) + \zP)   \right) \\ &+ \left( \sum_{x \in X} (\xi_0(x) \eta_0(x) + \zP) \right). 
\end{align*}
By the choice of $\varepsilon$ above, every term in the first three sums is in $\zP$, so that the only sum that remains is the fourth one, which equals $a + \zP$, completing the proof.
\end{proof}

The following is an analogue of the Riesz Representation Theorem:

\begin{proposition}\label{selfdual} As a topological group $(\qP(X), \tau)$ is self-dual, that is, there is an isomorphism of topological groups
\begin{align*}
    \qP(X) &\cong \Hom(\qP(X), \T) \\
    \xi &\mapsto \langle \xi, - \rangle 
\end{align*}
\end{proposition}
\begin{proof} See \cite{HR63}*{23.33, 25.1, 25.34}.
\end{proof}

Continuing the analogy with Hilbert spaces, we now discuss operators and their adjoints on $\qP(X)$. Here, by an operator we mean a $\zP$-linear continuous map $\qP(X) \to \qP(X)$. 

\begin{lemma} If $A\colon \qP(X) \to \qP(X)$ is a continuous group homomorphism, then $A$ is $\zP$-linear.
\end{lemma}
\begin{proof} One can easily prove this fact using the fact that $\zZ$ is dense in $\zP$.
\end{proof}

The following results relate the $\tau$-topology with the norm topology on $\qP(X)$.

\begin{lemma}\label{normbounded} If $K \subseteq \qP(X)$ is compact in the $\tau$-topology, then $K$ is norm bounded.
\end{lemma}
\begin{proof}
Let $K$ be a compact subset of $\qP(X)$ in the $\tau$-topology. We have
\begin{align*}
    K \subseteq \qP(X) = \bigcup_{\substack{P \subseteq X \\ P\text{ finite subset}}} \prod_{x \in P} \qP \times \prod_{x \not\in P} \zP. 
\end{align*}
Therefore we can find a finite subset $F \subseteq X$ such that
\begin{align*}
    K \subseteq \prod_{x \in F} \qP \times \prod_{x \not\in F} \zP.
\end{align*}
For every $n \in \nN$ write
\begin{align*}
    U_n = \prod_{x \in F} B(0,n) \times \prod_{x \not\in F} \zP
\end{align*}
where $B(0,n)$ is the open ball with center $0$ and radius $n$. Note that $U_n$ is open and
\begin{align*}
    K \subseteq \prod_{x \in F} \qP \times \prod_{x \not\in F} \zP = \bigcup_{n \in \nN} U_n
\end{align*}
As $K$ is compact, there is a $n \in \nN$ such that $K \subseteq U_n$. Hence $K$ is bounded.
\end{proof}

\begin{lemma}\label{normcontinuous} Let $T \colon \qP(X) \to \qP(X)$ be a $\zP$-linear function. Then $T$ is norm-continuous if and only if it is norm-bounded.
\end{lemma}
\begin{proof}
By Lemma~\ref{lem:Hilbert-space-Banach}, the space \(\qP(X)\) is a Banach \(\zP\)-module. Now use that continuity and boundedness are equivalent on Banach modules. \qedhere 
\end{proof}

\begin{corollary}\label{cor:operator-continuous} If $T\colon \qP(X) \to \qP(X)$ is $\zP$-linear and $\tau$-continuous operator, then $T$ is norm-continuous.
\end{corollary}
\begin{proof} By Lemma~\ref{normcontinuous}, it suffices to show that $T$ is norm-bounded. Let $B$ be the closed unit ball centered at $0$, that is, $B = \zP(X) = \prod_{x \in X} \zP$. As $B$ is a cartesian product of compact spaces, then $B$ is compact in the $\tau$-topology. As $T$ is $\tau$-continuous and the continuous image of a compact set is compact, $T(B)$ is compact in the $\tau$-topology. Using Lemma~\ref{normbounded} we conclude that $T(B)$ is norm-bounded, and therefore $T$ is norm-continuous.
\end{proof}

We now generalise a result of Claußnitzer-Thom in \cite{claussnitzer-thesis}*{Theorem 1.4.4} to case the where \(X\) is uncountable. We have weakened the statement of the theorem because for an uncountable set $X$, the space $\qP(X)$ is not Polish.

\begin{theorem}[Claußnitzer - Thom]\label{CT} Let $\sigma\colon \qP(X) \times \qP(X) \to \T$ be a biadditive form that is separately continuous in each entry with $\T$ carrying the discrete topology. Then there exists an unique $\zP$-linear map $A\colon \qP(X) \to \qP(X)$ such that for all $\xi$, $\eta \in \qP(X)$ we have 
\begin{align*}
    \langle A \xi, \eta \rangle = \sigma (\xi, \eta)
\end{align*}  
\end{theorem}
\begin{proof}
For each $\xi \in \qP(X)$, define
\begin{align*}
    \tau_{\xi}\colon \qP(X) &\to \T \\
    \eta &\mapsto \sigma (\xi, \eta).
\end{align*}
As $\qP(X)$ is self-dual by Proposition~\ref{selfdual}, there exists a unique $A(\xi) \in \qP(X)$ such that for all $\eta \in \qP(X)$, we have 
\begin{align*}
    \langle A(\xi), \eta \rangle = \sigma(\xi, \eta),
\end{align*}
which defines a map 
\begin{align*}
    A\colon \qP(X) &\to \qP(X) \\
    \xi &\mapsto A(\xi). 
\end{align*}
For every $\eta, \xi, \xi' \in \qP(X)$ we have 
\begin{align*}
    \langle A(\xi + \xi'), \eta \rangle &= \sigma( \xi + \xi', \eta ) \\
    &= \sigma(\xi, \eta) + \sigma(\xi', \eta) \\
    &= \langle A(\xi), \eta \rangle + \langle A(\xi'), \eta \rangle \\
    &= \langle A(\xi) + A(\xi'), \eta \rangle,
\end{align*}
yielding additivity.

Using a similar argument one shows that $A(\alpha \xi) = \alpha A(\xi)$ for $\xi \in \qP(X)$ and $\alpha \in \zP$.
\end{proof}

When $X$ is countable, the above result already yields that the map \(A\) is continuous, using that $\qP(X)$ is Polish. We however need the following more general result:

 \begin{proposition}\label{lcaadjoint}
     Let $G$ be a self-dual locally compact abelian group, with self-duality given by a bi-additive pairing $G\times G\to \T$, 
     $(x,y)\mapsto \braket{x}{y}$. Let $T\colon G\to G$ be an endomorphism which is \emph{adjointable} in the sense that 
     there is a (necessarily unique) map $S:=T^*\colon G\to G$ satisfying
     $$\braket{T(x)}{y}=\braket{x}{S(y)}.$$
     Then $S$ is an endomorphism with $S^*=T$, and $T$ is continuous if and only if $S$ is continuous.
     Moreover, $T$ (and hence also $S$) is continuous whenever $S=T^*$ exists and is \emph{bounded}, meaning that it 
     sends compact subsets $K\sbe G$ to precompact\footnote{A precompact set is a subset whose closure is a compact set.} subsets $S(K)\sbe G$.
 \end{proposition}
 \begin{proof}
     The first assertion follows from the bi-additivity of the pairing. Let us now assume that $S$ exists and is bounded and prove that $T$ is continuous. By assumption, the pairing induces an isomorphism of topological groups (i.e. a bi-continuous, bijective homomorphism)
     $$\phi\colon G\congto \widehat{G}$$
     given by $\phi(x)(y):=\braket{x}{y}$. Now we recall that the topology on $\widehat{G}\cong G$ is the topology of pointwise uniform convergence on compact subsets: this means that a net $(x_i)$ converges to $x$ in $G$ if and only if for every $\epsilon>0$ and every compact subset $K\sbe G$, there is $i_0$ with
     $$|\braket{x_i}{y}-\braket{x}{y}|<\epsilon\quad\forall y\in K, i\geq i_0.$$ 
     Now, since $S$ is assumed to be bounded, $S(K)$ is precompact, that is, its closure $L:=\overline{S(K)}\sbe G$ is compact. Since $x_i\to x$, we have $\braket{x_i}{S(y)}\to \braket{x}{S(y)}$ uniformly for $y\in K$ by the above characterisation of the topology on $\widehat G\cong G$. But this means that $\braket{T(x_i)}{y}\to \braket{T(x)}{y}$ uniformly for $y\in K$, and again by the above characterisation, this implies $T(x_i)\to T(x)$, and therefore $T$ is continuous.
 \end{proof}

We are going to apply the above result for the canonical self-duality pairing of the additive abelian topological group $G=\qP(X)$. This pairing is even symmetric, meaning that $\braket{x}{y}=\braket{y}{x}$, but this assumption was not need above.

We can now finally prove the existence of the adjoint of a continuous operator.

 \begin{corollary} Any $\zP$-linear, $\tau$-continuous operator $T\colon \qP(X) \to \qP(X)$ is adjointable.
 \end{corollary}
\begin{proof}
For a fixed \(\eta \in \Q_p(X)\), we have a continuous group homomorphism \(\Q_p(X) \to \T\), \(\xi \mapsto \sigma(T(\xi), \eta)\). By the duality isomorphism of Theorem~\ref{CT}, there is a \(T^*(\eta)\) such that \(\sigma(T(\xi), \eta) = \langle \xi, T^*(\eta) \rangle\). Finally, by Proposition~\ref{lcaadjoint} and Corollary~\ref{cor:operator-continuous} the adjoint is \(\Z_p\)-linear and continuous.
\end{proof}

Using adjoints one can get a very simple description of the space of bounded operators on $\qP(X)$. This description will make a lot easier to explicitly construct representations on those spaces.

There is a simple characterisation of the compact subsets of $\qP(X)$ in the $\tau$-topology. Using this characterisation together with Proposition~\ref{lcaadjoint} will lead to an equivalent definition of the space of bounded operators.

\begin{proposition}\label{prop:characterisation-tau-compact} A subset $K \subset \qP(X)$ is compact in the $\tau$-topology if and only if $K$ is norm-bounded, $\tau$-closed and there exists a finite subset $F \subseteq X$ such that 
\begin{align*}
    K \subseteq \prod_{x \in F} \qP \times \prod_{x \in X \setminus F} \zP
\end{align*}
\end{proposition}
\begin{proof}
Let $K \subset \qP(X)$ be a \(\tau\)-compact subset. By Lemma~\ref{normbounded}, we know that $K$ is norm-bounded. As $\qP(X)$ is Hausdorff for the \(\tau\)-topology, we know that $K$ is $\tau$-closed. For the third property, we can cover
\begin{align*}
    K \subseteq \bigcup_{\substack{F \subseteq X \\ \text{finite}}} \prod_{x \in F} \qP \times \prod_{x \in X \setminus F} \zP. 
\end{align*}
As this is an open cover and $K$ is compact, there is a finite subcover, and a finite union of sets of the form $\prod_{x \in F} \qP \times \prod_{x \in X \setminus F} \zP$ , is of this form.

Now suppose that $K$ is norm-bounded, $\tau$-closed and exist a finite subset $F \subseteq X$ such that 
\begin{align*}
    K \subseteq \prod_{x \in F} \qP \times \prod_{x \in X \setminus F} \zP.
\end{align*}
Note that $K$ has the subspace topology of the product $\prod_{x \in F} \qP \times \prod_{x \in X \setminus F} \zP$. As $K$ is $\tau$-closed, because of the product topology we can write $K$ as 
\begin{align*}
    \left(\prod_{x \in F} \qP \times \prod_{x \in X \setminus F} \zP \right) \setminus \bigcup_i U_i \times V_i
\end{align*}
with $U_i$ and $V_i$ open subsets of $\prod_{x \in F} \qP$ and $\prod_{x \in X \setminus F} \zP$ respectively. Denote $A = \prod_{x \in F} \qP$ and $B = \prod_{x \in X \setminus F} \zP$. We have $K$ is a intersection of subsets of the form
\begin{align*}
    ((A \setminus U_i) \times B) \cup (A \times (B \setminus V_i)) \cup ((A \setminus U_i) \times (B \setminus V_i))
\end{align*}
each of these are compact because of Tychonoff Theorem and because $K$ is norm bounded.
\end{proof}

The next proposition is essential for later computations and constructions. The proof of this result is taken from the proof of \cite{claussnitzer-thesis}*{Theorem 2.1.2}.

\begin{proposition}\label{entries-goes-zero} Let $X$ be a infinite set. Let $T\colon \qP(X) \to \qP(X)$ be a $\zP$-linear function. Let $y \in X$ be an element and $\delta_y\colon X \to \qP$ be the function that values $1$ on $y$ and $0$ otherwise. Then
\begin{align*}
    \lim_{x \to \infty} T(\delta_y)(x) = 0
\end{align*}
Meaning that for every $\varepsilon > 0$, there exists a finite subset $F \subseteq X$ such that for $x \in X \setminus F$ we have
\begin{align*}
    |T(\delta_y)(x)| < \varepsilon
\end{align*}
\end{proposition}
\begin{proof}
Suppose by contradiction that there exist $y \in X$ such that 
\begin{align*}
    \lim_{x \to \infty} T(\delta_y)(x) \neq 0
\end{align*}
this means that there exists an $\varepsilon > 0$ such that 
\begin{align*}
    |T(\delta_y)(x)| > \varepsilon
\end{align*}
for infinitely many $x \in X$. Choose $\lambda \in \qP$ such that 
\begin{align*}
    |\lambda| \varepsilon > 1
\end{align*}
$\lambda \delta_y \in \qP(X)$, hence $T(\lambda \delta_y) \in \qP(X)$, therefore $|T(\lambda \delta_y)(x)| > 1$ for a finite number of $x \in Y$, but this is not what is happening, note that
\begin{align*}
    |T(\lambda \delta_y)(x)| = |\lambda T(\delta_y)(x)| = |\lambda| |T(\delta_y)(x)| > |\lambda| \varepsilon > 1
\end{align*}
for infinitely many $x \in Y$, which is a contradiction because $T(\lambda \delta_y) \in \qP(X)$.
\end{proof}

\begin{proposition}\label{bounded} Let $T\colon \qP(X) \to \qP(X)$ be a norm-continuous $\zP$-linear operator with $\|T\| \leq 1$. If $K \subset \qP(X)$ is $\tau$-compact, then $T(K)$ is $\tau$-precompact.
\end{proposition}
\begin{proof}
We will use \ref{prop:characterisation-tau-compact}. There is a finite set $S \subseteq X$ with 
$$K \subset \prod_{x \in S} \qP \times \prod_{x \in X \setminus S} \zP.$$
As $K$ is norm-bounded and $T$ is norm-continuous and $\zP$-linear, $T(K)$ is norm-bounded. Let $\xi \in \overline{T(K)}$ and \( P_{\xi} = \set{ x \in X \mid |\xi(x)| > 1}\). Then
\begin{align*}
    B_{\xi} = \prod_{x \in P_{\xi}} B(\xi(x), 1) \times \prod_{x \not\in P_{\xi}} \zP
\end{align*}
is an open neighborhood of $\xi$. Then  $B_{\xi} \cap T(K) \neq \varnothing$. Let $\eta \in B_{\xi} \cap T(K)$; then $\|\xi - \eta\| \leq 1$ by definition of $B_{\xi}$. Writing
\begin{align*}
    \xi = \eta + (\xi - \eta),
\end{align*}
we have , using that \(\eta \in T(K)\).

We just proved that every element of $\overline{T(K)}$ can be writen as an element of $T(K)$ plus an element of norm less or equal than $1$. Therefore, as $T(K)$ is bounded, $\overline{T(K)}$ is bounded.  

Now it only remains to show that $\overline{T(K)} \subseteq \prod_{x \in F} \qP \times \prod_{x \in X \setminus F} \zP$ for some finite subset $F \subseteq X$. As $\prod_{x \in F} \qP \times \prod_{x \in X \setminus F} \zP$ is closed in $\qP(X)$ in the $\tau$-topology, it is enough to show that $T(K) \subseteq \prod_{x \in F} \qP \times \prod_{x \in X \setminus F} \zP$. 

Since $K \subset \qP(X)$ is compact, there is a finite set $F' \subseteq X$ such that 
\begin{align*}
    K \subseteq \prod_{x \in F'} \qP \times \prod_{x \in X \setminus F'} \zP.
\end{align*}
For each $\xi \in K$, define 
\begin{align*}
    \xi_{0} := \begin{cases}
        \xi(x) \text{, for } x \in X \setminus F' \\
        0 \text{, otherwise.}
    \end{cases} 
\end{align*}
and
\begin{align*}
    \xi_1(x) := \begin{cases}
        \xi(x) \text{, for } x \in F' \\
        0 \text{, otherwise.}
    \end{cases} 
\end{align*}
Then \(T(\xi) = T(\xi_0) + T(\xi_1)\), and as  $\|\xi_0\| \leq 1$ and $\|T\| \leq 1$, we have $T(\xi_0) \in \prod_{x \in X} \zP$. As $\xi_1$ has finite support we can write
\begin{align*}
    T(\xi_1) &= T \left( \sum_{y \in F'} \xi(y) \delta_y \right) \\
    &= \sum_{y \in F'} \xi(y) T(\delta_y). 
\end{align*}
Also, as $K$ is compact, it is also norm-bounded by Proposition~\ref{prop:characterisation-tau-compact}. Then, there exists an $r \in \rR$ such that $\|\eta\| < r$ for all $\eta \in K$. Using Proposition~\ref{entries-goes-zero}, for every $y \in F'$, there exists a finite subset $P_y \subseteq X$, such that if $x \in X \setminus P_y$ we have that
\begin{align*}
    |T(\delta_y)(x)| < r^{-1}
\end{align*}
Hence, if $x \in X \setminus \bigcup_{y \in F'} P_y$ we got
\begin{align*}
    |T(\xi_1)(x)| \leq \max_{y \in F'} |\xi(y)\|T(\delta_y)(x)| \leq r \cdot r^{-1} = 1
\end{align*}
Let $F = \bigcup_{y \in F'} P_y$ which is a finite set since $F'$ and $P_y$ are finite. Consequently, for an arbitrary $\xi \in K$, we have $T(\xi) \in \prod_{x \in F} \qP \times \prod_{x \in X \setminus F} \zP$, as required. That \(T(K)\) is precompact now follows from the characterisation of Proposition~\ref{prop:characterisation-tau-compact}.
\end{proof}

\section{Bounded operators on $p$-adic Hilbert spaces}

Finally, we can define one of the main objects of this article - the space of bounded operators \(\BB(\qP(X))\) on $\qP(X)$. But instead of considering all bounded operators, we will work with the unit ball $\BUN(\qP(X))$. There are several reasons for doing this. First, most of the examples we deal with can already be represented on $\BUN(\qP(X))$. Second, $\BUN(\qP(X))$ has a straightforward description in terms of adjoints, making it very easy to construct examples. Third, the category of (contracting) Banach algebras that can be isometrically represented in $\BUN(\qP(X))$ is much better behaved categorically, having all limits, colimits, tensor products, crossed products and the enveloping algebra construction as we are going to see later in the text.

\begin{definition}[Bounded Operators] For a nonempty set $X$, we define the \(\zP\)-algebra
\begin{align*}
    \BB(\qP(X)) := \set{ T\colon \qP(X) \to \qP(X) : T \text{ is } \zP\text{-linear and }\tau\text{-continuous functions}}
\end{align*}
and equip it with the natural operator norm:
$$\|T\|:=\sup_{\|\xi\|\leq 1}\|T(\xi)\|.$$
We also equip its closed unit ball 
\begin{align*}
    \BUN (\qP(X)) := \set{T\in \BB(\qP(X)): \|T\| \leq 1}
\end{align*}
 with the same norm topology.
\end{definition}

\begin{proposition}\label{lem:p-adic-complete}
 The algebra \(\BB(\qP(X))\) is complete in the operator norm. Moreover, the operator norm and the canonical \(p\)-adic norm \ref{ex:main-Banach} on \(\BUN(\qP(X))\) coincide, so that \(\BUN(\qP(X))\) is a bornological Banach \(\zP\)-algebra. 
\end{proposition}

The easiest way to prove Proposition~\ref{lem:p-adic-complete} is by representing an operator \(A \in \BB(\qP(X))\) by a matrix \(M_A = (A_{x,y})\), and showing that the operator norm coincides with the matrix norm \(\norm{A} = \max_{x,y \in X} \abs{A_{x,y}}_p = \norm{M_A}\). The proof that $\BB(\qP(X))$ is complete will follow from Corollary~\ref{cor:sup-norm} below. 

Proposition~\ref{lem:p-adic-complete} vastly simplifies analysis in the \(p\)-adic operator algebraic setting. In several situations that we demonstrate more concretely in the next section, we have a subalgebra \(A \subseteq \BUN(\qP(X))\). Then the operator-norm closure of $A$ is equivalent to the \(p\)-adic closure, which in turn is simply its \(p\)-adic completion. 

In what follows, we provide an alternative and simpler description of the unit ball of \(\BB(\qP(X)\):

\begin{theorem}\label{nicedescription} We have
\[
  \BUN (\qP(X)) = \left\{T\colon \qP(X) \to \qP(X): \text{$T$ has an adjoint and } \|T\| \leq 1 \right\}.
\]
\end{theorem}
\begin{proof}
The inclusion $\sbe$ is immediate. To prove the reverse inclusion, let $T$ be a function $\qP(X) \to \qP(X)$ with an adjoint and $\|T\| \leq 1$. To see that $T$ is $\zP$-linear, let $\xi, \eta, \zeta\in \qP(X)$ and $a \in \zP$. Since the pairing is $\zP$-linear in each entry, we have
\begin{align*}
    \langle T(\xi + a\eta), \zeta \rangle &= \langle \xi + a \eta , T^*(\zeta) \rangle \\
    &= \langle \xi, T^*(\zeta) \rangle + a \langle \eta, T^*(\zeta) \rangle \\
    &= \langle T(\xi), \zeta \rangle + a \langle T(\eta), \zeta \rangle \\
    &= \langle T(\xi) + aT(\eta), \zeta \rangle.
\end{align*}
Since $\zeta$ is arbitrary, we get that
\begin{align*}
    T(\xi + a\eta) = T(\xi) + aT(\eta)
\end{align*}
yielding $\zP$-linearity. Moreover $\|T\| \leq 1$ so by Lemma~\ref{normcontinuous}, $T$ is norm continuous. Therefore, by Proposition~\ref{bounded}, $T$ sends $\tau$-compact sets into $\tau$-precompact ones, hence $T$ is $\tau$-continuous by Proposition~\ref{lcaadjoint}. This completes the proof that $T \in \BUN(\qP(X))$.
\end{proof}

The reader may wonder how our contractive operators on \(p\)-adic Hilbert spaces relate to operators on general Banach \(\qP\)-vector spaces considered, for instance, in \cite{schneider2013nonarchimedean}. For a non-empty set \(X\), consider the Banach \(\qP\)-vector space \(c_0(X,\qP) = \setgiven{\xi \in \qP(X)}{\lim \abs{\xi}_p = 0}\) with the supremum norm. This is the largest Banach \(\qP\)-subspace of \(\qP(X)\) as for any $\qP$-subspace $K$ with \( c_0(X,\qP) \subseteq K \subset \qP(X)\), we have that \(p^{-n}K \subseteq K\), so that \(K = c_0(X,\qP)\). Furthermore, it is easy to see that any  operator in \(\BUN(\qP(X))\) restricts to a bounded linear operator on \(c_0(X,\qP)\). Conversely, if \(T \colon c_0(X,\qP) \to c_0(X,\qP)\) is a bounded, contractive \(\qP\)-linear operator with an adjoint \(T^*\) for the canonical symmetric bilinear form \(\langle \xi,\eta \rangle = \sum_{x \in X} \xi(x) \eta(x)\), \(\xi\), \(\eta \in c_0(X,\qP)\), this pairing takes values on $\qP$ rather than $\qP/\zP$. Then $T$ extends to a \(\zP\)-linear, adjointable, contractive map \(\qP(X)\) by the following density result:

\begin{proposition}\label{prop:null-sequence-dense}
    The space \(c_0(X,\qP)\) is $\tau$-dense in \(\qP(X)\) in the \(\tau\)-topology\footnote{It is important to note that the norm topology on $c_0(X,\qP)$ is not the subspace topology of $\qP(X)$}.
\end{proposition}
\begin{proof}
Let $\xi \in \qP(X)$ be an arbitrary function. By the definition of $\qP(X)$, there exists a finite subset $F \subseteq X$ such that $\xi \in \prod_{x\in F}\qP\times \prod_{x\in X\setminus F}\zP$. Let $U$ be an open neighborhood of $\xi$. Without loss of generality, we can suppose that $U \subseteq \prod_{x\in F}\qP\times \prod_{x\in X\setminus F}\zP$. By definition of the product topology, $U$ is of the form
    \begin{align*}
        U = \prod_{x\in F} U_x \times \prod_{x \in F'} V_x \times \prod_{x\in X \setminus (F\cup F')} \zP
    \end{align*}
 where $F$ and $F'$ are disjoint finite subsets of $X$, each $U_x$ is an open subset of $\qP$, and each $V_x$ is an open subset of $\zP$. Consider the element $\xi_0 = \sum_{x\in F\cup F'} \xi(x)\delta_x$, then $\xi_0 \in U \cap c_0(X,\qP)$. This concludes the claim that $c_0(X,\qP)$ is dense in $\qP(X)$ with respect to the $\tau$-topology.    
\end{proof}

By the paragraph preceding Proposition \ref{prop:null-sequence-dense}, we have the following:

\begin{corollary}
    For a non-empty set \(X\), we have \[\BUN(\qP(X)) \cong \setgiven{T \colon c_0(X,\qP) \to c_0(X,\qP)}{T \text{ is } \zP\text{-linear, adjointable, } \norm{T} \leq 1}.\]
    Here adjointable means that there exists a unique $\zP$-linear map $T^{\ast} \colon c_0(X,\zP) \to c_0(X,\zP)$ with $\langle T \xi, \eta \rangle = \xi, T^{\ast}\eta \rangle$ for every $\xi, \eta \in c_0(X,\zP)$, where $\langle \cdot, \cdot \rangle$ is the canonical bilinear pairing of $c_0(X,\zP)$ mentioned above.
\end{corollary}

    The main reason for working with operators on the \(\zP\)-module \(\qP(X)\) rather than on the more classically studied Banach \(\qP\)-spaces \(c_0(X,\qP)\) is that the latter is not self-dual for the canonical pairing.   

\subsection{Matrices of Operators}

In this section, we extend the results of \cite{claussnitzer-thesis} to study matrix representations of operators \(T \in \BB(\Q_p(X))\) in the non-separable case. Studying the matrix entries of an operator will give us several results that will help us with calculations involving $p$-adic operators in the next sections. We first have the following simple lemma:

\begin{lemma}\label{scalar} Let $T \in \BB(\qP(X))$. Let $\xi \in \qP(X)$ and $\lambda \in \qP$ such that $\lambda \xi \in \qP(X)$, then
\begin{align*}
    T(\lambda \xi) = \lambda T(\xi)
\end{align*}  
\end{lemma}
\begin{proof}
Clear.
\end{proof}

\begin{definition}[Matrix] Let $T \in \BB(\qP(X))$. For $x, y \in X$ we define the entries of $T$ by 
\begin{align*}
    T_{x, y} = T(\delta_y)(x).
\end{align*}
\end{definition}

We first record the expected result that the adjoint of the matrix of an operator is really the transpose:

\begin{proposition}[The adjoint is the transpose]\label{adjointtranspose} Let $T \in \BB(\qP(X))$. We have 
\begin{align*}
    (T^*)_{x, y} = T_{y, x}
\end{align*}
\end{proposition}
\begin{proof}
First note that $T^{\ast}$ exists and is an element of $\BB(\qP(X))$ by Proposition \ref{lcaadjoint}. For every $\lambda \in \qP$, for every $x, y \in X$ using Lemma~\ref{scalar}
\begin{align*}
    \langle T(\lambda \delta_x), \delta_y \rangle = \langle \lambda T( \delta_x), \delta_y \rangle = \sum_{z \in X} \lambda T( \delta_x)(z) \delta_y(z) + \zP = \lambda T_{y, x} + \zP.
\end{align*}
On the other hand
\begin{align*}
    \langle T(\lambda \delta_x), \delta_y \rangle = \langle \lambda \delta_x, T^*(\delta_y) \rangle = \sum_{z \in X} \lambda \delta_x (z) T^*(\delta_y)  (z) + \zP = \lambda (T^*)_{x, y} + \zP
\end{align*}
Therefore for every $n \in \N$, choosing $\lambda = p^{-n}$, we get \(p^{-n} (T_{y, x} - (T^*)_{x, y}) \in \zP\), so that the difference 
\(T_{y, x} - (T^*)_{x, y} \in p^n \zP \) for all \(n \in \N\), concluding the proof.
\end{proof}

The second interesting fact about the entries of an operator is that if we fix the row and move through the column, the entries tend toward $0$, and vice versa. This is a very important fact that will imply various convergence results in the  next sections. This is the $\BUN(\qP(X))$ version of \cite{claussnitzer-thesis}*{Theorem 2.1.2}.

\begin{proposition}\label{limitentries} Let $T \in \BB(\qP(X))$. Then 
\begin{align*}
    \lim_{x \to \infty} T_{x, y} = 0 \\
    \lim_{y \to \infty} T_{x, y} = 0
\end{align*}
Conversely, if $(T_{x,y})_{x, y}$ is a square matrix indexed by $X$ with coefficients in $\zP$ satisfying the two above limits, then, there is $T \in \BB(\qP(X))$ whose entries are the entries of this matrix.
\end{proposition}
\begin{proof} To see that the entries of an operator satisfy the limits of the proposition, just use \ref{entries-goes-zero} and \ref{adjointtranspose}. Now, given a matrix $(T_{x, y})_{x, y}$ satisfying that
\begin{align*}
    \lim_{x \to \infty} T_{x, y} = 0 \\
    \lim_{y \to \infty} T_{x, y} = 0
\end{align*}
For every $\xi \in \qP(X)$ define
\begin{align*}
    T(\xi)(x) := \sum_{y \in X} T_{x, y} \xi(y)
\end{align*}
One easily checks that $T$ is a well-defined operator with the desired entries.
\end{proof}

\begin{lemma}[Entries of a product]\label{entriesproduct} Let $S$, $T \in \BUN(\qP(X))$ and $x, y \in X$, then
\begin{align*}
    (ST)_{x, y} = \sum_{z \in X} S_{x, z} T_{z, y}
\end{align*}
\end{lemma}
\begin{proof}
As $\lim_{x \to \infty} T_{x, y} = 0$ we can write
\begin{align*}
    T(\delta_y) = \sum_{z \in X} T_{z, y} \delta_z 
\end{align*}
Applying $S$ and evaluating on $x$
\begin{align*}
    S(T(\delta_y))(x) &= \sum_{z \in X} T_{z, y} S(\delta_z)(x) \\
    &= \sum_{z \in X} T_{z, y} S_{x, z},
\end{align*}
concluding the proof.
\end{proof}

\begin{proposition}\label{unique-entries} Let $T \in \BB(\qP(X))$ be a operator. For every $\xi \in \qP(X)$ and every $x \in X$, we have
\begin{align*}
    T(\xi)(x) = \sum_{y \in X} T_{x, y} \xi(y)
\end{align*}
In particular, two operators with the same entries are equal.
\end{proposition}
\begin{proof} Define the operator $\tilde{T}$ by the formula
\begin{align*}
    \tilde{T}(\xi)(x) = \sum_{y \in X} T_{x, y} \xi(y)
\end{align*}
Note that $\tilde{T}_{x, y} = T_{x, y}$, so if $\xi$ has finite support, then $\tilde{T}(\xi) = T(\xi)$.

Also note that if $\|\xi\| \leq 1$ and $\|\eta\| \leq 1$, then $\langle \xi, \eta \rangle = 0 + \zP$.

Given $\zeta \in \qP(X)$, we can decompose it as $\zeta=\zeta_0+\zeta_1$ with $\zeta_1$ finitely supported, where
\begin{align*}
    \zeta_0(x) := \begin{cases}
        \zeta(x) \text{, if }|\zeta(x)| \leq 1 \\
        0 \text{, otherwise.}
    \end{cases}
\end{align*}
and
\begin{align*}
    \zeta_1(x) := \begin{cases}
        \zeta(x) \text{, if }|\zeta(x)| > 1 \\
        0 \text{, otherwise.}
    \end{cases}
\end{align*}
Now we are ready to complete the proof. For every $\xi$ and $\eta$ in $\qP(X)$ one has
\begin{align*}
    \langle T\eta, \xi \rangle &= \langle T(\eta_0 + \eta_1), \xi_0 + \xi_1 \rangle \\
    &= \langle T\eta_1, \xi_1 \rangle + \langle T\eta_1, \xi_0 \rangle + \langle T\eta_0, \xi_1 \rangle + \langle T\eta_0, \xi_0 \rangle \\
    &= \langle T\eta_1, \xi_1 \rangle + \langle T\eta_1, \xi_0 \rangle + \langle T\eta_0, \xi_1 \rangle \\
    &= \langle \tilde{T}\eta_1, \xi_1 \rangle + \langle \tilde{T}\eta_1, \xi_0 \rangle + \langle \eta_0, T^*\xi_1 \rangle \\
    &= \langle \tilde{T}\eta_1, \xi \rangle + \langle \eta_0, \tilde{T}^*\xi_1 \rangle \\
    &= \langle \tilde{T}\eta_1, \xi \rangle + \langle \tilde{T}\eta_0, \xi_1 \rangle + \langle \tilde{T}\eta_0, \xi_0 \rangle \\
    &= \langle \tilde{T}\eta, \xi \rangle
\end{align*}
therefore $T = \tilde{T}$.
\end{proof}

\begin{corollary}\label{cor:sup-norm}
    For every $T\in \BB(\qP(X))$, we have 
    \begin{enumerate}
        \item $\|T\|=\max_{x,y\in X}|T_{x,y}|$;
        \item $\|T^*\|=\|T\|$, that is, the involution is isometric on $\BUN(\qP(X))$.
    \end{enumerate}
\end{corollary}
\begin{proof}
    It is clear that $|T_{x,y}|=|T(\delta_y)(x)|\leq \|T\|$ for all $x,y\in X$, and the reverse inequality follows from Proposition~\ref{unique-entries}.  This proves the first item.
    And the second item follows from the first because the matrix of $T^*$ is just the transposed matrix of $T$.
\end{proof}

At this point a natural question appears, whether the \(C^*\)-identity:
$$\|a^*a\|=\|a\|^2$$
holds for an operator $a$ on a $p$-adic Hilbert space. By Corollary~\ref{cor:sup-norm} the involution is isometric so that we always have the inequality $\|a^*a\|\leq \|a\|^2$. But the equality might, indeed, fail even in the case of finite $2\times 2$-matrices:

\begin{example}\label{exa:counter-example-Cst-axiom}
Set $p = 2$ and consider $\qQ_2$ the $2$-adic numbers. Let
\begin{align*}
    a = \begin{pmatrix}
       1 & 1 \\
        1 & 1
    \end{pmatrix} \in \mathbb{M}_2(\zP)=\BUN(\Qp(X)),
\end{align*}
where $X=\{1,2\}$ is the set with two elements. Then $a$ is a self-adjoint matrix as it is symmetric. Note that $a^2 = 2a$. Computing the $2$-adic norm, we get $1 = \|a\|^2 \neq \|a^2\| = \frac{1}{2}$. This behaviour is very different from what we are used in the classical theory of \cstar{}algebras and shows that, in particular, the \(C^*\)-identity $\|a^*a\|=\|a\|^2$ does \emph{not} hold, even for self-adjoint elements $a=a^*$. 

Even worse, let $p$ be a prime. We may have a non-zero operator $a$ with $a^*a=0$. To see an example of this type, consider a matrix of the form
\begin{align*}
a = \begin{pmatrix}
       \alpha & \beta \\
        -\beta & \alpha
    \end{pmatrix} \in \mathbb{M}_2(\zP)=\BUN(\Qp(X)),
\end{align*}
with $\alpha,\beta\in \Zp$ satisfying $\alpha^2+\beta^2=0$. This equation has non-zero solutions: one might take here, for instance, $\beta=1$ and an appropriate prime $p \equiv 1 \ \textrm{mod} \ 4$ so that $-1$ admits a square root in $\Zp$. This can be done by Hensel's Lemma (see \cite{neukirch2013algebraic}*{Lemma 4.6}).
Then
\begin{align*}
    a^*a=\begin{pmatrix}
       \alpha^2+\beta^2 & 0 \\
        0 & \alpha^2+\beta^2
    \end{pmatrix} \in \mathbb{M}_2(\zP)=\BUN(\Qp(X)).
\end{align*}
so that $a^*a=0$ although $a\not=0$. We would like to thank Andreas Thom for bringing this example to our attention.
\end{example}

Before we end this section, we want to shortly introduce unitary operators:

\begin{definition}
    By a \emph{unitary operator} on a $p$-adic Hilbert space $\qP(X)$ we mean an operator $U\in \BUN(\Qp(X))$ for some set $X$ satisfying the relation $U^*U=UU^*=1$, that is, $U$ is invertible with $U^{-1}=U^*$. The group of all unitary operators on $\Qp(X)$ will be denoted by $\UU(\Qp(X))$.
\end{definition}

Therefore $\UU(\Qp(X))$ is a subgroup of the multiplicative semigroup $\BUN(\Qp(X))$.
Notice that unitary operators automatically have norm $1$, that is, $\UU(\Qp(X))$ is contained in the unit sphere of $\BB(\Qp(X))$. Indeed, we have $1=\|U^*U\|\leq \|U\|^2\leq 1$ because $\|U\|\leq 1$ by assumption, so that $\|U\|=1$. 

We can also define projections and, more generally, partial isometries: a \emph{projection} is a self-adjoint idempotent $p\in \BUN(\Qp(X))$. A \emph{partial isometry} is an operator $T\in \BUN(\Qp(X))$ satisfying $TT^*T=T$. It is an isometry if $T^*T=1$, and a co-isometry if $TT^*=1$. All these operators satisfy $\|T\|=1$ unless they are zero. Indeed, the equation $TT^*T=T$ implies $\|T\|=\|TT^*T\|\leq \|T\|^3$. If $T\not=0$, then we get $1\leq \|T\|^2$ and since $\|T\|\leq 1$ by assumption, this forces $\|T\|=1$. \vskip 1pc

\begin{example}
    One simple way to produce partial isometries is via partial bijections of $X$: if $f\colon U\to V$ is a bijection between subsets $U,V\sbe X$, then we can define $T_f(\delta_x):=\delta_{f(x)}$ if $x\in U$ and zero otherwise. This extends to a partial isometry $T_f\in \BUN(\Qp(X))$ given by the formula $T_f(\xi)(x)=\xi(f^{-1}(x))$ if $x\in V$ and zero otherwise. We have $T_f^*=T_{f^{-1}}$, where $f^{-1}\colon V\to U$ is the inverse partial bijection. We also have $T_f T_g=T_{f\circ g}$, where $f\circ g$ denotes the composition of partial bijections, that is, the composition on the largest domain where it makes sense. In particular, if $f\colon X\to X$ is a bijection, then $T_f$ is a unitary operator.
\end{example}

\subsection{Continuous functional calculus}

In this subsection, we briefly describe a notion of functional calculus in our setting. In the case where \(X\) is countable, this already appears in \cite{thom-claussnitzer-article}. For \(x \in \zP\) and \(n \in \N\), consider the binomial coefficient \[\binom{x}{n} \defeq \frac{x (x-1) \cdots (x-(n-1))}{n!},\] which lies in \(\zP\). Then Mahler's Theorem states that:

\begin{theorem}[\cite{thom-claussnitzer-article}*{Theorem 2.23}]
    Every element \(f \in C(\zP,\zP)\) has a unique representation of the form \[f(x) = \sum_{n=0}^\infty c_n(f) \binom{x}{n},\] where \(c_n(f) \in c_0(\N, \zP)\). Conversely, any such series converges uniformly to a continuous function, and we have \(\norm{f} = \max \abs{c_n(f)}\). Consequently, we have an explicit isometric isomorphism \[C(\zP, \zP) \to c_0(\N, \zP), f \mapsto (c_n)_{n \in \N}.\]
\end{theorem}

For an arbitrary set \(X\), an operator \(A \in \BUN(\qP(X))\) is called a \textit{normal contraction} if for all \(n \in \N\), \[\norm{A(A-1)\dotsc (A-(n-1))} \leq \abs{n!}_p.\] Examples of normal contractions include contractive diagonal operators on \(\qP(X)\). 

\begin{theorem}[Continuous functional calculus]
    Let \(A \in \BUN(\qP(X))\) be a normal contraction. Then there is a natural contractive homomorphism of Banach \(\zP\)-algebras 
    \[C(\zP,\zP) \to \BUN(\qP(X)), \mathrm{id}_{\zP} \mapsto A.\]
\end{theorem}

\begin{proof}
    The proof of \cite{thom-claussnitzer-article}*{Theorem 2.25} remains unchanged even if we allow \(X\) to be uncountable. 
\end{proof}

\section{$p$-adic Operator Algebras}

In the complex setting, the GNS construction shows that an abstract \(C^*\)-algebra can be concretely realised as a closed involutive subalgebra of bounded operators on a Hilbert space. This perspective is used to define more general classes of operator algebras, such as \(L^q\)-operator algebras for $q\in [1,\infty)$ a real number, which are closed subalgebras of bounded operators on Banach spaces of the form \(L^q(X,\mu)\), where \((X,\mu)\) is a measure space. Our strategy for defining what would constitute a ``\cstar{}algebra'' in the $p$-adic world draws inspiration from these alternative, representation-theoretic approaches to operator algebras. 

\begin{definition} A \textit{$p$-adic operator algebra} is a Banach $\zP$-algebra $A$ together with a $\Zp$-linear involution ${}^*\colon A \to A$,  such that there exists an isometric $*$-algebra isomorphism from $A$ onto a closed $*$-subalgebra of $\BUN(\qP(X))$ for some set $X$. 
\end{definition}

We define a morphism \(\phi \colon A \to B\) of \(p\)-adic operator algebras to be a contracting \(*\)-homomorphism between the underlying Banach algebras \(A\) and \(B\). That is, the morphisms are those of the category \(\mathsf{Ban}_{\zP}^{\leq 1}\) of contracting Banach \(\zP\)-modules. The \(p\)-adic operator algebras and morphisms just defined form a category that we denote by \(\OA\). 

We equip the algebra \(\zP\) with the trivial (that is, identity) involution. A key difference to the complex case is the lack of an analogue of the complex conjugation.
We briefly remark about this in Section~\ref{subsec:involution}.
In particular notice that the involution on $A$ induces an isometric $*$-isomorphism $A\cong A^{\mathrm{op}}$, the opposite algebra of $A$, meaning that $A^{\mathrm{op}}$ carries the same Banach $\Zp$-module structure, only the multiplication is reversed, so all our algebras are ``self-opposite''; this is not true for (complex) \cstar{}algebras, see for instance \cite{Phillips:opposite}.  

We now present several examples of $p$-adic operator algebras.

\begin{example}[Matrix algebras]\label{ex:matrix}
    Let \(X\) be a finite set of cardinality \(n\). Then \(\BB(\qP(X))\) and the \(p\)-adic operator algebra \(\BUN(\qP(X))\) can be identified with the matrix algebras \(\Mat_n(\qP)\) and \(\Mat_n(\zP)\), respectively. Notice that between these, only $\Mat_n(\Zp)$ is a $p$-adic operator algebra in our sense.
\end{example}

\begin{example}[Algebras of continuous functions]\label{exa:commutative-algebras}
Let $X$ be any topological space. Then the space $\contz(X,\Zp)$ of all continuous functions $f\colon X\to \Zp$ vanishing at infinity (meaning that $\{x\in X: |f(x)|\geq \epsilon\}$ is compact (not necessarily Hausdorff) for all $\epsilon>0$) is a \emph{commutative} $p$-adic operator algebra with respect the pointwise product of functions, the trivial involution ($f^*=f$ for all $f\in \contz(X,\Zp)$) and the supremum norm $\|f\|_\infty:=\sup_{x\in X}|f(x)|$, which coincides with the canonical \(p\)-adic norm, using that the supremum norm is discretely valued, so that \(\norm{f}_\infty = \abs{f(x_0)}\) for some \(x_0 \in X\). To represent it on a $p$-adic Hilbert space, just consider the usual representation $M\colon \contz(X,\Zp)\to \BUN(\Qp(X))$ by multiplication operators: $M_f(\xi)(x):=f(x)\xi(x)$. It is easy to see that this is an isometric representation.

A similar reasoning shows that the algebra $\contb(X,\Zp)$ of all (necessarily bounded) continuous functions $X\to \Zp$ is a $p$-adic operator algebra with respect to the same pointwise operations, the trivial involution, and the supremum norm: the representation by multiplication operator makes sense and is still isometric on $\contb(X,\Zp)$. 

If $X$ carries the discrete topology, we use the notations $c_0(X,\Zp)=C_0(X,\Zp)$ and $\ell^\infty(X,\Zp)=\contb(X,\Zp)$.
\end{example}

\begin{remark}
    As in the case of \cstar{}algebras, we should restrict to nice spaces $X$ in order to get enough $\contz$-functions and therefore interesting algebras of the form $\contz(X,\Zp)$. For this reason, we usually only consider locally compact Hausdorff spaces here.
    Moreover, since $\Zp$ is totally disconnected as a topological space, the algebras $\contz(X,\Zp)$ will only ``see'' a certain quotient of $X$ related to the connected components of $X$: indeed, notice that a continuous function $f\colon X\to \Zp$ is constant on the connected components of $X$. So for instance, for a connected space like $X=\R$ or the circle $X=\T$, we get the trivial algebras $\contz(\R,\Zp)=0$ and $\contz(\T,\Zp)=\Zp$.
\end{remark}

    In \cite{vanrooij}*{Chapter 6}, the author refers to Banach algebras of the form 
    \(C_0(X,K)\), where \(X\) is a totally disconnected space and $K$ a non-archimedean complete field as analogues of commutative \(C^*\)-algebras, which they call \(C\)-algebras. For $K=\Qp$, he calls these \emph{\(p\)-adic $C$-algebras}. It is shown in \cite{vanrooij}*{Corollary 6.8} that these are exactly the Banach algebras for which the Gelfand transform induces an isometric isomorphism \(A \cong C_0(\mathsf{Sp}(A), \Qp)\), where \(\mathsf{Sp}\) is the Gelfand spectrum of \(A\). 
    One could also use the same strategy in order to describe Banach algebras of the form $\contz(X,\Zp)$ as these are exactly the unit balls of $p$-adic $C$-algebras.
    However, this definition of a nonarchimedean commutative \(C^*\)-algebra is rather ad-hoc, as it does not relate the categories of topological spaces with any subcategory of Banach algebras. A more elegant approach appears in recent work (\cite{bambozzi2023homotopy}), where for a totally disconnected compact Hausdorff space \(X\), the authors define a functor \[\Gamma^X \colon \mathsf{CH}_{/X} \to \mathsf{C}_X^*\] from the category of compact Hausdorff spaces over \(X\) to the subcategory of Banach \(C(X,\zP)\)-algebras generated by algebras of the form \(C(Y,\zP)\). The functor is defined by assigning to a space \(Y \in \mathsf{CH}_{/X}\) the commutative algebra \(C(Y,\zP)\) with the induced action \(C(X,\zP) \to C(Y,\zP)\). By \cite{bambozzi2023homotopy}*{Lemma 3.8}, this functor is part of an equivalence between the category \(\mathsf{CH}_{/X}\) and the essential image of the inclusion \(\mathsf{C}_X^* \to \mathsf{Alg}_{\zP}(C(X,\zP))\). Example~\ref{exa:commutative-algebras} shows that our \(p\)-adic operator algebras include these examples of commutative \(C^*\)-algebras in the sense of \cite{bambozzi2023homotopy}.

We now describe \(p\)-adic convolution algebras associated to groups.

\begin{example}[Group algebras]\label{ex:group-ring} Let $G$ be a (discrete) group and consider the Banach $\Zp$-module $c_0(G,\Zp)$ of functions $G\to\Zp$ vanishing at infinity, endowed with supremum norm, which also coincides with the $p$-adic norm.
It is an involutive algebra with respect to the usual convolution product and involution defined by 
\[(\phi \ast \psi)(h) = \sum_{g \in G} \phi(g) \psi(g^{-1}h), \quad \phi^*(g) = \phi(g^{-1}).\] 
Note that if $\phi \in c_0(G, \zP)$, then its support $\supp(\phi)=\{g\in G:\phi(g) \neq 0\}$ is countable as the norm on $\Zp$ is discretely valued (it only assumes the values $p^{-n}$ with $n\in \N$ or zero). This ensures that the convolution product is well defined in $c_0(G,\Zp)$. Hence we may view $c_0(G,\Zp)$ as the completion of the group ring $\Zp[G]$, viewed as the $*$-subalgebra generated by delta functions $\delta_g\in c_0(G,\Zp)$.

Notice that the elements $\delta_g$ are unitary with $\delta_g*\delta_h=\delta_{gh}$ and $\delta_g^*=\delta_{g^{-1}}$. And every element $\phi \in c_0(G, \zP)$ can be written as $\phi = \sum_{g \in G} \phi(g) \delta_g$.

We now represent the convolution algebra \(c_0(G,\zP)\) on the \(p\)-adic Hilbert space $\qP(G)$, considering the \emph{left regular representation} 
\[\lambda \colon G \to \BUN(\qP(G)), \quad \lambda_g(\xi)(h) = \xi(g^{-1}h),\] 
for \(g\), \(h \in G\) and \(\xi \in \qP(G)\). In other words, $\lambda_g$ is implemented by the multiplication on $G$, which is a bijection. We have that $\lambda_g$ is adjointable with 
\(\lambda_g^*=\lambda_{g^{-1}}\); thus it is unitary with norm $1$. By the universal property of the group ring $\Zp$, the left regular representation has a unique extension to a $*$-homomorphism, still denoted by
\[\lambda\colon \zP[G] \to \BUN(\qP(G)), \quad \lambda(\delta_g) = \lambda_g.\] 
Indeed, the extension is given by the convolution $\lambda(\phi)\xi=\phi*\xi$, which still makes sense for $\phi\in \Zp[G]$ viewed as a function with finite support $G\to \Zp$ and $\xi\in \Qp(X)$.
Moreover, for $h\in G$ we get
\begin{align*}
    |\lambda(\phi)(\xi)(h)|= \left| \sum_{g \in G} \phi(g) \xi(g^{-1}h)  \right| \leq \max_{g \in G} |\phi(g) \xi(g^{-1}h)| \leq \max_{g \in G} \abs{\phi(g)} = \norm{\phi}.
\end{align*}
This shows \(\norm{\lambda_\phi} \leq \norm{\phi}\). On the other hand, choosing \(g_0 \in G\) with \(\abs{\phi(g_0)} = \norm{\phi}\), we obtain
\begin{align*}
    \|\phi\| = |\phi(g_0)| = |(\phi \ast \delta_{g_0^{-1}})(e)| \leq \sup_{\|\xi\| \leq 1} \max_{h \in G} |(\phi \ast \xi)(h)| = \| \lambda_{\phi}\|.
\end{align*}
This shows that $\lambda$ is an isometric, so it extends to an isometric $*$-homomorphism
\begin{align*}
    \lambda\colon c_0(G, \zP) \to \BUN (\qP(G)),\quad  \phi \mapsto \lambda_{\phi}
\end{align*}
and therefore $c_0(G, \zP)$ is a $p$-adic operator algebra.
Notice that $\lambda_\phi(\xi)=\phi*\xi$, as the convolution product still makes sense for $\phi\in c_0(G,\Zp)$ and $\xi\in \Qp(X)$.

In what follows we shall denote by $\OO_p(G)\sbe\BUN(\Qp(G))$ the image of $\lambda\colon c_0(G,\Zp)\to \BUN(\Qp(X))$, and call it the $p$-adic operator algebra of the group $G$. As $\lambda$ is an isometric $*$-isomorphism, we usually also use it to identify $\OO_p(G)\cong c_0(G,\Zp)$ as Banach $*$-algebras over $\Zp$. 
\end{example}

\begin{example}
Let us relate Example~\ref{exa:counter-example-Cst-axiom} to $p$-adic group algebras: consider the group $G=\{1,g\}\cong\Z/2$ with two elements. Then $\Qp(G)\cong\Qp^2$, and using the canonical identification $\BUN(\Qp(G))\cong \Mat_2(\Zp)$ given by the (ordered) basis $(\delta_1,\delta_g)$ of $\Qp(G)$, we get the operators
$$\lambda_1=\begin{pmatrix}
        1 & 0 \\
        0 & 1
    \end{pmatrix},\quad \lambda_g=\begin{pmatrix}
        0 & 1 \\
        1 & 0
    \end{pmatrix}.$$
For $a$ as in Example~\ref{exa:counter-example-Cst-axiom}, but for arbitrary $p$, we have that $a=\lambda_1+\lambda_g$ and, therefore, the $\Zp$-subalgebra $\OO_p(a,1)$ of $\BUN(\Qp(G))$ generated by $a$ and the identity matrix $1$ is precisely the $p$-adic group algebra $\Zp[G]\cong\OO_p(G)$. This is a commutative $p$-adic operator algebra. The involution is trivial (all elements are self-adjoint). In particular this shows that this algebra cannot be represented as a commutative algebra of functions $\contz(X,\Z_p)$ with the supremum norm as in Example~\ref{exa:commutative-algebras}, because in such algebras we do have the relation
$\|f^n\|=\|f\|^n$ for all $f\in \contz(G,\Zp)$, $n\in \N$. This example therefore shows that there is no obvious ``Gelfand representation'' for commutative $p$-adic operator algebras.
\end{example}

We now define the $p$-adic operator algebra of an étale groupoid. Recall that a groupoid is a small category with inverses. We shall usually use the notation $\GG$ for a groupoid, which we shall use to also abusively denote its set of arrows. We write $\GG^0$ for its objects and usually view this as a subset of $\GG$ via the unit map. We also use $\s,\rg\colon \GG\to \GG^0$ as notations for the source and range maps, and write $\GG^2:=\{(g,h)\in \GG\times\GG: \s(g)=\rg(h)\}$ for the set of composable pairs. We are mostly interested in topological étale groupoids. This means that $\GG$ carries a topology in which all structure maps (including multiplication and inversion) are continuous. It is \emph{étale} if $\GG^0$ is a locally compact Hausdorff space and the source and (hence also) range maps are local homeomorphisms. We refer the reader to \cite{EXEL01} for the basic theory of étale groupoids and their \cstar{}algebras.

Let $\GG$ be an étale groupoid. These groupoids are better described in terms of their bisections. Recall that an open subset $U \subseteq \GG$ is said to be a \emph{bisection} if $\s|_U$ and $\rg|_U$ are injective.
Those subsets are homeomorphic (via either $\s$ or $\rg$) to an open subspace of $\GG^0$, hence are automatically locally compact and Hausdorff. Let $\Bis(\GG)$ be the set of all bisections of $\GG$. This is canonically an inverse semigroup with respect to the operations:
    \begin{align*}
        U^{-1} &= \set{u^{-1} \mid u \in U} \\
        UV &= \set{uv \mid u \in U, v\in V, \rg(v) = \s(u)},
    \end{align*}
see \cite{EXEL01} for further details. Since $\GG$ is étale, $\Bis(\GG)$ forms a basis for its topology.

\begin{example}[Étale groupoid algebras]\label{ex:etale-groupoid} Let $\GG$ be an étale groupoid. Given $U \in \Bis(\GG)$, let $C_c(U,\zP)$ be the space of continuous maps $\phi\colon U \to \zP$ with compact support. Each $C_c(U,\zP)$ is contained in the space of all functions $\GG \to \zP$ simply extending the function sending every point outside of $U$ to $0$. Let $\Cc(\GG,\zP)$ be the linear span of the collection of all spaces $C_c(U,\zP)$ inside of the space of functions. If $\GG$ is Hausdorff (which is mostly the only case we shall consider in this paper), this is the same as the space $\contc(\GG,\Zp)$ of compactly supported continuous functions $\GG\to\Zp$. But if $\GG$ is not Hausdorff, the space $\Cc(\GG,\Zp)$ will generally contain non-continuous functions $\GG\to \Zp$.

Given $\phi, \psi \in \Cc(\GG,\zP)$, we define their \emph{convolution product}:
\begin{align}\label{eq:convolution-product}
    (\phi \ast \psi)(h) = \sum_{\substack{g \in \GG \\ \rg(g) = \rg(h)}} \phi(g)\psi(g^{-1}h)
\end{align}
This is a well-defined finite sum by the same reason it does in the archimedean case, see \cite{EXEL01}*{ Proposition~3.11}. Moreover, if $\phi\in \contc(U,\Zp)$ and $\psi\in \contc(V,\Zp)$ for bisections $U,V\sbe \GG$, then the same formula as in \cite{EXEL01}*{Equation~(3.11.1)} applies:
$$(\phi*\psi)(k)=\phi(g)\psi(h)$$
if $k=gh$ with $g\in U$ and $h\in V$, and it is zero otherwise. Also, the same arguments as in the complex case show that $\Cc(\GG,\Zp)$ is an associative $\Zp$-algebra with the above defined convolution product.
Moreover, it is also a $*$-algebra with respect the involution
$$\phi^*(g):=\phi(g^{-1}).$$

We conclude that $\Cc(\GG,\zP)$ is an involutive $\zP$-algebra. 
Endowing it with the supremum norm: $\|\phi\|_\infty:=\sup_{g\in \GG}|\phi(g)|$, it follows from the ultra-metric inequality that $\Cc(\GG,\Zp)$ is a norm $*$-algebra over $\Zp$, and therefore its completion
$$\Cz(\GG,\Zp):=\overline{\Cc(\GG,\Zp)}^{\|\cdot\|_\infty}$$
is a Banach $*$-algebra over $\Zp$.  Notice that the completion $\Cz(\GG,\Zp)$ can be realized as the closure of $\Cc(\GG,\Zp)$ in the Banach $\Zp$-module $\ell^\infty(\GG,\Zp)$, so that we can view it as a concrete space of functions $\GG\to \Zp$. Moreover, all functions $\phi\in \Cz(\GG,\Zp)$ vanish at infinity, although they might be not continuous. It is easy that the norm \(\norm{\cdot}_\infty\) on \(\Cz(\GG,\zP)\) again coincides with the canonical norm because as the range $|\cdot|_p$ is discrete, the norm is attained at some point. If $\GG$ is Hausdorff, then $\Cz(\GG,\Zp)=\contz(\GG,\Zp)$, the ordinary space of $\contz$-functions $\GG\to \Zp$. 

Next we show that $\Cz(\GG,\Zp)$ is a $p$-adic operator algebra, by representing it isometrically on a $p$-adic Hilbert space. Let $\phi \in \Cc(\GG,\zP)$. Then the convolution product $\phi \ast \xi$ is still a well-defined finite sum for any other function $\xi \colon \GG \to \qP$ (not necessarily in $\Cc(\GG,\Zp)$) and the result is a new function $\phi\ast\xi\colon \GG\to \Qp$. Moreover, if $\xi\in \Qp(X)$, then so is $\phi\ast\xi\in \Qp(X)$. Hence, for every $\phi \in \Cc(\GG,\zP)$ we obtain the map 
\begin{align*}
    \lambda_{\phi}\colon \qP(\GG) &\to \qP(\GG) \\
    \xi &\mapsto \phi \ast \xi.
\end{align*}
As the convolution product is bilinear, $\lambda_{\phi}$ is a $\zP$-module morphism . 
Next we show that \(\lambda_{\phi}\) is adjointable by computing its adjoint $\lambda_{\phi}^* = \lambda_{\phi^*}$. To this end, let $\phi \in \Cc(\GG)$ and $\xi, \eta \in \qP(\GG)$. Then
\begin{align*}
    \langle \lambda_{\phi} (\xi), \eta \rangle &= \sum_{l \in \GG} (\phi \ast \xi)(l) \eta(l) + \zP \\
    &= \sum_{l \in \GG} \sum_{\substack{k \in \GG \\ \rg(k) = \rg(l)}} \phi(k) \xi (k^{-1}l) \eta(l) + \zP \\
\end{align*}
Note that this is a finite sum as almost all terms are in $\zP$. Changing variables, we then have

\begin{align*}
    \langle \lambda_{\phi} (\xi), \eta \rangle &= \sum_{g \in \GG} \sum_{\substack{h \in G \\ \rg(h) = \rg(g) }} \phi(h^{-1}) \xi (g) \eta(h^{-1}g) + \zP \\
    &= \sum_{g \in G} \xi(g) \sum_{\substack{h \in G \\ \rg(h) = \rg(g) }} \phi^*(h) \eta(h^{-1}g) + \zP \\
    &= \sum_{g \in G} \xi(g) (\phi^* \ast \eta)(g) + \zP \\ 
    &= \langle \xi, \lambda_{\phi^*} (\eta) \rangle,
\end{align*}
proving the claim.

It is easy to see that the assignment \(\phi \mapsto \lambda_{\phi}\) is a \(\zP\)-algebra homomorphism. To see that it is an isometry, we first observe that
\begin{align*}
    \|\lambda_\phi\| &= \sup_{\|\xi\| \leq 1} \|\phi \ast \xi\| \\
    &= \sup_{\|\xi\| \leq 1} \max_{h \in \GG} |(\phi \ast \xi) (h)| \leq \norm{\phi}.
\end{align*}
On the other hand, by the same argument used in \ref{ex:group-ring}, the maximum is attained at some $g_0 \in \GG$. Now consider the function
\begin{align*}
    \xi_0 (\gamma) = \begin{cases}
        1 \text{, if } \gamma = \id_{\s(g_0)} \\
        0 \text{, otherwise.}
    \end{cases}
\end{align*}
Then
\begin{align*}
    (\phi \ast \xi_0) (g_0) = \sum_{\rg(g) = \rg(g_0)} \phi(g) \xi_0 (g^{-1}g_0) = \phi(g_0)
\end{align*}
concluding that
\begin{align*}
    \|\phi\| = |\phi(g_0)| = |(\phi \ast \xi_0) (g_0)| \leq \sup_{\|\xi\| \leq 1} \max_{h \in \GG} |\phi \ast \xi (h)| = \|\lambda_\phi\|
\end{align*}
We conclude that \(\lambda\) is an isometric $*$-homomorphism $\lambda\colon \Cc(\GG,\Zp)\to \BUN(\Qp(\GG))$.

Finally we define $\OO_p(\GG)$ to be the $p$-adic operator algebra generated by the image of $\lambda$. 
Since $\lambda$ is isometric, it extends to an isometric $*$-isomorphism of Banach $*$-algebras $$\lambda\colon \Cz(\GG,\Zp)\congto \OO_p(\GG)\sbe \BUN(\Qp(\GG)).$$
Therefore $\Cz(\GG,\Zp)$ is a $p$-adic operator algebra. We shall usually use the above map to identify $\OO_p(\GG)$ with $\Cz(\GG,\Zp)$. In particular, in the Hausdorff case we have
$$\OO_p(\GG)\cong \contz(\GG,\Zp).$$
\end{example}

\begin{remark}
    For an étale groupoid $\GG$, one can build also the convolution $*$-algebra $\Cc(\GG,\C)$ and complete it with (in general) different \cstar{}norms, one of them is the largest \cstar{}norm, leading to the universal completion $C^*_{\max}(\GG)$, and the reduced norm $\|\cdot\|_r$ via the regular representation that is defined as above via a complex version of $\lambda$ acting on the (complex) Hilbert space $\ell^2(\GG)$, lead to the reduced groupoid \cstar{}algebra $C^*_r(\GG)$, concretely represented in the \cstar{}algebra of bounded operators $\BB(\ell^2(\GG))$.

    The $p$-adic operator algebra $\OO_p(\GG)$ should be viewed as a $p$-adic analogue of the \cstar{}algebra $C^*_r(\GG)$. Later we are going to also define universal $p$-adic enveloping operator algebras and will also have a $p$-adic version of $C^*_{\max}(\GG)$. However, as should be expected, in the $p$-adic setting, many completions collapse, and we are indeed going to show, at least in the Hausdorff totally disconnected case, that the universal $p$-adic groupoid operator algebra coincides with $\OO_p(\GG)$, showing that non-amenability phenomena cannot be detected by the $p$-adic operator norm. 
\end{remark}

Having access to groupoids, we can now build several important examples of $p$-adic operator algebras. We shall use this now to construct algebras of graphs.

\begin{example}[Graph algebras]\label{ex:Leavitt-groupoid}
Let $E=(s,r\colon E^1\to E^0)$ be a graph, consisting of a set of vertices $E^0$, the set $E^1$ of edges and two maps (source and range) $s,r$ from $E^1$ to $E^0$.
It is well known that to $E$ we can attach an étale groupoid $\GG_E$. Up to certain ``singular'' vertices, this is essentially the Deaconu-Renault groupoid (\cite{SimsWilliams}) of the shift map on the infinity path space $E^\infty$ of $E$. The groupoid $\GG_E$ is always locally compact Hausdorff and totally disconnected, that is, it is an \emph{ample} (Hausdorff) groupoid. We refer to \cite{Rigby} and references therein for more details on the construction of $\GG_E$. It is proved in \cite{Rigby}*{Theorem~3.14} that the \emph{Leavitt path algebra} $L_R(E)$ of $E$ over a commutative unital ring $R$ is isomorphic to the Steinberg algebra $\SS_R(\GG_E)=\contc(\GG_E,R)$ of $\GG_E$, where we endow $R$ with the discrete topology.  Of course, we are interested here in the ring $R=\Zp$ of $p$-adic integers, and we define
$$\OO_p(E):=\OO_p(\GG_E).$$ We shall return to this example and its generalisations in Proposition \ref{pro:iso-Steinberg-universal} and Corollary \ref{cor:Steinberg}. 
\end{example}

\begin{example}[Discrete groupoids] The discrete groupoids are particular examples of étale Hausdorff groupoids. In this case we have
\begin{align*}
    \OO_p(\GG) = c_0(\GG, \zP) =  \set{ \phi \colon \GG \to \zP \mid \lim_{g \to \infty} |\phi(g)| = 0},
\end{align*}
which is the \(p\)-adic completion of the algebraic groupoid algebra $\zP[\GG]$. The matrix algebras from Example~\ref{ex:matrix}, and more generally algebras of compact operators are special cases of discrete groupoid algebras, see Example~\ref{exa:compact-operators}. Of course, also the group algebras from Example~\ref{ex:group-ring} are special cases of these.
\end{example}

\subsection{A remark on involutions}\label{subsec:involution}

We now make a remark on involutions in the \(p\)-adic setting. As the reader may have noticed, the formulae for involutions on, for instance, the group convolution algebra differs from its complex analogue by a complex conjugation. Unfortunately, the only reasonable involution on \(\qP\) or \(\zP\) is the trivial involution. To understand this conceptually, recall that the Galois group \(\mathsf{Gal}(\C/\R)\) of the extension \(\C\) over \(\R\) comprises of the identity on \(\C\) and complex conjugation \(\C \to \C\), \(z \mapsto \bar{z}\). The (absolute) Galois group \(G(\qP)\) of \(\qP\) is however far more complicated - it is a profinite group, whose structure has recently been studied in \cite{absolute-Galois}. The representation theory of profinite groups on Banach \(\qP\)-vector spaces is very rich. In a future article, we will explore \(p\)-adic operator algebras equipped with an action of \(G(\qP)\), analogous to the theory of \(G\)-\(C^*\)-algebras. This additional structure will play the role of the complex conjugation in the \(p\)-adic setting. 

\section{The category of $p$-adic operator algebras}

In this section, we discuss the completeness properties of the category of \(p\)-adic operator algebras. We will show that this category has all limits and colimits, which enables us to carry out several universal constructions, including in particular the \textit{enveloping \(p\)-adic operator algebra} construction. 

One thing that is different from the classical theory of operator algebras and even the theory of real $C^*$-algebras and $L^q$-operator algebras for $q\in [1,\infty)$, is that all of our algebras are contained in the unit ball of \(\BB(\qP(X))\). This is a crucial fact that enables us to build limits and colimits.

\subsection{Limits in \(\OA\)}

We first take the case of limits. For what follows below, denote by \(\mathsf{Alg}^*(\mathsf{Ban}_{\zP}^{\leq 1})\) the category of involutive Banach \(\zP\)-algebras with contractive, involution preserving algebra homomorphisms as morphisms. 

\begin{lemma}\label{prop:OA-has-limits}
The category \(\OA\) has all limits. 
\end{lemma}

\begin{proof}
It is enough to show that it has products and equalisers. Suppose \((A_i)_{i \in I}\) is a family of \(p\)-adic operator algebras, then its product is the algebraic product \(A = \prod_{i \in I} A_i\) with the norm \(\norm{(a_i)} = \sup_{i \in I} \norm{a_i}\). The representing \(p\)-adic Hilbert space is  \(\qP(\bigsqcup_{i \in I} X_i)\), where \(A_i \to \BUN(\qP(X_i))\) is an isometric \(*\)-representation for each \(i\).

For equalisers, we first note that if $B$ is a closed subalgebra of a $p$-adic operator algebra $A$, then $B$ is also a $p$-adic operator algebra. Now for a parallel pair of morphisms \(f,g \colon A \rightrightarrows B\), the subalgebra \(E = \setgiven{a \in A}{f(a) = g(a)}\) is closed, and hence is a $p$-adic operator algebra. It satisfies the universal property of the equalisers as can be checked in the category of Banach algebras.

Hence the category of $p$-adic operator algebras has all limits, and since the category \(\OA\) has the same morphisms as \(\mathsf{Alg}^*(\mathsf{Ban}_{\zP}^{\leq 1})\), we have a fully faithful functor \[\OA \to \mathsf{Alg}^*(\mathsf{Ban}_{\zP}^{\leq 1}).\]
As the product and the equalizers in both categories are the same, their limits also coincide.    
\end{proof}

We now provide an example of a \(p\)-adic operator algebra that naturally arises as a limit. 

\begin{example}
A profinite group is a topological group $G$ that can be written as an inverse limit
\begin{align*}
    G = \lim_{\substack{\longleftarrow \\ 
    i \in I}} G_i
\end{align*}
of finite discrete groups $G_i$.  For each $i \in I$, we may form the $p$-adic operator algebra $c_0(G_i, \zP) = \zP[G_i]$ of Example~\ref{ex:group-ring}. It is easy to see that the given inverse system of groups yields an inverse system \[I^{\mathrm{op}} \to \OA, \quad i \mapsto \zP[G_i]\] of \(p\)-adic operator algebras. Taking limits, \begin{align*}
    \Lambda[G] \defeq \lim_{\substack{\longleftarrow \\ 
    i \in I}} \zP[G_i],
\end{align*} we get a number-theoretically important \(p\)-adic operator algebra called the \textit{Iwasawa algebra}. Note here that since the groups \(G_i\) are finite, \(\zP[G_i]\) is already \(p\)-adically complete. And since \(\zP\) is compact, the topological rings \(\zP[G_i]\) and hence \(\Lambda[G]\) are compact and \(p\)-adically complete. 
\end{example}

\subsection{Enveloping $p$-adic operator algebra and colimits in \(\OA\)}

Colimits of $p$-adic operator algebras are a bit more complicated. As mentioned in Section 2, the issue already arises at the level of Banach \(\zP\)-modules due to torsion phenomena: the algebraic cokernel of a morphism between Banach \(\zP\)-modules need not any longer remain torsionfree, and therefore cannot be a normed \(\zP\)-module. We consequently need a construction that kills all ``bad quotients''. To this effect, we introduce a new construction, called the enveloping $p$-adic operator algebra of a $*$-algebra $A$, which is in some sense the best approximation of \(A\) to a \(p\)-adic operator algebra. More precisely, we construct a functor
\begin{align*}
    (-)^u\colon \set{*\text{-algebras over }\zP} &\to \OA \\
    A &\mapsto A^u
\end{align*}
that is the left adjoint to the inclusion functor $\OA \to \set{*\text{-algebras over }\zP}$. So to build colimits of $p$-adic operator algebras, we take the colimit in the category of $*$-algebras over $\zP$, and then take its envelope. We now show that every $*$-algebra over $\zP$ admits an enveloping $p$-adic operator algebra.

\begin{definition}
    Let $A$ be a $*$-algebra over $\zP$.
    We call a function $\varphi\colon A\to \R_+$ a \textit{$p$\nb-adic operator algebra seminorm} if there is a \(*\)-homomorphism \(\pi \colon A \to \BUN(\qP(X))\) such that
    $$\varphi(a)=\|\pi(a)\|\quad\mbox{for all } a\in A.$$
\end{definition}

Notice that a $\varphi$ as above is, indeed, a seminorm in the usual sense, and it satisfies the ultra-metric property $\varphi(a+b)\leq \max\{\varphi(a),\varphi(b)\}$. Also, since the operator norm is discretely valued, so are all $p$-adic operator algebra seminorms, that is, the non-zero values $\varphi$ are only negative powers $p^{-n}$ of $p$, with $n\in \N$.

Given a $p$-adic operator seminorm on $A$, we define its \emph{nucleus} as
$$\NN_\varphi\colon=\{a\in A: \varphi(a)=0\}.$$
Notice that this is always a $*$-ideal of $A$; and if $\varphi$ is represented by $\varphi(a)=\|\pi(a)\|$, then $\NN_\varphi=\ker(\pi)$. In particular, $\varphi$ is a norm if and only if $\pi$ is injective on $A$. Moreover, the seminorm always induces a norm on the quotient $*$-algebra
$$\tilde\varphi\colon \tilde A:=A/\NN_\varphi\to \R_+,\quad \tilde\varphi(a+\NN_\varphi):=\varphi(a),$$
and this is a $p$-adic operator norm that is represented by $\tilde\varphi(a)=\|\tilde\pi(a)\|$, where $\tilde\pi$ is the induced injective $*$-homomorphism
$$\tilde\pi\colon \tilde A\to \BUN(\qP(X)),\quad \tilde\pi(a+\NN_\varphi):=\pi(a).$$
The completion $A^\varphi$ of $\tilde A$ with respect to the norm $\tilde\varphi$ is a Banach $*$-algebra over $\zP$ and $\tilde\pi$, being isometric, extends to an isometric $*$-isomorphism
$$A^\varphi\congto \overline{\tilde\pi(A)}=\overline{\pi(A)}\sbe \BUN(\qP(X)),$$
turning $A^\varphi$ into a $p$-adic operator algebra.

\begin{lemma}\label{lem:sup-p-adic-seminorms}
    If $(\varphi_i)_{i\in I}$ is any family of $p$-adic operator algebra seminorms on $A$, then
    $$\varphi(a):=\sup_{i\in I}\varphi_i(a)$$
    is also a $p$-adic operator seminorm.
\end{lemma}
\begin{proof}
    Suppose $\varphi_i$ is represented by $\varphi_i(a)=\|\pi_i(a)\|$ for $*$-representations $\pi_i\colon A\to \BUN(\qP(X_i))$ for certain sets $X_i$ with $i\in I$. Then $\varphi$ is represented by the product $\pi=\Pi_i \pi_i\colon A\to \BUN(\qP(X))$ with $X=\sqcup_{i\in I}X_i$ the disjoint union of the $X_i$:
    $$(\pi(a)f)(x,i):=(\pi_i(a)f|_{X_i})(x,i),\quad x\in X_i.$$
    It is then clear that $\|\pi(a)\|=\sup_{i\in I}\|\pi_i(a)\|$ for all $a\in A$, as desired.
\end{proof}

\begin{definition}
    Let $A$ be $\ast$-algebra over $\zP$. Let $\mathcal{SN}_p(A)$ be the set of all $p$-adic operator algebra seminorms $\varphi\colon A\to \R_+$. The \emph{universal $p$-adic operator algebra seminorm} of $A$ is defined by
\begin{align*}
    \|a\|_u := \sup_{\varphi \in \mathcal{SN}_p(A)} \varphi(a)
\end{align*}
The Hausdorff completion $A^u:=A^{\|\cdot \|_u}=\overline{A/\NN_{\|\cdot\|_u}}$ of $A$ with respect to $\|\cdot\|_u$ will be called the \emph{enveloping $p$-adic operator algebra} of $A$.
\end{definition}

Notice that, by Lemma~\ref{lem:sup-p-adic-seminorms}, $A^u$ is indeed a $p$-adic operator algebra, that is, it can be isometrically represented into $\BUN(\qP(X))$ for some set $X$. By construction, $A^u$ can be determined, up to isometric $*$-isomorphism, by the following universal property:

\begin{proposition}[Universal property of enveloping $p$-adic operator algebras]
Given a $*$-algebra $A$ over $\zP$, its enveloping $p$-adic operator algebra $A^u$ admits a $*$-homomorphism $\iota\colon A\to A^u$ with dense image, and has the following universal property: for every $*$-homomorphism $\sigma\colon A\to B$ into a $p$-adic operator algebra $B$, there is a unique morphism of $p$-adic operator algebras $\sigma^u\colon A^u\to B$ such that $\pi^u\circ\iota=\pi$. In other words, the functor \(A \mapsto A^u\) from the category of \(*\)-algebras over \(\zP\) to \(\OA\) is left adjoint to the inclusion functor. 
\end{proposition}
\begin{proof}
    By construction, we have a canonical $*$-homomorphism $\iota\colon A\to A/\NN_{\|\cdot\|_u}\sbe A^u$ with dense image.
    And given a $*$-homomorphism $\sigma\colon A\to B$ into a $p$-adic operator algebra, 
    choose an isometric $*$-representation $\pi\colon B\to \BUN(\qP(X))$ for some set $X$. 
    Then the composition $\iota\circ\pi\colon A\to \BUN(\qP(X))$ is a $*$-homomorphism and its associated $p$-adic seminorm is $$\|\pi(a)\|=\|\iota(\pi(a))\|\leq \|a\|_u.$$
    This implies that $\ker(\pi)\supseteq \NN_{\|\cdot\|_u}$ so that $\pi\colon A\to B$ factors through a contractive $*$-homomorphism $A/\NN_u\to B$ that, therefore, extends to a contractive $*$-homomorphism $\pi^u\colon A^u\to B$ which has the desired property by construction.  
\end{proof}

\begin{remark}
    If \(A\) is a $*$-algebra over the residue field \(\fF_p\), viewed as an involutive \(\zP\)-algebra via the quotient map, then \(A^u = 0\) as there are no nontrivial representations of such algebras on \(\BUN(\qP(X))\) for any \(X\). This is of course needed as although colimits create torsion, the enveloping algebra construction collapses all such algebras to the zero Banach algebra. 
\end{remark}

\begin{remark}
If $A$ is already a $p$-adic operator algebra then,  by the universal property of $A^u$, the identity map induces a contractive (surjective) $*$-homomorphism $\sigma\colon A^u \to A$ satisfying $\sigma(\iota(a))=a$ for all $a\in A$. In particular canonical $*$\nb-ho\-mo\-mor\-phism $\iota\colon A\to A^u$ is injective in this case. But it need not be contractive (see Example~\ref{ex:univ-diff-padic}). Indeed, it will be contractive if and only if $A^u=A$, meaning that $\|\cdot\|_A=\|\cdot\|_u$ is the universal enveloping norm. 
\end{remark} 

We now show that if $A$ is already a $p$-adic operator algebra and its norm is the canonical $p$-adic norm (Example~\ref{ex:main-Banach}), then $A^u\cong A$ as a $p$-adic operator algebras. In particular this applies if the underlying Banach algebra of \(A\) is the unit ball of a Banach \(\qP\)-algebra, it coincides with its enveloping algebra.

\begin{proposition}\label{prop:A=A^u}
    Let \(A\) be a \(*\)-subalgebra of \(\BUN(\qP(X))\) for some \(X\). Suppose further that the induced \(p\)-adic operator algebra norm on \(A\) coincides with its canonical \(p\)-adic norm. Then its $p$-adic completion \(\widehat{A}\) with the canonical \(p\)-adic norm is isometrically $*$-isomorphic to \(A^u\). In particular, if \(A\) is a \(p\)-adic operator algebra whose underlying norm is the canonical \(p\)-adic norm, then \(A\cong A^u\).  
\end{proposition}

\begin{proof}
   Let \(\iota \colon A \to \BUN(\qP(X))\) denote the embedding \(*\)-homomorphism. We first observe that the \(p\)-adic norm on \(\BUN(\qP(X))\), which by Proposition \ref{lem:p-adic-complete} is the operator norm, restricts to one on \(A\). Consequently, we have \(\norm{a}_p = \norm{\iota(a)}_p \leq \norm{a}_u\). For the other implication, let \(\phi \colon A \to \BUN(\qP(Y))\) be an arbitrary \(*\)-homomorphism. Writing an arbitrary element in \(A\) as \(a = u p^{\nu_p(a)}\) for \(u \in A\) of unit norm, we have \begin{align*}
  \norm{\phi(a)} = \norm{\phi(p^{\nu_p(a)}u)} = p^{-\nu_p(a)} \norm{\phi(u)} \leq p^{-\nu_p(a)} = \norm{\iota(a)} = \norm{a}_p,
\end{align*} so that the maximal norm is bounded by the \(p\)-adic norm of \(a\). \qedhere     
\end{proof}

Most \(p\)-adic operator algebras we construct satisfy the assumption from Proposition~\ref{prop:A=A^u}. That is, we take a torsionfree \(\zP\)-algebra admitting a faithful \(p\)-adic Hilbert space representation \(A \subseteq \BUN(\qP(X))\) (with torsionfree quotient \(\BUN(\qP(X))/A\)), and then take its \(p\)-adic completion inside \(\BUN(\qP(X))\). In particular, the underlying Banach algebra of the enveloping operator algebra in this case is automatically a bornological Banach \(\zP\)-algebra.

\begin{example}\label{ex:univ-diff-padic}
The norm of a $p$-adic operator algebra is, in general, not the canonical $p$-adic norm. Indeed, consider the principal ideal $p \zP\idealin \Zp$; this is a $p$-adic operator algebra because it is a closed $*$-subalgebra of $\zP$. Note that the element $p$ is not divisible by $p$ inside $p\zP$ because $1 \not\in p\zP$. Consequently, $\nu_p(p) = 0$ inside of $p\zP$, and thus the $p$-adic norm $\|\cdot\|_p$ of $p\zP$ differs from the norm induced from $\zP$. 

Indeed, we can describe $A^u$ for $A=p\Zp$ and show that $A^u$ and $A$ are \emph{not} isomorphic as $p$-adic operator algebras. To see this, 
consider the matrix $a\in \Mat_p(\Zp)$ with all entries $a_{ij}=1$; for $p=2$ this is the matrix appearing in Example~\ref{exa:counter-example-Cst-axiom}. Then $a^2=pa$ and there exists a unique $*$-homomorphism $A\to \Mat_p(\Zp)$ sending $p\mapsto a$. Since $\Mat_p(\Zp)$ is a $p$-adic operator algebra with $\|a\|=1$, this implies $\|p\|_u=1=\|p\|_p$, the $p$-adic norm of $p$ in $A=p\Zp$. By induction it follows $\|p^{n+1}\|_u=1/p^n=\|p^{n+1}\|_p$ and therefore $\|\cdot\|_u=\|\cdot\|_p$. We conclude that $A^u$ equals $A=p\Zp$ but with the different norm $\|\cdot\|_u=\|\cdot\|_p\not=\|\cdot\|_A$.
\end{example}

\begin{remark}
Example \ref{ex:univ-diff-padic} leads us to highlight several interesting and related observations:
\begin{enumerate}
    \item We see that there exists an element \(a \in A \subset \BUN(\qP(X))\) in a \(p\)-adic operator algebra with \(\norm{a}_p > \norm{a}\) if and only if \(\BUN(\qP(X))/A\) is torsionfree. 
    \item The property that $A^u=A$, that is, $\|\cdot\|_u=\|\cdot\|_A$ for $A$ a $p$-adic operator algebra does not pass to (closed, two-sided) $*$-ideals. In particular this implies that \emph{not} every representation $J\to \BUN(\Qp(X))$ from an ideal $J\idealin A$ extends to a representation $A\to \BUN(\Qp(X))$. The analogue of this extension property is known to hold in the category of \cstar{}algebras, see \cite{DixmierBook}*{Proposition~2.10.4} or \cite{Arveson}*{Section~1.3}.
    \item Finally, we see that for an isometric embedding $A\into B$ of $p$\nb-adic operator algebras, the induced morphism $A^u\to B^u$ is not isometric in general, even if $A$ is an ideal of $B$. This means that for $a\in A$, we have $\|a\|_{B^u}\leq \|a\|_{A^u}$ but the inequality might be strict.
\end{enumerate}    
\end{remark}

Proposition~\ref{prop:A=A^u} provides a recipe for several natural examples of \(p\)-adic operator algebras. We describe in what follows some concrete examples.

\begin{example}[Toeplitz algebra]
Consider the unilateral right-shift operator \(T \colon \qP(\N) \to \qP(\N)\) that takes \((x_0,x_1,\dotsc,) \mapsto (0, x_0, x_1, \dotsc)\); this appears in \cite{claussnitzer-thesis}*{Section 5.2}. The adjoint operator is given by the left-shift \(T^* \colon \qP(\N) \to \qP(\N)\), \((x_0, x_1,\dotsc,) \mapsto (x_1,x_2, \dotsc,)\), and we have \(T^* T = 1\), that is, \(T\) is an isometry. Then the \(\zP\)-algebra \(\zP\langle T,T^* \rangle\)  generated by \(T\) and \(T^*\) is a \(*\)-subalgebra of \(\BUN(\qP(\N))\), and a simple computation shows that for any \(x = \sum_{j,k = 0}^{l} \lambda_{j,k} T^j (T^*)^k\), we have \(\norm{x}_p = \max_{j,k} \abs{\lambda_{j,k}}_p = \norm{x}\). By Proposition~\ref{prop:A=A^u}, its \(p\)-adic completion is the enveloping algebra \(\mathcal{T}_p\) of \(\zP\langle T,T^* \rangle\) which we think of as the $p$-adic analogue of the \textit{Toeplitz algebra}. This is justified by a result of Jacobson \cite{jacobson1950some}*{Theorem 4}, which specialises to the fact that \(\zP\langle T,T^* \rangle\) is the universal \(\zP\)-algebra generated by a proper isometry. Consequently, we deduce that the Toeplitz algebra \(\mathcal{T}\) (as defined in \cite{claussnitzer-thesis}) is the \emph{universal \(p\)-adic operator algebra generated by a proper isometry}, which should be viewed as a \(p\)-adic analogue of Coburn's Theorem (\cite{Murphy}*{Theorem~3.5.18}).
\end{example}

\begin{example}[Twisted group algebras]
Let $G$ be a group and let $\omega\colon G\times G\to \zP^\times$ be a (normalized) $2$-cocycle on $G$ with values in the multiplicative group $\zP^\times$ of invertible elements of $\zP$. This means that $\omega$ satisfies the relations
\begin{equation}\label{eq:cocycle-conditions}
    \omega(g,1)=\omega(1,h)=1,\quad \omega(g,h)\omega(gh,k)=\omega(g,hk)\omega(h, k).
\end{equation}
The twisted group algebra is the $*$-algebra $\zP[G,\omega]=\contc(G,\zP)$ over $\zP$ consisting of finitely supported functions $G\to \zP$ endowed with $\omega$-twisted convolution product and involution given by the formulas
$$(\varphi*_\omega \psi)(g):=\sum_{h\in G}\varphi(h)\psi(h^{-1}g)\omega(h,h^{-1}g),\quad \varphi^{*,\omega}(g):=\omega(g^{-1},g)\varphi(g^{-1}).$$
On the generators (i.e. $\delta$-functions), this reads as
\begin{equation}\label{eq:relations-twisted-group-algebra}
    \delta_g*_\omega\delta_h=\omega(g,h)\delta_{gh},\quad \delta^{*,\omega}_g=\omega(g^{-1},g)\delta_{g^{-1}}.
\end{equation}
It is well know and easily verified that these operations turn $\zP[G,\omega]$ into an involutive $\zP$-algebra. The cocycle condition~\eqref{eq:cocycle-conditions} is exactly what one needs for the twisted convolution to be associative. 

As in the untwisted group algebra case (Example~\ref{ex:group-ring}), we can represent $\zP[G,\omega]$ faithfully into $\BUN(\qP(G))$ via the $\omega$-twisted left regular representation given by
$$\lambda^\omega_g(\xi)(h):=\omega(g,g^{-1}h)\xi(g^{-1}h).$$
This is a $\omega$-representation meaning that it satisfies $\lambda^\omega_g\lambda^\omega_h=\omega(g,h)\lambda^\omega_{gh}$ for all $g,h\in G$.
One proves, as in the untwisted case, that one gets $\|\lambda^\omega(\xi)\|=\|\xi\|_\infty$, the supremum norm of $\xi$. Therefore $\OO_p(G,\omega):= c_0(G,\zP)$ is a $p$-adic operator algebra with respect to the above operations, and it is isometrically embedded as a $*$-subalgebra of $\BUN(\Qp(G))$ via $\lambda^\omega$. Notice that the twisted convolution and involution make sense on $c_0(G,\zP)$.

Proposition~\ref{prop:A=A^u} implies that $\OO_p(G,\omega)=\OO_p(G,\omega)^u=\zP[G,\omega]^u$ is the enveloping $p$-adic operator algebra of $\zP[G,\omega]$. In other words, this is the universal (unital) $p$-adic operator algebra generated by (unitaries) $\delta_g$ for $g\in G$ with product and involution satisfying~\eqref{eq:relations-twisted-group-algebra}.
\end{example}

\begin{example}[Rotation algebras]
\label{ex:Rotation-Algebras}
As a special case of the previous example, consider the additive abelian group $G=\zZ^2=\zZ\times\zZ$. Given any $z\in \zP^\times$, we get a $2$-cocycle $\omega_z$ by the formula
$$\omega_z((k,l),(m,n)):=z^{lm}.$$

The twisted group algebra $\zP[\zZ^2,\omega_z]$ is the universal unital $*$-algebra over $\zP$ generated by two unitaries $U=\delta_{(0,1)}$ and $V=\delta_{(1,0)}$ satisfying the relations
$$UV=z VU.$$
Therefore $\OO_p(\zZ^2,\omega_z)\cong c_0(\zZ^2,\zP)$ is the universal unital $p$-adic operator algebra generated by two unitaries satisfying the same relation. We view this as a $p$-adic version of the rotation \cstar{}algebras $C^*(\T_\theta)$ that can be realized in a similar way as twisted group \cstar{}algebras $C^*(\zZ^2,\omega_\theta)$ for $\omega_\theta((k,l),(m,n))=e^{2\pi i\theta lm}=z_\theta^{lm}$, where $z_\theta=e^{2\pi i \theta}\in \T$, see \cite{Gillaspy}.

Notice that the $p$-adic operator algebra $\OO_p(\zZ^2,\omega_z)$ can be represented isometrically into $\BUN(\qP(\zZ^2))$ via $\lambda^{\omega_z}$. This gives the concrete representation of this algebra as the $p$-adic operator subalgebra of $\BUN(\qP(\zZ^2))$ generated by the unitary operators $U:=\lambda^{\omega_z}_{(0,1)}$ and $V:=\lambda^{\omega_z}_{(1,0)}$, which are given on the standard basis $(\delta_{m,n})_{(m,n)\in \zZ^2}\sbe \qP(\zZ)$ by the formulas
$$U(\delta_{m,n})=z^m\delta_{m,n+1},\quad V(\delta_{m,n})=\delta_{m+1,n}.$$
\end{example}

Recall that for an ample (i.e étale totally disconnected) Hausdorff groupoid, we may define its \emph{Steinberg algebra} \begin{align*}
    \SS_p(\GG) := \set{\phi\colon \GG \to \zP\colon \phi \text{ is locally constant and compactly supported}}
\end{align*} over \(\zP\).

\begin{proposition}\label{pro:iso-Steinberg-universal}
Let $\GG$ be an ample Hausdorff groupoid and consider its $p$-adic operator algebra $\OO_p(\GG)$ as in Example~\ref{ex:etale-groupoid}. We then have $\SS_p(\GG)^u = \OO_p(\GG)$, where we view the Steinberg algebra as 
\end{proposition}
a $*$-subalgebra of the convolution $*$-algebra $\contc(\GG,\Zp)\sbe \OO_p(\GG)$.
\begin{proof}
The space $\SS_p(\GG)$ consists of $\Zp$-linear combinations of characteristic functions of compact open subsets. In particular this space separates points of $\GG$ and therefore it forms a dense subalgebra of $\contz(\GG,\Zp)=\OO_p(\GG)$  with the supremum norm by
Kaplansky's non-archimedean version of the Stone-Weierstra\ss' Theorem \cite{Kaplansky-Weierstrass}.

Let $B$ be a $p$-adic operator algebra. Let $\pi\colon \SS_p(\GG) \to B$ be an arbitrary $*$-ho\-mo\-mor\-phism of $*$-algebras. Let $\phi \in \SS_p(\GG)$. Using proposition $1.19$ of \cite{Rigby-Steinberg-Algebras}, we know there are mutually disjoint bisections $U_1, \dots, U_n$ and scalars $\alpha_1, \dots, \alpha_n \in \zP$ such that 
\begin{align*}
    \phi = \sum_{k=1}^n \alpha_k \chi_{U_k}
\end{align*}
Thus
\begin{align*}
    \|\pi(\phi)\| = \Big\|\sum_{k=1}^n \alpha_k \pi(\chi_{U_k})\Big\| \leq \max_{1 \leq k \leq n} |\alpha_k| \|\pi(\chi_{U_k})\| \leq \max_{1 \leq k \leq n} |\alpha_k| = \|\phi\|
\end{align*}
Here we are considering the supremum norm in $\SS_p(\GG)$. From this, we see that $\pi$ is contractive. As $\SS_p(\GG)$ is dense in $\OO_p(\GG)$, there exists a unique contractive $*$-ho\-mo\-mor\-phism $\OO_p(\GG) \to B$ extending $\pi$. Therefore $\OO_p(\GG)$ satisfies the universal property of the enveloping of $\SS_p(\GG)$ concluding that $\SS_p(\GG)^u = \OO_p(\GG)$.
\end{proof}

In particular for algebras associated to graphs, we get the following consequence:

\begin{corollary}\label{cor:Steinberg}
    Let $E=(\s,\rg\colon E^1\to E^0)$ be a directed graph. Then $\OO_p(E)$ is the universal $p$-adic operator algebra generated by pairwise orthogonal projections $v$ with $v\in E^0$ and partial isometries $e$ with $e\in E^1$ satisfying the following relations:
\begin{enumerate}
    \item $vw=\delta_{v,w}v$ for all $v,w\in E^0$;
    \item $s(e)e= e= er(e)$ for all $e\in E^1$;
    \item $e^*f=\delta_{e,f}r(e)$ for all $e,f\in E^1$;
    \item $\sum_{\s(e)=v}ee^*=v$ for all $v\in E^0$ whenever $s^{-1}(v)$ is finite and non-empty.
\end{enumerate}
\end{corollary}
\begin{proof}
    This follows from the fact that $\OO_p(E)=\SS_p(\GG_E)^u$ and that the Steinberg algebra $\SS_p(\GG_E)$ is isomorphic to the Leavitt path algebra $L_{\Zp}(E)$, which is the universal involutive $\Zp$-algebra generated by $v\in E^0$ and $e\in E^1$ satisfying the same relations as in the statement, see e.g. \cite{Abrams-Aranda}.
\end{proof}

\begin{example}\label{exa:Cuntz-algebra}
The above in particular gives us access to $p$-adic operator algebras versions of the classical Cuntz-Krieger algebras. Taking, for instance, the graph $E_n$ with a single vertex and $n$ loops, we get the $p$-adic version of the Cuntz algebra $\OO_{p,n}:=\OO_p(E_n)$. This is the universal unital $p$-adic operator algebra generated by $n$-isometries $s_1,\ldots,s_n$ satisfying $s_i^*s_j=\delta_{i,j}1$ and the Cuntz relation $$s_1s_1^*+\ldots+s_ns_n^*=1.$$
\end{example}

We are now ready to prove that $\OA$ has colimits. Let $I$ be a small category, and let $F\colon I \to \OA$ be a functor. Also, let $U\colon \OA \to \set{*\text{-algebras over }\zP}$ be the forgetful functor. The category of $*$-algebras over $\zP$ has all colimits; hence, we can define $C = \colim \,  U \circ F$.

We claim that $C^u = \colim F$. To see this, let $B$ be a cone under the diagram $F$; then, $U(B)$ is a cone under the diagram $U \circ F$. Consequently, there is a unique map $C \to B$ making the diagram commute. Since $(-)^u$ is the left adjoint of $U$, there exists a unique map $C^u \to B$ making the diagram commute. Thus, $C^u = \colim \, U \circ F$.

\begin{corollary}\label{prop:OA-has-colimits}
The category $\OA$ has all colimits.
\end{corollary}

\begin{example}[Compact operators]\label{exa:compact-operators} 
We now construct the compact operators as an inductive limit in the category of \(p\)-adic operator algebras of finite matrix algebras. Recall from Example~\ref{ex:matrix} that finite sets \(X\) yield matrix algebras \(\Mat_n(\zP)\), where \(n\) is the cardinality of the set \(X\). Now consider the inductive system 

\begin{align*}
    \Mat_1(\zP) \rightarrow \Mat_2(\zP) \rightarrow \Mat_3(\zP) \rightarrow \cdots ,
\end{align*}

where the structure maps are the usual block inclusions \(a \mapsto \begin{pmatrix}
    a & 0 \\
    0 & 0
\end{pmatrix}\). The inductive limit is the enveloping algebra of the direct union 
\(\Mat_\infty(\zP) = \bigcup_n \Mat_n(\zP)\) of matrix algebras, which by Proposition~\ref{prop:A=A^u} is the \(p\)-adic completion \(\widehat{\Mat_\infty(\zP)}\) with the canonical norm. By \cite{cortinas2019non}*{Example 6.4}, this inductive limit can be identified with the algebra of matrices with entries converging to zero at infinity, which coincides with the algebra of contractive compact operators on \(\qP(\N)\) by \cite{claussnitzer-thesis}*{Definition ~3.2}. 

 This example can also be approached via groupoids: if we consider the pair groupoid $\GG=X\times X$ of a countably infinite set $X$ endowed with the discrete topology, then $\SS_p(\GG)=c_c(\GG,\zP)\cong \Mat_\infty(\zP)$, see \cite{Rigby}*{Proposition~1.28}. And then by Proposition~\ref{pro:iso-Steinberg-universal}, 
 $$\Mat_\infty(\zP)^u\cong \OO_p(\GG)\cong c_0(\GG,\zP)\cong\K_{\leq 1}(\qP(X)).$$
\end{example}

\section{Tensor products}

The category of $p$-adic operator algebras admits at least two natural tensor products, namely the maximal and the spatial tensor products, similar to the category of \cstar{}algebras. In the next sections we present the definitions and some basic examples for these tensor products and show that for many $p$-adic operator algebras both tensor products coincide. Moreover, we show that the spatial tensor product of $p$-adic operator algebras coincides with their projective tensor product whenever the operator algebras involved carry the $p$-adic norm.

\subsection{The maximal tensor product}

Given two $*$-algebras $A$ and $B$ over $\Zp$, we shall write $A\oalg B$ for the algebraic tensor product (over $\Zp$). This is viewed as another $*$-algebra over $\Zp$ in the usual way:
$$(a\otimes b)(a'\otimes b')=aa'\otimes bb',\quad (a\otimes b)^*=a^*\otimes b^*.$$

\begin{definition} Given two $p$-adic operator algebras $A$ and $B$, we define the \textit{maximal tensor product} 
\begin{align*}
    A \omax B \defeq (A \oalg B)^u
\end{align*} as the enveloping algebra of $A \oalg B$.
\end{definition}

We now formulate the universal property of the maximal tensor product:

\begin{proposition}[Universal property of the maximal tensor product] 
\label{pro:univ-property-tensor-max}
Let $A$, $B$ and $C$ be $p$-adic operator algebras. Let $\phi\colon A \to C$ and $\psi\colon B \to C$ be two (not necessarily contractive) $*$-homomorphisms satisfying $\phi(a)\psi(b)=\psi(b)\phi(a)$ for all $a\in A$, $b\in B$. Then there exists a unique morphism $\pi\colon A \omax B \to C$ such that 
\begin{align*}
    \pi(a \otimes b) = \phi(a) \psi(b)\quad\mbox{for all }a\in A, b\in B.
\end{align*}
\end{proposition}
\begin{proof}
Consider the map 
\begin{align*}
    \phi \times \psi\colon A \times B &\to C \\
    (a, b) &\mapsto \phi(a) \psi(b)
\end{align*}
by the universal property of the algebraic tensor product, there is a unique $\zP$-linear morphism
\begin{align*}
    \phi \otimes \psi\colon A \otimes B &\to C \\
    a \otimes b &\mapsto \phi(a) \psi(b)
\end{align*}
One easily checks that this is a $*$-algebra morphism, hence the universal property of the enveloping algebra yields a unique continuous $*$-morphism $A \omax B \to C$ with the desired properties.
\end{proof}

\begin{remark}\label{rem:unital-embeddings-tensor}
If $B$ is unital, we get a canonical $*$-homomorphism $\iota_A\colon A\to A\omax B$ given by $\iota_A(a)=a\otimes 1_B$. If, in addition, $A$ is unital, then so is $A\oalg B$ and therefore also $A\omax B$ with unit $1_A\otimes 1_B$, and in this case we also have a $*$-homomorphism $\iota_B\colon B\to A\omax B$. Both $\iota_A$ and $\iota_B$ are unital $*$-homo\-mor\-phisms and have commuting ranges as in the above proposition, that is, $\iota_A(a)\iota_B(b)=\iota_B(b)\iota_A(a)$ for all $a\in A$, $b\in B$. The induced morphism $A\omax B\to A\omax B$ is the identity map.
But the homomorphisms $\iota_A$, $\iota_B$ need not be contractive, in general.
For example, taking $B=\Zp$, we get $A\oalg \Zp=A$ and therefore $A\omax \Zp=A^u$ and $\iota_A$ represents the canonical homomorphism $A\to A^u$, which is contractive if and only if $\|\cdot\|_A=\|\cdot\|_{A^u}$, and this is not always the case, see Example~\ref{ex:univ-diff-padic}.
\end{remark}

\begin{proposition}\label{pro:alg-tensor-enveloping}
Let $A$ and $B$ be unital $*$-algebras over $\Zp$. Then we have a canonical isomorphism of $p$-adic operator algebras
$$(A\oalg B)^u\cong A^u\omax B^u.$$
\end{proposition}
\begin{proof}
Given a $*$-homomorphism $\rho\colon A\oalg B\to C$ for $C$ a $p$-adic operator algebra, we get $*$-homomorphisms $\pi\colon A\to C$, $\pi(a):=\rho(a\otimes 1_B)$ and $\sigma\colon B\to C$, $\sigma(b):=\rho(1_A\otimes b)$. By the universal properties of the envelope, we obtain unique extensions of these to morphisms $\pi^u\colon A^u\to C$ and $\rho^u\colon B^u\to C$, and these still commute pointwise by density and continuity.
By Proposition~\ref{pro:univ-property-tensor-max} we obtain a unique morphism of $p$-adic operator algebras $\rho^u\colon (A\oalg B)^u\to C$ satisfying $\rho^u(a\otimes b)= \pi^u(a)\sigma^u(b)=\pi(a)\sigma(b)=\rho(a\otimes b)$. This shows that
$A^u\omax B^u$ has the desired universal property of $(A\oalg B)^u$ and therefore $A^u\omax B^u\cong (A\oalg B)^u$.
\end{proof}

We can apply the above to $p$-adic operator algebras in order to get:

\begin{corollary}\label{cor:unital-p-adic-max-tensor}
Let $A$ and $B$ be unital $p$-adic operator algebras. Then
$$A\omax B\cong A^u\omax B^u.$$
\end{corollary}

\subsection{The spatial tensor product}
As in the archimedean case, there is also a spatial tensor product that was already studied in \cite{claussnitzer-thesis}. We recall this construction: let \(X\) and \(Y\) be nonempty sets. Define \(\qP(X) \otimes \qP(Y) \defeq \qP(X \times Y)\), and for \(\xi \in \qP(X)\) and \(\eta \in \qP(Y)\), we define \[\xi \otimes \eta(x,y) \defeq \xi(x) \eta(y)\] for all \(x\in X\) and \(y \in Y\). Then for operators \(U \in \BUN(\qP(X))\) and \(V \in \BUN(\qP(Y))\), we define \[(U \otimes V) (\delta_x \otimes \delta_y) \defeq U(\delta_x) \otimes V(\delta_y)\] for \(x\in X\) and \(y \in Y\). Note that in order that \(U \otimes V\) is indeed an operator on \(\qP(X) \otimes \qP(Y)\), we really need to work with unit balls of bounded operators on Hilbert spaces instead of the entire algebra of operators. 

\begin{definition}\label{def:tensor-product}
    Let \(A\) and \(B\) be \(p\)-adic operator algebras and let \(\pi \colon A \hookrightarrow \BUN(\qP(X))\) and \(\sigma \colon B \hookrightarrow \BUN(\qP(Y))\) be isometric embeddings. Their \textit{spatial tensor product} \(A \hot B\) is defined as the completion of the \(\zP\)-algebra generated by \(\setgiven{\pi(a) \otimes \sigma(b)}{a \in A,b \in B}\) in \(\BUN(\qP(X) \otimes \qP(Y))\) in the induced norm.
\end{definition}

A priori, the spatial norm depends on the choice of the representations $\pi$ and $\sigma$ and this should be remembered in our notation by something like $A\hot_{{\pi,\sigma}} B$. This dependence, however, disappears for all algebras of our interest by Proposition~\ref{prop:minimal=projective}, hence we just use the simpler notation $A\hot B$ in what follows. 

Recall that there is another tensor product of Banach algebras, namely, the completed projective tensor product. As a module, this represents bounded bilinear maps between Banach \(\zP\)-modules. More explicitly, given two Banach \(\zP\)-modules \(V\) and \(W\), their completed projective tensor product \(V \haptimes W\) is given by the completion of the algebraic tensor product \(V \otimes W\) in the norm defined by 
\[\norm{u}_p \defeq \inf\bigg\{\underset{1 \leq i \leq r}\max \norm{a_i} \norm{b_i} : {u = \sum_{i=1}^r a_i \otimes b_i, \quad a_i \in V, b_i \in W}\bigg\}.\] In another difference to the archimedean \(C^*\)-algebraic case, the completed projective tensor product of two \(p\)-adically complete Banach  \(\zP\)-algebras is a \(p\)-adic operator algebra as the following result demonstrates:

\begin{proposition}\label{prop:minimal=projective}
    Let \(A\) and \(B\) be \(p\)-adic operator algebras that are bornological. Then the Banach \(\zP\)-algebras \(A \hot B\) and \(A \haptimes B\) are isometrically isomorphic. 
\end{proposition}
\begin{proof}
   This is essentially \cite{claussnitzer-thesis}*{Lemma 5.1.4}, but we briefly sketch the proof for the convenience of the reader. Let \(a \in A\) and \(b \in B\) be two elements in two \(p\)-adic operator algebras, represented on \(p\)-adic Hilbert spaces \(\qP(X)\) and \(\qP(Y)\). Then it is easy to see using the definition of the tensor product of operators and the multiplicativity of the \(p\)-adic norm that \(\norm{a \otimes b} = \norm{a} \norm{b} = \norm{a \otimes b}_{p}\), where the latter equality is immediate from the definition of the completed projective tensor product. Now consider a specific representation \(u = \sum_{i=1}^r a_i \otimes b_i\), where \(a_i \in A\) and \(b_i \in B\). By \cite{claussnitzer-thesis}*{Lemma 2.3.5}, we may arrange that the collection \((a_i)\) is an orthogonal system, so that \(\norm{\sum_{i=0}^r \lambda_i a_i} = \max \norm{\lambda_i a_i}\). Consequently by the computation in the proof of \cite{claussnitzer-thesis}*{Lemma 5.1.5}, we have \[\norm{u} = \max_{s,x \in X, t,y \in Y} \abs{\sum_{i=1}^r (a_i \otimes b_i)(\delta_s \otimes \delta_t)(x \otimes y)} = \max_{1 \leq i \leq r} \norm{a_i} \norm{b_i}.\] 
  Finally, varying the representations \(u = \sum_{i=1}^r a_i \otimes b_i = \sum_{j=1}^l p_j \otimes q_j\) and taking their infimum shows that the projective and the operator norms coincide on the dense subspace \(A \oalg B\), which concludes the result. 
\end{proof}

Note that for bornological Banach \(\zP\)-modules, the completed projective tensor product coincides with the \(p\)-adic completion of the algebraic tensor product. Furthermore, in what is a particular nice feature of the nonarchimedean setting, the completed projective tensor product is particularly well-behaved for inclusions:

\begin{corollary}\label{cor:Grothendieck-approx}
    Let \(A \to B\) be an injective, bounded \(*\)-homomorphism of \(p\)-adic operator algebras, whose underlying norms are the canonical \(p\)-adic norms. Then for any \(p\)-adic operator algebra \(D\) with respect to the canonical \(p\)-adic norm, we have an injective, bounded  \(*\)-homomorphism \(f \otimes \id_D \colon A \hot D \to B \hot D\) of \(p\)-adic operator algebras. 
\end{corollary}

\begin{proof}
    Specialise \cite{cortinas2019non}*{Section 2.4} and use Proposition~\ref{prop:minimal=projective}.
\end{proof}

At this point it is of course a natural question when all the tensor products defined so far coincide. As in the \(C^*\)-algebraic case, this happens for a large class of \(p\)-adic operator algebras:

\begin{theorem}\label{thm:nuclearity1}
    Let \(A \to \BUN(\qP(X))\) and \(B \to \BUN(\qP(Y))\) be two  \(p\)-adic operator algebras with the \(p\)-adic norm, such that the induced norm on \(A \oalg B \subseteq \BUN(\qP(X) \otimes \qP(Y))\) is the \(p\)-adic norm. Then \[A \omax B \cong A \hot B \cong A \haptimes B.\]
\end{theorem}

\begin{proof}
    The hypotheses on \(A\) and \(B\) guarantee that \(A \oalg B \subseteq A \haptimes B\), and the latter by Proposition~\ref{prop:minimal=projective} is the same as \(A \hot B \subseteq \BUN(\qP(X) \otimes \qP(Y))\). The result now follows from Proposition~\ref{prop:A=A^u}.  
\end{proof}

As a consequence of Theorem~\ref{thm:nuclearity1}, we get that the maximal and spatial tensor products coincide for large classes of $p$-adic operator algebras, beginning with:

\begin{corollary}
    For any two sets $X$ and $Y$, we have
\begin{align*}
\BUN(\Qp(X))\omax\BUN(\Qp(Y))&=\BUN(\Qp(X))\hot\BUN(\Qp(Y))\\
&=\BUN(\Qp(X))\haptimes \BUN(\Qp(Y)).
\end{align*}
\end{corollary}
\begin{proof}
    Follows from Proposition~\ref{lem:p-adic-complete} and Theorem~\ref{thm:nuclearity1}
\end{proof}

The above result is in sharp contrast with the theory of \cstar{}algebras as for the \cstar{}algebra $\BB(\ell^2\N)$ of bounded operators on the (complex) Hilbert space $\ell^2\N$, Theorem~6 in \cite{OzawaPisier} shows that $\BB(\ell^2\N)\oalg \BB(\ell^2\N)$ has a continuum of \cstar{}norms!

Next we apply our results to $p$-adic operator algebras of (ample) groupoids and show that also there all tensor products coincide:

\begin{corollary}\label{pro:tensors=groupoids}
    Let $G$ and $H$ be ample Hausdorff groupoids. Then 
    \begin{align}\label{eq:isos-tensors-max-min-pi}
     \OO_p(G) \omax \OO_p(H) \cong \OO_p(G) \otimes_\pi \OO_p(H)
     \cong \OO_p(G) \hot \OO_p(H) \cong \OO_p(G \times H) 
    \end{align}
    with $G\times H$ denoting the direct product (ample) groupoid. 
\end{corollary} 

\begin{proof}
    From the main result of \cite{Rigby:Tensor}, we have a canonical $*$-isomorphism of Steinberg algebras 
$$\SS_p(G)\oalg\SS_p(H)\cong \SS_p(G\times H)$$
that sends an elementary tensor $\phi\otimes\psi$ in $\SS_p(G)\oalg\SS_p(H)$ to $[(g,h)\mapsto \phi(g)\psi(h)]$ in $\SS_p(G\times H)$. The result now follows from Proposition~\ref{pro:iso-Steinberg-universal} and Theorem~\ref{thm:nuclearity1}. 
\end{proof}

The above result applies, in particular, for group algebras, but also for many other algebras. Indeed, most algebras considered in this paper are groupoid algebras. 

We end this section with the nonarchimedean analogue of the well-known \(\mathcal{O}_2\)-stability problem in \(C^*\) and \(L^q\)-operator algebra theory. Recall that the Cuntz \cstar{}algebras $\OO_2=C^*(E_2)$ on two generators as well as the Cuntz \cstar{}algebra $\OO_\infty=C^*(E_\infty)$ on countably infinite generators have the self-absorption property for \cstar{}tensor products:
\begin{equation}\label{eq:Cuntz-cstar-absorption}
\OO_2\otimes \OO_2\cong \OO_2\quad\mbox{and}\quad \OO_\infty\otimes \OO_\infty\cong \OO_\infty .
\end{equation}
Indeed, these isomorphisms can even be chosen to be ``strongly absorbing'', see \cite{TomsWinter}. Here $\otimes$ denotes the spatial (i.e. minimal) tensor product of \cstar{}algebras -- the Cuntz algebras are nuclear, so the above also coincides with the maximal tensor product. The above self-absorption property of the Cuntz algebras (and other \cstar{}algebras, like the Jiang-Su algebra) is a very important feature in connection with Elliott's classification program of \cstar{}algebras, see e.g. \cite{Rordam:Classification,Winter:Localizing}.

On the other hand, in the realm of uncompleted algebras, the situation changes completely: it is shown in \cite{Ara-Cortinas} that for the Leavitt path algebra $L_{2,k}$ over a field $k$, \emph{there is no} isomorphism $L_{2,k}\oalg L_{2,k}\cong L_{2,k}$, and a similar result holds for $L_{\infty,k}$.

We finish this section with a simple application of our results showing that a $p$-adic analogue of self-absorption does not hold for the Cuntz algebras  $\OO_{p,2}$ and   $\OO_{p,\infty}$.

\begin{corollary}
    For every prime number $p$, the $p$-adic Cuntz algebras $\OO_{p,2}:=\OO_p(E_2)$ and $\OO_{p,\infty}:=\OO_p(E_\infty)$ are \emph{not} self-absorbing in the sense that $$\OO_{p,2}\otimes\OO_{p,2}\ncong \OO_{p,2}\quad\mbox{and}\quad \OO_{p,\infty}\otimes\OO_{p,\infty}\ncong \OO_{p,\infty},$$
    where $\otimes$ denotes either one of the tensor products $\omax$, $\hot$, $\otimes_\pi$.
\end{corollary}
\begin{proof}
By Proposition~\ref{pro:tensors=groupoids}, all tensor products coincide for the Cuntz algebras and we just write $\otimes$ for them.
Suppose we have an isomorphism $\OO_{2,p}\otimes\OO_{p,2}\cong\OO_{2,p}$.
Taking reduction mod $p$, this induces an isomorphism \[(\OO_{2,p} \otimes \OO_{p,2}) /p \cong (\OO_{p,2}/p) \otimes (\OO_{p,2}/p) \cong L_{\Fp}(E_2)\otimes L_{\Fp}(E_2)\cong L_{\Fp}(E_2),\] but this is a contradiction to \cite{Ara-Cortinas}*{Theorem 5.1}. Similarly, using \cite{Ara-Cortinas}*{Proposition~5.3} we get the result for $\OO_{p,\infty}$.
\end{proof}

\begin{remark}
Using a similar strategy as in the above proof, \cite{Ara-Cortinas}*{Theorem 5.1} actually gives a stronger result for the $p$-adic Cuntz algebra $\OO_{p,2}$: for $m,n\in \N$, we have an isomorphism of $p$-adic operator algebras
$$\underbrace{\OO_{p,2}\otimes\cdots\otimes\OO_{p,2}}_m\cong \underbrace{\OO_{p,2}\otimes\cdots\otimes\OO_{p,2}}_n$$
if and only if $m=n$. A similar result is proved in \cite{CGT:RigidityLp} for $L^q$-operator algebras.
\end{remark}

%\begin{lemma}[Every triangle is isosceles]\label{lem:triangle-isosceles} Let $M$ be an abelian group with a nonarchimedean norm. If $m, n \in M$ are two elements with $\|m\| \neq \|n\|$, then $\|n + m\| = \max \set{\|n\|, \|m\|}$.
%\end{lemma}
%\begin{proof}
%Immediately $\|n + m\| \leq \max \set{\|n\|, \|m\|}$. Now suppose without loss of generality that $\|n\| > \|m\|$, note that
%\begin{align*}
%    \|n\| \leq \|(n + m) - m\| \leq \max \set{\|n + m\|, \|m\|}
%\end{align*}
%but as $\|n\| > \|m\|$ the only option is that 
%\begin{align*}
%    \|n\| \leq \|n + m\|
%\end{align*}
%Concluding the result.
%\end{proof}

\begin{theorem}The Tate algebra $$\zP \langle T \rangle := \set{\sum_{n\in \nN} a_n T^n \mid \lim_n a_n = 0}$$is a $p$-adic operator algebra.
\end{theorem} 
\begin{proof} Let $s\colon \qP(\zZ)\to \qP(\zZ)$ be the shift operator defined by
\begin{align*}
    s(\xi(n)) = \xi(n + 1)
\end{align*}
for $\xi\in \qP(\zZ)$ and $n\in \zZ$. Note that $s^* = s^{-1}$. Let $\tau = s + s^{-1}$.

Consider the representation
\begin{align*}
    \phi\colon \zP[T] &\to \BUN(\qP(\zZ)) \\
    T &\mapsto \tau
\end{align*}
Let us show that $\phi$ is isometric with $\zP[T]$ carrying the norm
\begin{align*}
    \bigg\|\sum_{k \leq n}a_kT^k\bigg\| = \max_{k \leq n} |a_k|
\end{align*}
We need to prove that 
\begin{align*}
    \bigg\|\sum_{k \leq n}a_k\tau^k\bigg\| = \max_{k \leq n} |a_k|
\end{align*}
We are going to do induction on $n$. For $n = 0$ this is clear. Let $n > 0$. Suppose that $\norm{\phi(f)} = \norm{f}$ for every polynomial \(f\) with degree less than $n$. Write
\begin{align*}
    f = g + aT^n
\end{align*}
with $g\in \zP[T]$ with degree less than $n$ and $a \in \zP$. By the induction hypothesis $\|\phi(g)\| = \|g\|$. To see that $\|a\tau^n\| = |a|$ note that the matrix of $\tau^n$ always has a $1$ as one of its entries, for example
\begin{align*}
    (\tau^n)_{n, 0} &= \tau^n(\delta_0)(n) \\
                   &= \sum_{k \leq n} {n \choose k} s^{2k - n} (\delta_0)(n) \\
                   &= \sum_{k \leq n} {n \choose k} (\delta_{2k - n})(n) \\
                   &= \delta_n(n) = 1
\end{align*}
If $\|g\| \neq \|a\|$, then by the strict triangle inequality, we have $$\|\phi(f)\| = \max \set{\|g\|, |a|} = \|f\|.$$ If $\|g\| = |a|$, then note that the matrix entry $g(\tau)_{n, 0} = 0$, because if $k < n$, then $\tau^{k}(\delta_0)(n) = 0$ by the calculation we just made. Hence we conclude that $f(\tau)_{n,0} = |a| = \|f\|$. Thus $|f| \leq \|\phi(f)\|$. Reciprocally $\|\phi(f)\| \leq \|f\|$ easily by the ultrametric inequality.

As $\zP\langle T\rangle$ is the completion of $\zP[T]$ with this norm, the isometry extends to the completion. 
\end{proof}

\begin{corollary} The Tate algebra $\zP \langle T_1, T_2, \dots, T_n \rangle$ is a $p$-adic operator algebra.
\end{corollary}
\begin{proof}
As each $\zP\langle T_i \rangle$ is a $p$-adic operator algebra whose underlying Banach module carries the $p$-adic norm,
\begin{align*}
    \zP\langle T_1, T_2, \dots, T_n \rangle \cong \zP\langle T_1 \rangle \hot \cdots \hot \zP\langle T_n \rangle
\end{align*}
is a $p$-adic operator algebra.
\end{proof}

\begin{example}[Affinoid algebras] An affinoid \(\zP\)-algebra is a quotient of a Tate algebra by a closed ideal.

A particular consequence of the above result is that affinoid Banach $\zP$-algebras are $p$-adic operator algebras. This is an important class of algebras that appear in rigid analytic geometry. 
\end{example}

\section{Crossed Products}

Another construction that we can carry from the archimedean to the nonarchimedean world is the crossed product. Just like the case with the tensor product, there are two reasonable ways to construct the crossed product: one direct way by explicitly representing it, and one indirect way by using the enveloping algebra construction. We are going to see that in good cases, these two constructions coincide.

\subsection{The reduced crossed product}\label{crossed-productI} 

We first consider the \textit{reduced crossed product} that arises naturally as the completion of the algebraic crossed product with the supremum norm, and can be represented on a $p$-adic Hilbert space via an analogue of the regular representation construction from the archimedean case.

Let $A$ be a $p$-adic operator algebra. Let $G$ be a group acting on $A$, that is, we have a group homomorphism $\alpha\colon G \to \Aut(A)\colon g \to \alpha_g$ such that each $\alpha_g$ is a \(*\)-automorphism.

Define
\begin{align*}
    A \rtimes_{\alpha,r} G := c_0 (G, A) = \set{\phi\colon G \to A \mid \lim_{g \to \infty} \|\phi(g)\| = 0}
\end{align*}
with usual sum and multiplication given by the convolution product
\begin{align*}
    (\phi \ast \psi)(h) = \sum_{g \in G} \phi(g) \alpha_g(\psi(g^{-1}h)).
\end{align*}
It is routine to check that the convolution is well-defined and associative. The involution on $c_0(G, A)$ is defined by the formula
\begin{align*}
    (\phi^*)(g) := \alpha_g(\phi(g^{-1})^*)
\end{align*}
turning it into a Banach $*$-algebra over $\zP$. To represent it on a \(p\)-adic Hilbert space, we use the representation \(\pi \colon A \to \BUN(\qP(X))\) of \(A\). Using \(\pi\), we may construct the following representation: for \(\xi \in \qP(X \times G)\) and \(a \in A\), define
\begin{align*}
    \tilde{\pi}(a)(\xi)\colon X \times G &\to \qP \\
    (y, h) &\mapsto \pi(\alpha_h^{-1}(a))(\xi(\cdot, h))(y)
\end{align*}
where $\xi(\cdot, h)$ is the function that sends $x \mapsto \xi(x, h)$. To see that this assignment is well defined, that is, $\tilde{\pi}(a)(\xi) \in \qP(X \times G)$, we write $\xi = \xi_0 + \xi_1$ where
\begin{align*}
    \xi_0(y, h) = \begin{cases}
        \xi(y, h) \text{, if } |\xi(y, h)| \leq 1 \\
        0 \text{, otherwise}
    \end{cases}
\end{align*}
and 
\begin{align*}
    \xi_1(y, h) = \begin{cases}
        \xi(y, h) \text{, if } |\xi(y, h)| > 1 \\
        0 \text{, otherwise.}
    \end{cases}
\end{align*}
Then naturally $\tilde{\pi}(a)(\xi_0) \in \qP(X \times G)$, so we only need to show that $\tilde{\pi}(\xi_1) \in \qP(X \times G)$. Note that $\xi_1$ has finite support, so we may write
\begin{align*}
    \xi_1 = \sum_{(x, g) \in X \times G} \delta_x \otimes \delta_g \xi_1(x, g), 
\end{align*}
so that \(\xi_1(\cdot, h) = \sum_{x \in X} \delta_x \xi_1 (x, h)\). Consequently, we have \[\pi(\alpha_h^{-1}(a))_{y, x} \xi_1 (x, h)\]
As the support of $\xi_1$ is finite, there are finitely many $h \in G$ such that $\xi_1 (x, h) \neq 0$ for some $x \in X$. Let $h_0, h_1, \dots, h_n$ be all such $h$. Fix some $k \in \set{0, 1, \dots, n}$. Let $x_1, \dots, x_m \in X$ such that $\xi_1(x_l, h_k) \neq 0$. Write
\begin{align*}
\tilde{\pi}(a)(\xi)(y, h_k)=\pi(\alpha_{h_k}^{-1}(a))_{y, x_1} \xi_1 (x_1, h_k) + \cdots + \pi(\alpha_{h_k}^{-1}(a))_{y, x_m} \xi_1 (x_m, h_k)
\end{align*}
For each $l \in \set{0, 1, \dots, m}$, using \ref{limitentries} we get
\begin{align*}
    \lim_{z \to \infty} \pi(\alpha_{h_k}^{-1}(a))_{z, x_l} = 0,
\end{align*}
so that there exists a finite subset $F_l \subseteq X$ such that for $y \in X \setminus F_l$ we have
\begin{align*}
    |\pi(\alpha_{h_k}^{-1}(a))_{y, x_l}| < \|\xi_1\|^{-1}. 
\end{align*}
Therefore, for $y \in X \setminus \bigcup_{0 \leq l \leq m} F_l$
\begin{align*}
    |\tilde{\pi}(a)(\xi)(y, h_k)| \leq \max_{l} |\pi(\alpha_{h_k}^{-1}(a))_{y, x_l} \xi_1 (x_l, h_k)| \leq \|\xi_1\|^{-1} \cdot \|\xi_1\| = 1
\end{align*}
concluding that $|\tilde{\pi}(a)(\xi)(y, h)| > 1$ for finitely many $(y, h) \in X \times G$.

To see what the adjoint of $\tilde{\pi}(a)$ is, for $\xi, \eta \in \qP(X \times G)$ we have
\begin{align*}
    \langle \tilde{\pi}(a)(\xi), \eta \rangle &= \sum_{(y,h) \in X \times G} \tilde{\pi}(a)(\xi)(y, h) \eta(y, h) + \zP \\
    &= \sum_{(y,h) \in X \times G} \pi(\alpha_h^{-1}(a))(\xi(\cdot, h))(y) \eta(y, h) + \zP \\
    &= \sum_{h \in G} \sum_{y \in X} \pi(\alpha_h^{-1}(a))(\xi(\cdot, h))(y) \eta(y, h) + \zP \\
    &= \sum_{h \in G} \langle \pi(\alpha_h^{-1}(a))(\xi(\cdot, h)), \eta( \cdot , h) \rangle \\
    &= \sum_{h \in G} \langle \xi(\cdot, h), \pi(\alpha_h^{-1}(a^*))(\eta( \cdot , h)) \rangle \\
    &= \sum_{h \in G} \sum_{y \in X} \xi(y, h), \pi(\alpha_h^{-1}(a^*))(\eta( \cdot , h))(y) + \zP \\
    &= \sum_{(y,h) \in X \times G}  \xi(y, h), \pi(\alpha_h^{-1}(a^*))(\eta( \cdot , h))(y) + \zP \\
    &= \sum_{(y,h) \in X \times G} \xi(y, h) \tilde{\pi}(a^*)(\eta)(y, h) + \zP \\
    &= \langle \xi, \tilde{\pi}(a^*) (\eta) \rangle,
\end{align*}
proving that $\tilde{\pi}(a)$ is adjointable with $\tilde\pi(a)^*=\tilde{\pi}(a^*)$. Thus, $\tilde{\pi}$ preserves the involution.

It is easy to see that \(\tilde{\pi}\) is \(\zP\)-linear and a routine computation shows that it is also multiplicative. Finally, to see that \(\tilde{\pi}\) is an isometry, for \(a \in A\) we have  
\begin{align*}
    \|\tilde{\pi}(a)\| = \sup_{\|\xi\| \leq 1} \|\tilde{\pi}(a)(\xi)\| &= \sup_{\|\xi\| \leq 1} \max_{(y, h) \in X \times G} |\tilde{\pi}(a)(\xi)(y, h)| \\
    &= \sup_{\|\xi\| \leq 1} \max_{(y, h) \in X \times G} |\pi(\alpha_h^{-1}(a))(\xi(\cdot, h))(y)|. \\
\end{align*}
We want to show that the above supremum is equal to 
\begin{align*}
    \sup_{\|\eta\| \leq 1} \max_{(y, h) \in X \times G} |\pi(\alpha_h^{-1}(a))(\eta)(y)|
\end{align*}
To see this, note that the sets
\begin{align*}
    S_0 := \set{|\pi(\alpha_h^{-1}(a))(\xi(\cdot, h))(y)| : (y, h) \in X \times G \text{ and }\|\xi\| \leq 1} \\
    S_1 := \set{|\pi(\alpha_h^{-1}(a))(\eta)(y)| : (y, h) \in X \times G \text{ and }\|\eta\| \leq 1}
\end{align*}
are equal: the inclusion $S_0 \subseteq S_1$ is immediate. Now let $|\pi(\alpha_h^{-1}(a))(\eta)(y)| \in S_1$. For every $x \in X$ define $\xi(x, g) = \eta(x)$ if $g = h$ and $\xi(x, g) = 0$ otherwise. Then
\begin{align*}
    |\pi(\alpha_h^{-1}(a))(\eta)(y)| = |\pi(\alpha_h^{-1}(a))(\xi(\cdot, h))(y)| \in S_0
\end{align*}
Therefore the supremum of $S_0$ and $S_1$ are equal and using that $\pi \circ \alpha_h^{-1}$ is an isometry
\begin{align*}
    \|\tilde{\pi}(a)\| &= \sup_{\|\eta\| \leq 1} \max_{(y,h) \in X \times G} |\pi(\alpha_h^{-1}(a))(\eta)(y)| \\ 
    &= \max_{h \in G} \| \pi(\alpha_h^{-1}(a))\| \\
    &= \|a\|
\end{align*}
concluding the claim that $\tilde{\pi}$ is an isometric representation.

Now consider the morphism
\begin{align*}
    \tilde\lambda\colon G &\to \BUN(\qP(X \times G))
\end{align*}
defined by the formula
\begin{align*}
    \tilde\lambda_g(\xi)(y, h) := \xi(y, g^{-1}h).
\end{align*}
Define the map
\begin{align*}    
  \rho:=\tilde\pi\rtimes\tilde\lambda\colon c_0(G,A)\to \BUN(\qP(X\times G)). \\
\end{align*}
given by 
\begin{align}\label{equation-crossproduct1}
    \rho(\phi):=\sum_{g\in G}\tilde\pi(\phi(g))\tilde\lambda_g
\end{align}
The above sum is well defined because $\phi$ is a $c_0$-function. For $\xi \in \qP(X \times G)$ and $(y, h) \in X \times G$ we have
\begin{align*}
    \rho(\phi)(\xi)(y, h) &= \sum_{g\in G}\tilde\pi(\phi(g))(\tilde\lambda_g(\xi))(y, h) \\
    &= \sum_{g\in G} \pi(\alpha^{-1}_h(\phi(g)))(\xi(\cdot, g^{-1}h))(y).
\end{align*}
Let $g_0 \in G$ be an element such that \( \|\phi(g_0)\| = \|\phi\|.\) Consider the set $S$ of the functions $\xi \in \qP(X \times G)$ with $\xi(z, k) \neq 0$ if and only if $k = e$ and $\|\xi\| \leq 1$. The supremum indexed by a set is at least the supremum over some subset, so that \[
    \sup_{\xi \in S} \max_{y \in X} |\rho(\phi)(\xi)(y,g_0)| \leq \sup_{\|\xi\| \leq 1} \max_{(y,h) \in X \times G} |\rho(\phi)(\xi)(y,h)| = \|\rho(\phi)\|.\]
   Consequently, we have \[ \sup_{\xi \in S} \max_{y \in X} |\pi(\alpha_{g_0}^{-1}(\phi(g_0)))(\xi(\cdot, e))(y) \leq \|\rho(\phi)\|,\] which yields that \(   
     \sup_{\xi \in S} \|\pi(\alpha_{g_0}^{-1}(\phi(g_0)))(\xi(\cdot, e))\| \leq \|\rho(\phi)\|,\) or 
    \[ \|\pi(\alpha_{g_0}^{-1}(\phi(g_0)))\| \leq \|\rho(\phi)\| \\
    \implies \|\phi\| = \|\phi(g_0)\| \leq \|\rho(\phi)\|.
\]
The inequality $\|\rho(\phi)\| \leq \|\phi\|$ is a simple consequence of the formula \ref{equation-crossproduct1} and the ultrametric inequality. This shows that the map \(\rho\) is isometric. It is clear from the formula that \(\rho\) is $\zP$-linear and preserves involution, and it is again a routine verification that it is multiplicative. Hence the reduced crossed product $A \rtimes_{\alpha,r} G$ is a $p$-adic operator algebra.

\subsection{The maximal crossed product} 
A natural way to construct $p$-adic operator algebras is to consider the purely algebraic version of the $*$-algebra that we want and then apply the enveloping functor. We shall carry out this process for the algebraic crossed product, and call the resulting \(p\)-adic operator algebra the \textit{maximal crossed product}.

Let \(\alpha \colon G \to \mathsf{Aut}(A)\) be an action of a (discrete) group $G$ on a $*$-algebra $A$ over $\Zp$ by $*$-automorphisms $\alpha_g\colon A\to A$. Consider the $\zP$-module
\begin{align*}
    A \rtimes_{\alpha,\alg} G = c_c(G, A) = \set{ \phi \colon G \to A \mid \phi \text{ has finite support}}
\end{align*}
with usual sum, the multiplication by convolution and involution given by
\begin{align*}
   (\phi \ast \psi)(h) = \sum_{g \in G} \phi(g) \alpha_g(\psi(g^{-1}h)),\quad     \phi^*(g) = \alpha_g(\phi(g^{-1})^*).
\end{align*}
Straightforward computations show that $A\rtimes_{\alpha,\alg}G$ is a $*$-algebra over $\Zp$ with these operations. Given $a\in A$ and $g\in G$, we write $a\delta_g$ for the element of $A\rtimes_{\alpha,\alg}G$ that, as a function $G\to A$ takes the value $a$ at $g\in G$ and is zero otherwise. Notice that these elements span the entire algebra. And on these generators, the algebraic operations defined above read as
\begin{equation}\label{eq:algebraic-relations-CP}
    (a\delta_g)\cdot (b\delta_h)=a\alpha_g(b)\delta_{gh},\quad (a\delta_g)^*=\alpha_{g^{-1}}(a^*)\delta_{g^{-1}}.
\end{equation}
Indeed, we may view $A\rtimes_{\alpha,\alg}G$ as the universal $*$-algebra over $\Zp$ generated by $a\delta_g$ with $a\in A$, $g\in G$, satisfying the above relations.
Notice that $\iota_A\colon A\to A\rtimes_{\alpha,\alg}G$, $a\mapsto a\delta_1$ is a $*$-homomorphism. 

If $A$ is unital, then so is $A\rtimes_{\alpha,\alg}G$ with $1_A\delta_1$ playing the role of the unit. In this case $A\rtimes_{\alpha,\alg}G$ is the universal unital $*$-algebra over $\Zp$ generated by a copy of $A$ via a unital $*$-homomorphism $\iota_A\colon A\to A\rtimes_{\alpha,\alg}G$, $\iota_A(a)=a\delta_1$ and a unitary representation $\iota_G\colon G\to \UU(A\rtimes_{\alpha,\alg}G)$ that forms a covariant pair in the sense that
\begin{equation}\label{eq:cov-relations-CP}
\iota_A(\alpha_g(a))=\iota_G(g)\iota_A(a)\iota_G(g)^*\quad\mbox{for all } a\in A,\, g\in G.
\end{equation}

\begin{definition}
Given an action $\alpha\colon G\to \Aut(A)$ of a discrete group $G$ on a $p$-adic operator algebra $A$, we define its \textit{maximal crossed product} by setting
\begin{align*}
    A \rtimes_{\alpha, \max} G = (A \rtimes_{\alpha,\alg} G)^u
\end{align*}
\end{definition}

Hence, $A\rtimes_{\alpha,\max}G$ is the universal $p$-adic operator algebra generated by $a\delta_g$ with $a\in A$, $g\in G$ satisfying the relations~\eqref{eq:algebraic-relations-CP}. And if $A$ is unital, it is the universal unital $p$-adic operator algebra generated by a covariant pair $(\iota_A,\iota_G)$ with $\iota_A\colon A\to A\rtimes_{\alpha,\max} G$ a unital $*$-homomorphism and $\iota_G\colon G\to \UU(A\rtimes_{\alpha,\max}G)$ a unitary representation satisfying~\eqref{eq:cov-relations-CP}.  

\begin{remark}
    The canonical $*$-homomorphism $\iota_A\colon A\to A\rtimes_{\alpha,\max}G$ need not be contractive. Indeed, if $G$ is the trivial group, this homomorphism reduces to the canonical homomorphism $\iota\colon A\to A^u$, which is contractive if and only if $\|\cdot\|_A=\|\cdot\|_u$, that is, $A=A^u$.
\end{remark}

\begin{proposition}\label{lem:crossed-product-good}    
If $A$ is a unital $p$-adic operator algebra, and $\alpha\colon G\to \Aut(A)$ is an action, then 
$$A\rtimes_{\alpha,\max}G=A^u\rtimes_{\alpha^u,\max}G=A^u\rtimes_{\alpha^u,r}G $$ where $\alpha^u$ is the action on $A^u$ induced from $\alpha$ and the functoriality of $A\mapsto A^u$.
\end{proposition}
\begin{proof}
    Take $\rho\colon A\rtimes_{\alpha,\alg}G\to B$ be a $*$-homomorphism to a $p$-adic operator algebra $B$. Then $\rho=\pi\rtimes u$ for a covariant pair $(\pi,v)\colon (A,G)\to B$ with $\pi\colon A\to B$ and $v\colon G\to B$ $*$-homomorphisms. In particular, notice that each $v_g$ is a partial isometry, so that $\|v_g\|\leq 1$. We have 
    \begin{equation}\label{eq:cov-equation}
        \rho(\phi)=\sum_{g\in G}\pi(\phi(g))v_g\quad\mbox{for all }\phi\in A\rtimes_{\alpha,\alg}G.
    \end{equation}
    In particular, $\pi\colon A\to B$ is a $*$-homomorphism, so it extends to a contractive $*$-homomorphism $\pi^u\colon A^u\to B$. By density and continuity, the pair $(\pi^u,v)\colon (A^u,G)\to B$ is still covariant. Moreover, setting $\rho^u:=\pi^u\rtimes v\colon A^u\rtimes_{\alpha^u,\alg}G\to B$, then~\eqref{eq:cov-equation} still holds with $\rho^u$ and $\pi^u$ in place of $\rho$ and $\pi$, respectively. Hence
    $$\|\rho^u(\phi)\|\leq \max_{g\in G}\|\pi^u(\phi(g))\|\leq \max_{g\in G}\|\phi(g)\|_u=\|\phi\|_{A^u\rtimes_{\alpha,r}G},$$
    where $\|\cdot\|_u$ denotes the norm from $A^u$. The above estimate implies that 
    $\rho^u$ extends (uniquely) to a $*$-homomorphism $A^u\rtimes_{\alpha,r}G\to B$.
    This argument shows that $A^u\rtimes_{\alpha,r}G$ has the same universal properties as those of $A\rtimes_{\alpha,\max}G$ and $A^u\rtimes_{\alpha^u,\max}G$, and therefore the result follows.    
\end{proof}

\begin{corollary}\label{cor:unital-max=red}
    If $A$ is a unital $p$-adic operator algebra with $A^u=A$, then for every action $\alpha\colon G\to\Aut(A)$, we have
    $$A\rtimes_{\alpha,\max}G\cong A\rtimes_{\alpha,r}G.$$
\end{corollary}

Recall that the property \(A = A^u\) happens when \(A\) is a \(p\)-adic operator algebra with the canonical \(p\)-adic norm. Indeed, in this case, we can prove the same conclusion of the above corollary even without assuming that $A$ is unital:

\begin{proposition} \label{pro:crossed-product-adic-norm} 
If $A$ is a $p$-adic operator algebra that carries the canonical $p$-adic norm, then for every action $\alpha\colon G\to \Aut(A)$, we have
\begin{align*}
    A \rtimes_{\alpha, \mathrm{max}} G = A \rtimes_{\alpha,r} G
\end{align*}
\end{proposition}
\begin{proof}
We show that $A\rtimes_\alpha G$ has the universal property of the enveloping algebra of $A\rtimes_{\alpha,\alg}G$. For this we consider $A\rtimes_{\alpha,\alg}G$ with the supremum norm, so that it becomes a dense $*$-subalgebra of $A\rtimes_{\alpha,r} G=c_0(G,A)$.
Let $f\colon A \rtimes_{\alpha,\alg} G \to B$ be a continuous $*$-homomorphism into a $p$-adic operator algebra $B$. For an arbitrary $\phi \in A \rtimes_{\alpha,\alg} G$, note that 
\begin{align*}
    \|\phi\| = \max_{g \in G} \|\phi(g)\| = \max_{g \in G} \|\phi(g)\|_p = \|\phi\|_p 
\end{align*}
Assuming $\phi\not=0$, we can write it as $ \phi = p^n \phi'$ with $\|\phi'\| = 1$, so that
\begin{align*}
    \|f(\phi)\| = \|f(p^{n} \phi')\| = |p^n|\|f(\phi')\| \leq |p^n| = \|\phi\|
\end{align*}
Thus $f$ is continuous (contractive). As $A \rtimes_{\alpha,\alg} G$ is a dense subalgebra of $A \rtimes_{\alpha} G$, it follows that $f$ has an unique continuous extension $A \rtimes_{\alpha,r} G \to B$ making the following diagram commute
\[\begin{tikzcd}
	{A \rtimes_{\alpha,\alg} G} & B \\
	{A \rtimes_{\alpha,r} G}
	\arrow["f", from=1-1, to=1-2]
	\arrow[hook, from=1-1, to=2-1]
	\arrow[dashed, from=2-1, to=1-2]
\end{tikzcd}\]
This shows that $A \rtimes_{\alpha,r} G$ satisfies the desired universal property of the enveloping of $A \rtimes_{\alpha,\alg} G$ and concludes the proof.
\end{proof}

\begin{example}[Trivial actions]
    Suppose $A$ is a unital $p$-adic operator algebra and $\id\colon G\to \Aut(A)$ denotes the trivial action, that is, $\id_g=\id_A$ for all $g\in G$. Then $A\rtimes_{\id,\alg}G=A\oalg \Zp[G]$. By Proposition~\ref{pro:alg-tensor-enveloping} and Corollary~\ref{cor:unital-max=red},
    \begin{align*}
        A\rtimes_{\id,\max}G&=A^u\rtimes_{\id^u,\max}G=A^u\omax \OO_p(G)\\
        &=A\rtimes_{\id^u,r}G=A^u\hot \OO_p(G).
    \end{align*}
    If $A$ is not necessarily unital, but carries the $p$-adic norm, so that $A^u=A$, then Propositions\ref{pro:alg-tensor-enveloping} and~\ref{pro:crossed-product-adic-norm} yield
        \begin{align*}
        A\rtimes_{\id,\max}G=A\rtimes_{\id^u,r}G=A\omax \OO_p(G)=A\hot \OO_p(G).
        \end{align*}
\end{example}

\subsection{Action groupoid algebras as crossed products}\label{crossedii} 
In this section we consider some special types of crossed products that can be written as groupoid algebras.

Let $X$ be a locally compact Hausdorff topological space and suppose a group $H$ acts on $X$ by homeomorphisms via a (left) action written as $H\times X\to X$, $(h,x)\mapsto h\cdot x$. From this data we get a locally compact Hausdorff étale groupoid $\GG=X\rtimes H$, called the \emph{action groupoid}, which as a topological space is just the cartesian product $\GG=H\times X$ endowed with the product topology. Its unit space is $\GG^{(0)} = X\cong \{e\}\times X\sbe \GG$. The source and range maps are $\s(h,x):=x$ and $\rg(h,x):=h\cdot x$, and composition is given by $(g,y) \circ (h,x) = (g h, x)$
for $y = h\cdot x$. Finally, the inverse operation is $(h,x)^{-1}=(h^{-1},h\cdot x)$.

We can then form $\contc(\GG,\Zp)$, this is going to be the sum of $\zP$-modules of the form $C_c(\set{h} \times U,,\Zp)$ with $U \subseteq X$ open set, we view $\contc(\GG,,\Zp)$ inside of the space of functions $\GG \to \zP$. If $\phi, \psi \in \contc(\GG,,\Zp)$ and $(h, x) \in \GG$, then
\begin{align*}
    (\phi \ast \psi)(h, x) = \sum_{g \in H} \phi(g, (g^{-1}h)\cdot x) \psi(g^{-1} h, x)
\end{align*}
and $\OO_p(\GG)$ is the completion of this algebra.

The relationship with the algebra of the action groupoid and the reduced crossed product is analogous to what happens in the archimedean case.

\begin{proposition}\label{prop:action-groupoid-iso} Let $X$ be a locally compact Hausdorff topological space. Let $H$ be a group with an action $a\colon H \to \Aut(X)$. The action $a$ induces an action $\alpha$ of $H$ on $C_c(X, \zP)$ by $\alpha_h(\phi)=\phi \circ a_h^{-1}$ and there exists a canonical isomorphism
\begin{align*}
    C_0(X, \zP) \rtimes_{\alpha,\max} H = C_0(X, \zP) \rtimes_{\alpha,r} H \cong \OO_p(X \rtimes H).
\end{align*}
\end{proposition}
\begin{proof}
We consider the algebraic crossed product $C_c(X, \zP) \rtimes_{\alpha,\alg} H = c_c(H, C_c(X, \zP))$. Every element of $c_c(H, C_c(X, \zP))$ can be written as a finite sum of the form
\begin{align*}
    \phi = \sum_{g \in H} \phi_g \delta_g
\end{align*}
with $\phi_g \in C_c(X, \zP)$. On the generators $\phi_g \delta_g$ we define
\begin{align*}
    \rho(\phi_g \delta_g) (h, x) := \begin{cases}
        \phi_g(a_g(x)) \text{, if }g = h \\
        0 \text{, otherwise.}
    \end{cases}
\end{align*}
This gives an element of $C_c(X\rtimes_a H,\Zp)$ and an argument as in \cite{Beuter-Goncalves}*{Theorem~3.2} shows that the above linearly extends to a $*$-isomorphism $$\rho\colon C_c(X, \zP) \rtimes_{\alpha,\alg} H\congto C_c(X\rtimes_a H,\Zp).$$
Since the $p$-adic operator algebra $\contz(X,\Zp)$ carries the $p$-adic norm, Proposition~\ref{pro:crossed-product-adic-norm} implies that both maximal and reduce crossed products coincide. Moreover, $\rho$ is isometric for the supremum norms, which are the norms in both algebras $C_c(X, \zP) \rtimes_{\alpha,\alg} H\sbe \contz(X,\Zp)\rtimes_{\alpha,r}H$ and $C_c(X\rtimes_a H,\Zp)\sbe \contz(X\rtimes_a H,\Zp)=\OO_p(X\rtimes_a H)$. Indeed, for an element of the form $\phi_g \delta_g$ we have $\|\tilde{\rho}(\phi_g \delta_g) \| = \|\phi_g\|$ which implies $\|\rho(\phi) \|_\infty$ = $\| \phi \|_\infty$. The desired result follows.
\end{proof}

\begin{example}[Bunce-Deddens Algebra] A \textit{supernatural number} is a formal product
\begin{align*}
    S = \prod_{p \text{ prime}} p^{v_p(S)}, \quad v_p(S) \in \nN \cup \set{\infty}
\end{align*}
Given two supernatural numbers $S$ and $T$, we say that $S$ divides $T$ if $v_p(S) \leq v_p(T)$ for every prime $p$. A supernatural number $S$ is finite if $\sum_p v_p(S) < \infty$, we identify a finite supernatural number with its correspondent natural number. Let $S$ be a supernatural number, define
\begin{align*}
    \mathrm{Div}(S) = \set{k \in \nN \mid k \text{ divides }S}
\end{align*}
Given $k, l \in \mathrm{Div}(S)$, say that $k \leq l$ if $k$ divides $l$, turning $\mathrm{Div}(S)$ into a directed set. Note that if $k \leq l$, then there is a canonical ring morphism 
\begin{align*}
    \zZ/l\zZ &\to \zZ/k\zZ \\
    x\mod l &\mapsto x\mod k
\end{align*} We can take the inverse limit inside the category of topological groups with each $\zZ/k\zZ$ carrying the discrete topology
\begin{align*}
    \frac{\zZ}{S\zZ} := \lim_{\substack{\longleftarrow \\ k \in \mathrm{Div}(S)}} \frac{\zZ}{k\zZ} = \set{(a_k) \in \prod_{k \in \mathrm{Div}(S)} \frac{\zZ}{k\zZ} \mid k \leq l \implies a_l \equiv a_k \mod k}
\end{align*}
For example, if $S = p^{\infty}$, then $\zZ/p^{\infty} \zZ = \zP$ the $p$-adic integers. If $S = \prod_{p} p^{\infty}$, then $\zZ/S\zZ = \widehat{\zZ}$, known as the profinite completion of $\zZ$. 

The topological group $\zZ/S\zZ$ is compact, Hausdorff and totally disconnected. There is a natural injection
\begin{align*}
    i\colon \zZ &\to \frac{\zZ}{S\zZ} \\
    a &\mapsto (a \mod k)_{k \in \mathrm{Div}(S)}
\end{align*}
with dense image. Define the successor function
\begin{align*}
    s \colon \frac{\zZ}{S\zZ} &\to \frac{\zZ}{S\zZ} \\
    x &\mapsto x + i(1)
\end{align*}
the map $s$ is a homeomorphism and induces an action 
\begin{align*}
    \alpha\colon \zZ &\to \Homeo(\zZ/S\zZ) \\
    n &\mapsto s^n
\end{align*}
Notice that this action is minimal, i.e., orbits are dense.

As above, we can form the action groupoid $\zZ/S\zZ \rtimes\Z$, which is an étale (ample) Hausdorff group, and hence its $p$-adic operator algebra $\OO_p(\zZ/S\zZ \rtimes \zZ)$. By Proposition~\ref{prop:action-groupoid-iso},
\begin{align*}
    \OO_p(\zZ/S\zZ \rtimes \zZ) = C(\zZ/S\zZ, \zP) \rtimes \zZ
\end{align*}
These algebras are $p$-adic versions of the Bunce-Deddens \cstar{}algebras, see \cite{KMMP-BD}. 
\end{example}

\begin{example}[Rotation algebras as crossed products]
Fix $z\in \Zp^\times$. Consider the automorphism $\alpha_z\in \Aut(\OO_p(\Z))$ given on the generator $U:=\delta_1\in \OO_p(\Z)$ by $\alpha_z(U):=zU$. To see that this indeed gives an automorphism, we can use that $\OO_p(\Z)$ is the universal unital $p$-adic operator algebra generated by a unitary $U$.
Alternatively, realizing $\OO_p(\Z)\cong c_0(\Z,\Zp)$, the automorphism $\alpha_z$ is given by $\alpha_z(f)(n)=z^n f(n)$.

The automorphism $\alpha_z$ induces an action of $\Z$ on $\OO_p(\Z)$ by $n\mapsto \alpha_z^n$. 
Consider the crossed product $A^z_p:=\OO_p(\Z)\rtimes_\alpha\Z$. We can take here either the maximal 
or the reduced crossed product as they are isomorphic by Proposition~\ref{pro:crossed-product-adic-norm}. Viewing $A^z_p$ as the maximal 
crossed product and using its universal property via covariant representations, we get that $A^z_p$ 
is the universal $p$-adic operator algebra generated by two unitaries $U,V$, with $U$ corresponding 
to the generator $1\in \Z\to U(\OO_p(\Z))$ and $V$ corresponding the generator $V\in \OO_p(\Z)\sbe A^z_p$. The covariance condition means then the relation
$$UVU^*=\alpha_z(U)=zU, \quad\mbox{that is,}\quad UV=zVU.$$
Therefore $A^z_p$ coincides (up to isomorphism) with the rotation algebra introduced in Example~\ref{ex:Rotation-Algebras}.
\end{example}

\begin{remark}
    It seems plausible that Proposition \ref{prop:action-groupoid-iso} generalises to actions of inverse semigroups on locally compact Hausdorff totally disconnected spaces, but we leave this case for a future article. 
\end{remark}

\section{Analytic \(K\)-theory and some simple computations}

In this section, we take a look at an appropriate version of topological \(K\)-theory for \(p\)-adic operator algebras. As in the case of the complex numbers, we define such a functor in the generality of bornological \(\zP\)-algebras. More precisely, in \cite{mukherjee2022nonarchimedean}, we define a functor \[KH_n^\an \colon \mathsf{Alg}_{\zP} \to \mathsf{Ab}\] on the category of complete, torsionfree bornological \(\zP\)-algebras, for each \(n \in \Z\). The properties of this functor are summarised by the following:

\begin{theorem}\label{thm:analytic-K-theory}\cite{mukherjee2022nonarchimedean}*{Theorem 6.3, 6.8}
    For each \(n \in \Z\), the functor \(KH_n^\an \colon \mathsf{Alg}_{\zP} \to \mathsf{Ab}\) satisfies the following properties:
    \begin{enumerate}
        \item Finite additivity;
        \item Homotopy invariance: the natural map \(A \to A \haptimes \widehat{\zP[t]}\) induces an isomorphism \(KH_n^\an(A) \overset{\cong}\to KH_n^\an(A \otimes_\pi \widehat{\zP[t]})\);
        \item Stability with respect to the Banach algebra \(\mathcal{M}^\mathrm{cont} = \widehat{\mathbb{M}_\infty(\zP)}\) of compact operators, that is, the upper left hand corner inclusion \(A \to A \haptimes \mathcal{M}^\mathrm{cont}\) induces an isomorphism \(KH_n^\an(A) \overset{\cong}\to KH_n^\an(A \otimes_\pi \mathcal{M}^\mathrm{cont})\);
        \item Excision: for any extension \(0 \to A \to B \to C \to 0\) of complete, torsionfree bornological \(\zP\)-algebras with a bounded \(\zP\)-linear section, we get an induced long exact sequence \[\cdots \to KH_n^\an(A) \to KH_n^\an(B) \to KH_n^\an(C) \to KH_{n-1}^\an(A) \to \cdots\] in analytic \(K\)-theory groups;
        \item For a bornological Banach \(\zP\)-algebra \(A\), we have \(KH_n^\an(A) \cong KH_n(A/p)\), where the right hand side denotes Weibel's homotopy algebraic \(K\)-theory. 
    \end{enumerate}
\end{theorem}

The fifth property of Theorem~\ref{thm:analytic-K-theory} is particularly interesting as it says that the \(K\)-theory of Banach \(\zP\)-algebras only depends on the reduction mod \(p\) of the algebra - the latter is a purely algebraic object. We also remark that the same result holds for local cyclic homology defined in \cite{meyer2023local}, where the following quasi-isomorphism is shown:

\[\mathbb{HL}(A) \simeq \mathbb{HA}(A/p)\] between the local cyclic homology complex of a Banach algebra \(A\) and the analytic cyclic complex of the reduction mod \(p\). This behaviour is manifestly specific to the \(p\)-adic setting. We end this article with two prototypical computations of the homotopy analytic \(K\)-theories of \(p\)-adic operator algebras. By virtue of the last part of Theorem~\ref{thm:analytic-K-theory}, these computations essentially reduce to Quillen's computation of algebraic \(K\)-theory of \(\fF_p\): 
    
\begin{equation}\label{eq:Quillen}
K_m(\fF_p) \cong \begin{cases}
\Z/ (p^i -1) \quad m =
2i-1 , \quad i \geq 1,\\
0 \quad m = 2i, \quad i \geq 1, \\
\Z \quad m = 0 \\
0 \quad m<0
\end{cases}
\end{equation}
\subsection{Leavitt path algebras}

Recall from Example~\ref{ex:Leavitt-groupoid}, that the Leavitt path algebra of a graph \(E = (E^0, E^1, s,r)\) is specific case of a Steinberg algebra of a groupoid, whose enveloping algebra \(\OO_p(E)\) yields a \(p\)-adic operator algebra. This in turn coincides with the \(p\)-adic completion of the underlying Leavitt path algebra \(L_{\Zp}(E)\). Now for any graph \(E\), consider the matrix \(N_E\) whose entries are defined as \[
N_E \colon E^0 \times \mathrm{reg}(E)  \to \zZ, \quad (v,w) \mapsto \delta_{v,w} - \abs{s^{-1}(\{w\}) \cap r^{-1}(\{v\})},
\] where \(\mathrm{reg}(E)\) is the set of regular vertices of \(E\). We then have the following: 

\begin{theorem}\label{thm:analytic-K-Leavitt}
    Let \(E\) be a directed graph. Then the analytic \(K\)-theory is given as follows:
    \[KH_n^\an(\OO_p(E)) \simeq KH_n(L_{\fF_p}(E)) \cong \pi_n(\mathsf{cofib}(\mathsf{K}(\fF_p)^{(\mathrm{reg}(E))} \overset{N_E^t}\to \mathsf{K}(\fF_p)^{(E_0)})) \] for each \(n \in \Z\).
\end{theorem}

\begin{proof}
    By Theorem~\ref{thm:analytic-K-theory}, the computation of the analytic \(K\)-theory of \(\OO_p(E)\) is equivalent that of the homotopy algebraic \(K\)-theory groups \(KH_n(L_{\fF_p}(E))\) for each \(n \in \Z\). Now by \cite{cortinas2021algebraic}*{Theorem 5.4}, there is a distinguished triangle \[\mathsf{KH}(\fF_p)^{(\mathrm{reg}(E))} \overset{N_E^t}\to \mathsf{KH}(\fF_p)^{(\mathrm{E}^0)} \to \mathsf{KH}(L_{\fF_p}(E))\] in the homotopy category of spectra. Since \(\fF_p\) is a field, the homotopy algebraic \(K\)-theory spectrum is equivalent to the algebraic \(K\)-theory spectrum, as desired.
\end{proof}

\begin{example}
    As an example, consider the graph \(E_n\) with a single vertex and \(n\) loops for $n\geq 1$. The associated \(p\)-adic operator algebra is \(\OO_{p,n}:=\OO_p(E_n) \), the so-called \(p\)-adic \emph{Cuntz algebra}. Notice that for $n=1$ we get $\OO_{p,1}\cong\OO_p(\Z)$, the $p$-adic operator algebra of the integers; this can be also viewed as $\Zp[t,t^{-1}]^u=\widehat{\Zp[t,t^{-1}]}$, the universal (or $p$-adic) completion of the $\Zp$-algebra of Laurent polynomials $\Z[t,t^{-1}]$. 
    
    The matrix \(N_E\) of $E = E_n$ is (the $1\times 1$-matrix) \(1 - n \in \Z\).  By Theorem~\ref{thm:analytic-K-Leavitt}, to compute \(KH_m^\an(\OO_p(E))\), one computes the homotopy cofibre via the long exact sequence 
    \[\dotsc \to K_m(\fF_p) \overset{1-n}\to K_m(\fF_p) \to KH_m(L_{\fF_p}(E_n)) \to K_{m-1}(\fF_p) \overset{1-n}\to K_{m-1}(\fF_p) \to \dotsc \] for each \(m \in \Z\). Using part (5) of Theorem \ref{thm:analytic-K-theory} and \ref{eq:Quillen}, if \(m<0\), we have \(KH_m^\an(\OO_p(E_n)) = 0\). Furthermore, if \(n-1\) is invertible in \(\fF_p\), then by \cite{ara2009k}*{Example 8.6} we have \[KH_m(L_{\fF_p}(E_n)) \cong K_m(\fF_p, \zZ/ (n-1)).\] By the Universal Coefficient Theorem, we may compute the \(\mod\) \((n-1)\) algebraic \(K\)-theory groups by the short exact sequence \[0 \to K_m(\fF_p) \otimes \zZ/(n-1) \to K_m(\fF_p, \zZ/(n-1)) \to \mathrm{Tor}_1(K_{m-1}(\fF_p), \zZ/(n-1)) \to 0\] for each \(m\), which splits if \(n-3\) is not divisible by \(4\) (\cite{weibel1989homotopy}*{Proposition 1.6, Remark 1.6.1}). In particular, for \(m=0\), we have \(KH_0^\an(\OO_p(E_n)) \cong \zZ/(n-1)\), and for all other even \(m\), \[KH_m^\an(\OO_p(E_n)) \cong \zZ/\mathrm{gcd}(p^i -1, n-1).\] If \(m\) is odd, then \(KH_m^\an(\OO_p(E_n)) \cong \zZ/\mathrm{gcd}(p^i -1, n-1)\), where \(m = 2i -1\) for \(i \geq 1\). Therefore, if \(n-1\) divides \(p-1\), we have \(KH_m^\an(\OO_p(E_n)) \cong \zZ/(n-1)\) for all \(m \geq 0\). 
\end{example}

\subsection{\(p\)-adic noncommutative tori}

We now compute the homotopy analytic \(K\)-theory of the \(p\)-adic noncommutative torus. The computation uses tools from bivariant algebraic \(K\)-theory, just as in the complex case. Furthermore, as we have seen in the previous example, to preserve torsion information (that naturally arises from \(K_1(\fF_p)\)), one must work integrally. In the context of the noncommutative torus, we shall see that unlike in the \(C^*\)-algebraic case, the \(K\)-theory groups depend on the rotation algebra parameter \(z \in \zP^\times\). The reason for this is the fact that in the \(C^*\)-algebraic case, different rotation algebras are all homotopic, using the connectedness of the circle, which allows one to construct a continuous homotopy. This of course fails in the \(p\)-adic setting. 

To set up the computation, we first recall from \cite{cortinas2007bivariant} the Pimsner-Voiculescu sequence in bivariant algebraic \(K\)-theory (specialised to our setting):

\begin{theorem}\cite{cortinas2007bivariant}*{Theorem 7.4.1}\label{thm:exact-triangle}
    Let \(A\) be an \(\fF_p\)-algebra and \(\alpha \colon A \to A\) an algebra automorphism. Then there is an exact triangle \[ A \overset{1 - \alpha}\to A \overset{\iota}\to A \rtimes_{\alpha} \Z \to \Sigma A\] in bivariant algebraic \(K\)-theory.
\end{theorem}

We shall apply this to the reduction mod \(p\) of the \(p\)-adic noncommutative torus \(A_p^z\), which is given by the crossed product algebra of the action \(\alpha \colon \fF_p[t,t^{-1}] \to \fF_p[t,t^{-1}]\) defined by \(\alpha(t) \defeq \lambda t\) for \(\lambda \in \fF_p^\times\). Denote this crossed product algebra by \(A_\lambda\). By the Bass Fundamental Theorem \cite{cortinas2007bivariant}*{Theorem 7.3.1} in bivariant \(K\)-theory, the algebra \(\fF_p[t,t^{-1}]\) is \(kk\)-equivalent to \(\sigma \fF_p \oplus \fF_p\), where \(\sigma \fF_p = \ker(\fF_p[t,t^{-1}] \overset{\ev_1}\to \fF_p) = (t-1)\fF_p[t,t^{-1}].\) Note also that we have a split-extension \[\sigma \fF_p \to \fF_p[t,t^{-1}] \to \fF_p,\] where the splitting is given by the canonical map \(1 \mapsto t\). Under this \(kk\)-equivalence, the map \(\alpha \colon \fF_p[t,t^{-1} \to \fF_p[t,t^{-1}]\), \(t \mapsto \lambda t\) restricts to the identity on  \(\fF_p\), and is given by the map \[\sigma \fF_p \to \sigma \fF_p \oplus \fF_p, \quad t \mapsto (\lambda t, 1) \in \sigma \fF_p \oplus \fF_p.\] Consequently, the exact triangle of Theorem~\ref{thm:exact-triangle} is equivalent to the triangle 
\begin{equation}\label{eq:triangle}
\sigma \fF_p \oplus \fF_p   \overset{\begin{pmatrix}
    0 & 0 \\
    \lambda^{-1} & 0
\end{pmatrix}}\to \sigma \fF_p \oplus \fF_p \to A_\lambda \to \sigma \fF_p.
\end{equation}

The map \(\begin{pmatrix}
    0 & 0 \\
    \lambda^{-1} & 0
\end{pmatrix}\) can be replaced with the diagonal matrix \(\begin{pmatrix}
    0 & 0 \\
    0 & \lambda^{-1}
\end{pmatrix}\) upon composition with \(\begin{pmatrix}
    0 & 1 \\
    0 & 0
\end{pmatrix}\), so that triangle in Equation \ref{eq:triangle} splits into a direct sum of the triangles 

\begin{equation}\label{eq:triangle2}
    \Sigma \fF_p \overset{\lambda^{-1}}\to \fF_p \to C_\lambda \to \Sigma^2 \fF_p \quad \text{and } \quad  \fF_p \overset{0}\to \fF_p \to \fF_p[t,t^{-1}] \to \Sigma\fF_p,   
\end{equation}

where \(C_\lambda \oplus \fF_p[t,t^{-1}]\) is \(kk\)-equivalent to \(A_\lambda\). Consequently, the computation of the homotopy algebraic \(K\)-theory (or any invariant satisfying homotopy invariance, excision and matricial stability) reduces to the computation of \(KH_n(C_\lambda)\) and \(KH(\fF_p[t,t^{-1}])\). Here the algebra \(C_\lambda\) is \((t-\lambda)(t-1) \fF_p[t,t^{-1}]\), but we will ultimately not need this:

\begin{lemma}
    We have \[
    KH_n(C_\lambda) = 
    \begin{cases}
        0 \quad n < 0 \\
        \Z \quad n = 0 \\
        \fF_p^\times / \langle \lambda^{-1} \rangle \quad n = 1 \\
        k \Z  \text{ when } \lambda \text{ is a primitive } k \text{ root of unity}, \text{ and } 0 \text{ otherwise }, n = 2 \\
       0 \text{ for all other even } n,  
 \end{cases}
    \] and for all odd \(n \geq 3\), we have short exact sequences 
    \[0 \to K_n(\fF_p) \to KH_n(C_\lambda) \to K_{n-2}(\fF_p) \to 0 .\]
\end{lemma}

\begin{proof}
    The exact triangle in Equation \ref{eq:triangle2} induces a long exact sequence in homotopy algebraic \(K\)-theory 

\[
\dotsc \to KH_{n-1}(\fF_p) \overset{\lambda^*}\to KH_n(\fF_p) \to KH_n(C_\lambda) \to KH_{n-2}(\fF_p) \overset{\lambda^*}\to KH_{n-1}(\fF_p) \to  \dotsc.  
\] Since \(\fF_p\) is in particular a regular ring, homotopy algebraic \(K\)-theory and algebraic \(K\)-theory coincide. The result now follows upon inspection of Equation \ref{eq:Quillen}.
\end{proof}

Note that the computation of homotopy algebraic \(K\)-theory for Laurent polynomials \(\fF[t,t^{-1}]\) is well known by the Bass Fundamental Theorem \cite{weibel1989homotopy}*{Theorem 1.2 (iii)}:

\[
KH_n(\fF_p[t,t^{-1}]) = \begin{cases}
    0 \quad n < 0, \\
    \Z \quad n = 0, \\
    \Z \oplus \Z/ (p-1), \quad n =1, \\
    \Z/(p-1) \quad n \geq 2.
\end{cases}
\]

Combining the computations so far, and using the last part of Theorem~\ref{thm:analytic-K-theory}, we have:

\begin{corollary}\label{cor:homotopy-K-nctorus}
    Let \(A_p^z\) be the \(p\)-adic noncommutative torus and \(A_\lambda\) its reduction mod \(p\), where \(\lambda \in \fF_p^\times \cong \Z/(p-1)\) is the image of \(z\) under the quotient map \(\zZ_p \to \fF_p\). We then have \[KH_n^{\an}(A_p^z) \otimes \Q \cong KH_n(A_\lambda) \otimes \Q = 
    \begin{cases}
    0 \quad n <0 \\
    \Q^2 \quad n =0 \\
    \Q \quad n = 1 \\
    \Q  \quad\lambda  \text{  primitive root of unity, else } 0,  \quad n = 2 \\
    0 \quad  \text{ all other } n.
    \end{cases}
    \]
 
\end{corollary}

 In particular, for \(z \in \zP^\times\) and \(\theta \in \R\), we have \(KH_0^\an(A_p^z) \otimes \Q \cong K_0^{\mathrm{top}}(A_\theta) \otimes \Q\).  But as expected conceptually, we obtain different information in the other homotopy groups. 

\subsection{Rigid analytic varieties}

Let \(A\) be a unital \(p\)-adic operator algebra whose reduction mod \(p\) is \emph{\(K_0\)-regular} in the sense that \(K_0(A/p) \simeq K_0(A/p[x_1,\dotsc, x_n])\) for all \(n\). We may more generally define \(K_n\)-regularity for any \(n \in \Z\), though by Vorst's Theorem, \(K_n\)-regularity implies \(K_{n-1}\)-regularity (\cite{vorst1979localization}). We then have the following identifications 
\[
KH_n^\an(A) \simeq KH_n(A/p) \simeq \begin{cases}
    KV_n(A/p) \quad n \geq 1 \\
    K_n(A/p) \quad n \leq 0,
\end{cases}
\] where \(KV_n\) denote the Karoubi-Villamayor \(K\)-theory groups. 

More generally, if \(A/p\) is \(K_n\)-regular for some \(n\), then the comparison map \[K_m(A/p) \to KH_m(A/p)\] is an isomorphism of abelian groups for all \(m \leq n\). We call a ring \emph{\(K\)-regular} if it is \(K_n\)-regular for all \(n\). 

\begin{lemma}\label{lem:supercoherent}
    Let \(A\) be a \(p\)-adic operator algebra such that \(A/p\) is a regular supercoherent ring in the sense of \cite{ara2009k}*{Section 7}. Then \(KH_n^\an(A) \simeq K_n(A/p)\) for all \(n \in \Z\).
\end{lemma}

\begin{proof}
    By \cite{waldhausen1978algebraic}*{Theorem 4}, any regular supercoherent ring is \(K\)-regular.  
\end{proof}

\begin{corollary}\label{cor:affinoid-K}
    Let \(A\) be a finite type \(\zP\)-affinoid algebra. Then \(KH_n^\an(A) \simeq K_n(A/p)\).
\end{corollary}

As a consequence of Corollary \ref{cor:affinoid-K}, we obtain a Chern character from algebraic \(K\)-theory to rigid cohomology, which is the analogue of de Rham cohomology in mixed characteristic:

\begin{theorem}
    Let \(R = \zP[x_1,\dotsc,x_n]/I\) be a smooth, finite-type \(\zP\)-algebra and \(A = \widehat{R}\).  We then have natural group homomorphisms \(K_n(A/p) \to \bigoplus_{m \in \Z} \mathrm{H}_{n+2m}^{\mathrm{rig}}(A/p, \qP)\) for all \(n \in \Z\), where the right hand side is Berthelot's rigid cohomology.
\end{theorem}

\begin{proof}
    The hypotheses on \(A\) imply that \(A/p\) is regular, so that by Corollary \ref{cor:affinoid-K}, we have \begin{multline*}
        K_n(A/p) \cong KH_n(A/p) \cong KH_n^\an(A) \\ \cong KH_n^\an(R^\dagger) \to HP_n(R^\dagger \otimes_{\zP} \qP) \cong \bigoplus_{m \in \Z} H_{\mathrm{rig}}^{n+2m}(A/p, \qP),
    \end{multline*} where the third isomorphism follows from \cite{mukherjee2022nonarchimedean}*{Theorem 6.8}, and the last isomorphism is \cite{cortinas2017nonarchimedean}*{Theorem 6.5}.
\end{proof}

\subsection{Some remarks on the \(K\)-theory of Steinberg algebras}
The expert may have noticed that the computations for Leavitt path algebras and noncommutative tori are essentially only possible because they lie in the \emph{bootstrap category} (\cite{cortinas2007bivariant}*{Section 8.3}) of bivariant algebraic \(K\)-theory (over \(\fF_p\)). However, as this category is not known to have arbitrary coproducts, the size of this bootstrap category is unclear. Consequently, the same techniques can no longer be used for more complicated algebras such as group(oid) algebras. To this end, we point to Xin Li's recent work (\cite{li2022ample}), which identifies the homotopy groups of algebraic \(K\)-theory spectrum \(K(\mathfrak{B}_G)\) of a certain permutative category \(\mathfrak{B}_G\) associated to an ample groupoid \(G\) with locally compact Hausdorff unit space, with the cohomology of the groupoid. Conjecturally (due to Xin Li), one should be able to relate the \(K\)-theory spectrum of the Steinberg algebra with the homology of the groupoid \(G\) with coefficients in the \(K\)-theory spectrum \(K(R)\) via an analogue of the Farrel-Jones assembly map. The \(C^*\)-algebraic version of this result appears in (\cite{miller2024isomorphisms}). Now by part (5) of Theorem~\ref{thm:analytic-K-theory}, we have \(KH_*^\an(\mathcal{S}(G, \zP)) \cong KH_*(\mathcal{S}(G, \fF_p))\), and the latter should be related to the homology of a suitable ample groupoid with coefficients in \(K(\fF_p)\) by the results of Li and Miller. This will be the subject of a forthcoming article.

\end{document}